\documentclass[12pt]{amsart}
\usepackage[utf8]{inputenc}
\usepackage{amsfonts,amsmath,amssymb,amsxtra,stmaryrd}
\usepackage{float}

\usepackage{a4wide}

\usepackage{colonequals}
\usepackage[utf8]{inputenc}
\usepackage[british]{babel}
\usepackage{hyperref}
\hypersetup{colorlinks=true,urlcolor=blue,citecolor=blue,linkcolor=blue}

\usepackage[all,cmtip]{xy}

\usepackage[dvipsnames]{xcolor}
\usepackage[normalem]{ulem}
\usepackage{cleveref}
\Crefname{equation}{}{}

\newcommand{\quat}[2]{\displaystyle{\biggl(\frac{#1}{#2}\biggr)}}

\usepackage{tikz-cd}
\usetikzlibrary{patterns,shapes,cd,arrows,arrows.meta}
\usepackage[auto]{contour}
\contourlength{2pt}

\newcommand{\calC}{\mathcal{C}}
\newcommand{\calG}{G}
\newcommand{\calA}{\mathcal{A}}
\newcommand{\calJ}{\calA}
\newcommand{\frakA}{\mathfrak{A}}

\newcommand{\calT}{\mathcal{T}}
\newcommand{\calX}{\mathcal{X}}
\newcommand{\calY}{\mathcal{Y}}
\newcommand{\frakp}{\mathfrak{p}}
\newcommand{\fraks}{\mathfrak{s}}
\newcommand{\frakP}{\mathfrak{P}}
\newcommand{\Cprim}{\mathcal{D}}
\newcommand{\Jprim}{B}
\newcommand{\calJprim}{\mathcal{B}}
\newcommand{\frakJprim}{\mathfrak{B}}
\newcommand{\Cbfour}{\Gamma^{(g-4)}}
\newcommand{\Cbtwo}{\Gamma^{(g-2)}}
\newcommand{\QuotC}{P}
\newcommand{\calK}{\Q(a)}
\newcommand{\calL}{\Q(b)}
\newcommand{\Q}{\mathbb{Q}}
\newcommand{\QQ}{\Q}
\newcommand{\A}{\mathbb{A}}
\newcommand{\C}{\mathbb{C}}
\newcommand{\F}{\mathbb{F}}

\newcommand{\PP}{\mathbb{P}}
\newcommand{\R}{\mathbb{R}}
\newcommand{\Z}{\mathbb{Z}}
\newcommand{\Qbar}{\Q^{\textup{al}}}
\newcommand{\Kbar}{K^{\textup{al}}}
\newcommand{\Abar}{A^{\textup{al}}}

\newcommand{\Cbar}{C^{\textup{al}}}
\newcommand{\calO}{\mathcal{O}}
\newcommand{\al}{\textup{al}}
\newcommand{\AAp}{U}
\newcommand{\BB}{V}

\newcommand{\aunder}{a}

\DeclareMathOperator{\End}{End}
\DeclareMathOperator{\Lie}{Lie}
\DeclareMathOperator{\GL}{GL}

\DeclareMathOperator{\SO}{SO}
\DeclareMathOperator{\GO}{GO}
\DeclareMathOperator{\GSO}{GSO}
\DeclareMathOperator{\Sp}{Sp}
\DeclareMathOperator{\Hom}{Hom}
\DeclareMathOperator{\Gal}{Gal}

\DeclareMathOperator{\GSp}{GSp}
\DeclareMathOperator{\opchar}{char}
\DeclareMathOperator{\Aut}{Aut}
\DeclareMathOperator{\Hg}{Hg}
\DeclareMathOperator{\MT}{MT}

\DeclareMathOperator{\M}{M}
\DeclareMathOperator{\Nm}{Nm}
\DeclareMathOperator{\Frob}{Frob}
\DeclareMathOperator{\Jac}{Jac}
\DeclareMathOperator{\PGL}{PGL}
\DeclareMathOperator{\Spec}{Spec}
\DeclareMathOperator{\Sym}{Sym}
\DeclareMathOperator{\trd}{trd}
\DeclareMathOperator{\nrd}{nrd}

\newcommand{\et}{\mathrm{et}}

\newcommand{\tors}{{\rm tors}}
\newcommand{\Kconnup}{K^{\textup{conn}}}
\newcommand{\Kconn}{K(\varepsilon_A)}
\newcommand{\KEnd}{K(\End A)} 

\newcommand{\Qconn}{\Q(\varepsilon_A)}
\newcommand{\QEnd}{\Q(\End A)} 
\newcommand{\Endzero}[1]{{\End(#1)_\Q}}

\usepackage{mathrsfs}

\numberwithin{equation}{subsection}

\newtheorem{theorem}[equation]{Theorem}

\newtheorem{conjecture}[equation]{Conjecture}
\newtheorem{corollary}[equation]{Corollary}
\newtheorem{cor}[equation]{Corollary}

\newtheorem{lemma}[equation]{Lemma}
\newtheorem{lem}[equation]{Lemma}

\newtheorem{prop}[equation]{Proposition}
\newtheorem{proposition}[equation]{Proposition}
\newtheorem{question}[equation]{Question}

\newcommand{\tbigwedge}{\textstyle{\bigwedge}}

\newcommand{\defi}[1]{\textsf{#1}} 	

\theoremstyle{definition}
\newtheorem{definition}[equation]{Definition}
\newtheorem{setup}[equation]{Setup}

\theoremstyle{remark}
\newtheorem{remark}[equation]{Remark}
\newtheorem{example}[equation]{Example}

\newenvironment{enumalph}
{\begin{enumerate}}
{\end{enumerate}}

\newenvironment{enumroman}
{\begin{enumerate}}
{\end{enumerate}}

\title[Monodromy groups of Jacobians with definite QM]{Monodromy groups of Jacobians with definite quaternionic multiplication}

\author{Victoria Cantoral-Farf\'an}
\address{Mathematisches Institut, Georg-August Universit\"at G\"ottingen,
Bunsenstra\ss e 3-5, 37073 Goettingen, Germany}
\email{victoria.cantoralfarfan@mathematik.uni-goettingen.de}

\author{Davide Lombardo}
\address{Dipartimento di Matematica, Universit\`a di Pisa, Largo Bruno Pontecorvo 5, 56127 Pisa, Italy}
\email{davide.lombardo@unipi.it}

\author{John Voight}
\address{Department of Mathematics, Dartmouth College, Kemeny Hall, Hanover, NH 03755, USA; Carslaw Building (F07), Department of Mathematics and Statistics, University of Sydney, NSW 2006, Australia}
\email{jvoight@gmail.com}

\keywords{Hyperelliptic curves, $\ell$-adic monodromy groups, Galois representations, Jacobians, definite quaternionic multiplication. \\ Accepted for publication in \textit{Proceedings of the London Mathematical Society}}


\newcommand{\ellindex}{s}

\begin{document}

\begin{abstract}
Let $A$ be an abelian variety over a number field. The connected monodromy field of $A$ is the minimal field over which the images of all of the $\ell$-adic torsion representations have connected Zariski closure. We show that for all even $g \geq 4$, there exist infinitely many geometrically nonisogenous abelian varieties $A$ over $\Q$ of dimension $g$ where the connected monodromy field is strictly larger than the field of definition of the endomorphisms of $A$. Our construction arises from explicit families of hyperelliptic Jacobians with definite quaternionic multiplication.
\end{abstract}

\maketitle

\tableofcontents

\section{Introduction}\label{sec: introduction}

\subsection{Motivation}

Let $K$ be a number field with algebraic closure $\Kbar$ and absolute Galois group $\Gal_K \colonequals \Gal(\Kbar\,|\,K)$.  Let $A$ be an abelian variety defined over $K$ with $g \colonequals \dim A$, and write $\Abar \colonequals A \times_K \Kbar$.
Let $\KEnd$ be the minimal extension of $K$ over which the geometric endomorphism ring $\End \Abar$ is defined, called the \defi{endomorphism field of $A$}.  By a result of Silverberg~\cite[Proposition 2.2]{MR1154704}, for all $m \geq 3$ we have
\begin{equation} \label{eqn:Silv}
\KEnd \subseteq K(A[m]),
\end{equation}
where $K(A[m])$ is the field generated by the $m$-torsion of $A$.

Let $\ell$ be prime.  The rational $\ell$-adic Tate module $V_\ell(A) \colonequals T_\ell(A) \otimes_{\Z_\ell} \Q_\ell$ of $A$ has a natural $\Gal_K$ action, furnishing an $\ell$-adic representation $\rho_{A,\ell} \colon \Gal_K \to \GL(V_\ell(A))$.  Let $\calG_{A,\ell}$ be the Zariski closure of the image of $\rho_{A,\ell}$, a linear algebraic group over $\Q_\ell$ called the \defi{$\ell$-adic monodromy group of $A$}.  Let $\calG_{A,\ell}^0 \leq \calG_{A,\ell}$ be the connected component of the identity.  Serre~\cite[no.~133]{serre-IV} (see also Larsen--Pink~\cite[Propositions (6.12) and (6.14)]{LP92}) proved that there exists a finite group $\Phi_A$ and a surjective group homomorphism
\begin{equation} 
\varepsilon_A \colon \Gal_K \to \Phi_A, 
\end{equation}
such that the induced map to the component group of $G_{A,\ell}$
\begin{equation}\label{eq:EpsilonEll}
    \varepsilon_{A,\ell} : \Gal_K \to \pi_0(\calG_{A,\ell})=\frac{\calG_{A,\ell}(\Q_\ell) }{ \calG_{A,\ell}^0(\Q_\ell)}
\end{equation}
factors through $\varepsilon_A$ via a canonical group isomorphism $\Phi_A \simeq \pi_0(\calG_{A,\ell})$.  The finite extension of $K$ cut out by $\ker \varepsilon_A$ is accordingly denoted $\Kconn$ (also denoted $\Kconnup_A$ by other authors) and called the \defi{connected monodromy field} of $A$.  The field $\Kconn$ is the minimal extension over which the $\ell$-adic monodromy groups of the base change of $A$ become connected for all primes $\ell$.  By a theorem of Larsen--Pink~\cite[Theorem 0.1]{MR1441234}, the connected monodromy field is also described as the intersection over all primes $\ell$ of the fields generated by all $\ell$-power torsion points of $A$:
\begin{equation} \label{eqn:LP}
\Kconn = \bigcap_{\ell} K(A[\ell^\infty]).
\end{equation}
Combining~\cref{eqn:Silv,eqn:LP}, we see that 
\begin{equation}  \label{eqn:KendKconn}
\KEnd \subseteq \Kconn.
\end{equation}

We are therefore led to the following (folklore) question.

\begin{question} \label{ques:kendkconn}
Let $A$ be an abelian variety over a number field $K$.  Under what conditions on $A$ does the equality $\KEnd = \Kconn$ hold?  
\end{question}

In other words, when is it the case that the $\ell$-adic monodromy groups are connected whenever all endomorphisms of $A$ are defined?  Banaszak--Kedlaya~\cite[Theorem 6.10]{MR3320526} proved that equality holds when $\dim A \leq 3$.  Of course, equality holds for any principally polarized $A$ with maximal $\ell$-adic image $\GSp_{2g}(\Z_\ell)$ for some prime $\ell$, since then the $\ell$-adic monodromy group $G_{A, \ell}=\GSp_{2g, \Q_\ell}$ is already connected; hence `most' abelian varieties $A$ over $K$ have $K(\End A)=K(\varepsilon_A)=K$.  

Silverberg--Zarhin \cite[Examples 4.1, 4.2]{SZ} constructed examples of abelian varieties $A$ where $\KEnd \subsetneq \Kconn$, coming from twists and the theory of complex multiplication.  
They give general criteria  \cite[Theorems 3.4, 3.5]{SZ} yielding $\Kconn \neq \KEnd$ when a CM field embeds in the center of the geometric endomorphism algebra or among quadratic twists of a product of two abelian varieties admitting a common CM subfield.  

\subsection{Results}

Our first main result is as follows.  

\begin{theorem}\label{thm:InfinitelyManyExamplesInequality}
For all even integers $g \geq 4$, there exist infinitely many geometrically nonisogenous, geometrically simple abelian varieties $A$ over $\Q$ with $\dim A=g$ and $\QEnd \subsetneq \Qconn$.
\end{theorem}

Recently, Banaszak--Cantoral-Farf\'an~\cite[Corollary 5]{ComponentsSatoTate} showed that, for a class of abelian varieties that includes ours, the relative degree $[\Kconn : \KEnd]$ is at most $2$.  (For a precise statement, see \Cref{thm:UpperBoundConnectedField} below.)   \Cref{thm:InfinitelyManyExamplesInequality} implies that this bound is in general sharp within this class. 

We now briefly indicate our construction.  We find our examples among those with geometric endomorphism algebra of Type III (definite quaternionic multiplication) in the Albert classification.  (In particular, our abelian varieties do not overlap with the classes studied by Silverberg--Zarhin \cite{SZ}.)  More precisely, we work with Jacobians of curves whose automorphism group contains the quaternion group $Q_8$ of order~$8$. 

For $g \geq 4$ even, let $d \colonequals g/2-1$, and define the family of nice curves with affine model
\begin{equation} \label{eqn:Cag}
\calC^{(g)} : \; y^2 = x(x^4-1)\left( x^{2g-4}+1 + \sum_{j=1}^{d} a_j( x^{2g-4-2j} + x^{2j}) \right)
\end{equation}
over the complement of the discriminant locus $\Delta \subset \A_{\Q}^d=\Spec(\Q[a_1,\dots,a_d])$.  We will be particularly interested in the case $g=4$ (with $d=1$), given by 
\[ y^2=x(x^4-1)(x^4+2a_1x^2+1) \] 
with $\Delta=-2^{24}(a_1^2-1)^6$.  The fibres $\calC_a^{(g)}$ for $a \in (\A^d \smallsetminus \Delta)(K)$ are nice curves of genus $g$ with $Q_8 \hookrightarrow \Aut \calC^{(g)}_a$ whenever $i=\sqrt{-1} \in K$; accordingly, their Jacobians $\calA^{(g)}_a \colonequals \Jac \calC^{(g)}_a$ have endomorphisms by the Lipschitz (quaternion) order 
\begin{equation}
\calO \colonequals \Z + \Z i + \Z j + \Z k \subset \quat{-1,-1}{\Q}
\end{equation}
whenever $i \in K$.  

The family $\calC^{(g)}$ has many pleasant properties.  This family is `universal' in the sense that every hyperelliptic curve of genus $g$ with a certain action~\cref{eqn:Q8rho} by $Q_8$ over $K$ is isomorphic to a specialization $\calC^{(g)}_a$ for some $a \in K^d$, up to quadratic twist (\Cref{prop:gcwitha0univ}).  We also observe that, whenever a curve $C$ has geometrically simple Jacobian and an action by $Q_8$, then it is hyperelliptic (\Cref{lem:jacsimp}).  Finally, for $g=4$, we show the family is universal in a second sense: the closure $\PP^1$ of the base $\A^1 \smallsetminus \Delta$ is in fact the coarse space of a certain moduli space of abelian varieties of explicit PEL type (\Cref{thm: from the moduli space to the Neron model}); it is striking that this Shimura variety admits such an explicit description.

With this family in hand, our main result---one that implies a precise form of \Cref{thm:InfinitelyManyExamplesInequality}---is as follows.

\begin{theorem}\label{thm:MainHigherGenus}
Let $g \geq 4$ be even.  Then for a density $1$ subset of points $\aunder=(a_1,\ldots,a_d) \in \Q^d$ ordered by height, the curve $\calC^{(g)}_a$ over $\Q$ is nice of genus $g$ and its Jacobian $A \colonequals \calA^{(g)}_a$ satisfies all of the following:
\begin{enumroman}
\item the geometric endomorphism ring is $\End \Abar = \mathcal{O}$;
\item the endomorphism field of $A$ is $\Q(\End A)=\Q(i)$;
\item the connected monodromy field of $A$ is $\Q(\varepsilon_{A})=\Q(\zeta_8)$;
\item $A$ satisfies the Hodge, Tate, and Mumford--Tate conjectures; and
\item the Mumford--Tate group of $A$ is a $\Q$-form of $\GSO_{g}$.
\end{enumroman}
Moreover, as $\aunder \in \Q^d$ varies within this subset, the Jacobians $\calA^{(g)}_a$ fall into infinitely many distinct $\Qbar$-isogeny classes.
\end{theorem}

Theorem \ref{thm:MainHigherGenus} relates to several recent papers.
\begin{itemize}
\item We complete a recent example of Zywina~\cite[Section 1.8]{Zywina2020}: we prove that for the Jacobian of the curve $y^2 = x(x^{20}+7x^{18}-7x^2-1)$ over $\Q$, the connected monodromy field is $\Q(\zeta_8)$, as claimed.  For details, see~\Cref{rmk: ExampleZywina}, with key input provided by the theorem of Banaszak--Cantoral-Farf\'an mentioned above.
\item 
In recent work, Gajda--Hindry~\cite{HindryGajda} prove the Mumford--Tate conjecture for certain abelian varieties with extra endomorphisms. They considered the family $\calA_a^{(g)}$ as a source of examples to which their results might apply and, in private communication with the authors, conjectured that its geometric endomorphism ring is equal to $\mathcal{O}$ for generic values of $a$. \Cref{thm:MainHigherGenus}(i) proves this conjecture.
\item Goodson \cite[\S 5.2.3]{Goodson} gives an example of an abelian variety of Fermat type where the Sato--Tate group (see~\Cref{rmk:Sato-Tate group}) becomes connected over a degree $2$ extension of the endomorphism field. However, her identification of this extension was only conjectural, based on computational evidence; the identification was recently made unconditional in \cite[Example 6.4.8]{GalleseGoodsonLombardoI} and \cite{GalleseGoodsonLombardoII}.
\end{itemize}

The family \cref{eqn:Cag} is indeed special, but our approach also suggests a general strategy to understand the arithmetic of families of Jacobians of curves with large automorphism group.  See also work of van Geemen--Verra \cite{MR1928644} for abelian $8$-folds with quaternionic endomorphisms, parametrizing certain Prym varieties in genus $17$ and
related work of Donagi--Livn\'e \cite{DonLiv}.

\subsection{Context and discussion}

It remains a fundamental problem that certain abelian varieties carry exceptional classes in their (Betti or $\ell$-adic \'etale) cohomology which are not as yet known to be linear combinations of algebraic classes~\cite{murty-84}. Such exceptional classes can only exist on abelian varieties of dimension $g\geq 4$.
This can be explained quickly by combining Poincar\'e duality with the Lefschetz theorem on $(1,1)$ classes (or its $\ell$-adic counterpart, due to Faltings): all Hodge classes in $H^2(A)$ are algebraic, hence by duality the same is true for classes in $H^{2g-2}(A)$, and this suffices to prove the Hodge conjecture for abelian varieties of dimension $g \leq 3$. Therefore, the first nontrivial case to consider is $H^4(A)$ with $\dim A=4$, which is the topic of this paper.

By work of Murty~\cite[Theorem 3.1]{murty-84}, the following are equivalent for a simple abelian variety $A$:
\begin{itemize}
    \item every algebraic class on any power of $A$ is generated by divisor classes; and
    \item the Hodge and Lefschetz groups of $A$ coincide (see~\Cref{def:Lefschetz}), and $A$ is not of Type III in the sense of Albert.
\end{itemize}
One can show that all divisor classes are defined over the field $K(\End A)$: working in \'etale cohomology, we have that $H^2_{\et}(\Abar, \mathbb{Q}_\ell)(1)$ is a direct summand of 
\begin{equation}\label{eq: Tate}
\left(H^1_{\et}(\Abar, \mathbb{Q}_\ell) \otimes H^1_{\et}(\Abar, \mathbb{Q}_\ell) \right)(1)=\Hom_{\Q_\ell}(V_\ell(A), V_\ell(A)),
\end{equation}
see for example Tate~\cite[p.~143]{MR206004}.
It follows that any (divisor) class in $H^2_{\et}(\Abar, \mathbb{Q}_\ell)(1)$ fixed by a finite-index subgroup $\Gal_L \leq \Gal_K$ lies in the subspace of $\Gal_L$-fixed points 
\[ \Hom_{\Q_\ell}(V_\ell(A), V_\ell(A))^{\Gal_L} = \End(A_L) \otimes \Q_\ell, \] 
hence it is fixed by $\Gal_{K(\End A)}$.
Thus simple abelian varieties $A$ with $K(\End A) \neq K(\varepsilon_A)$ should either have different Hodge and Lefschetz groups or they should be of Type III.  The former seems challenging to construct explicitly, so in this paper we focus on the case of varieties of Type III.  For these, Murty~\cite[Theorem 3.2]{murty-84} shows the existence of Hodge classes in $A^2$ not generated by divisors.  These classes underlie many of the constructions in this paper.  We describe them geometrically (see~\cref{sect:SatoTateField}) as Weil classes for a quadratic subfield of the endomorphism ring (\Cref{rmk: exceptional class as Weil class}) and as a geometric incarnation of a certain determinant on a submodule of cohomology (\Cref{rmk:DeterminantIsAlgebraicClass}).

\subsection{Strategy of proof} 

The overall strategy of the proof of our main result (\Cref{thm:MainHigherGenus}) is first to give an essentially direct argument in genus $g=4,6$, spreading out from explicit computation on a specialization, then to argue by induction to treat the general case $g \geq 4$.  Significant input is given by the aforementioned theorem of 
Banaszak--Cantoral-Farf\'an~\cite{ComponentsSatoTate} and the construction of cycles provided by Schoen~\cite{MR930145}.  As an alternative to our inductive approach, it may be possible to establish some parts of our main result by studying the Galois action on the $2$-torsion points (as in Zarhin \cite{Zarhin}).

We now make a few comments about the proof of part (iv) of~\Cref{thm:MainHigherGenus}.  On one hand, it is known in many cases that every `sufficiently general' member of a family of abelian varieties satisfies the Mumford--Tate conjecture. For example, Zywina~\cite{ZywinaFamilies} shows that, in a family of abelian varieties over a rational base, a density 1 subset of the fibres have the same $\ell$-adic monodromy group as the generic fibre; this can often be used to show that a density 1 subset of the fibres satisfy the Mumford--Tate conjecture.  We also adopt this approach. Noot~\cite{MR1355123} (based on ideas of Serre) constructs families of abelian varieties whose Mumford--Tate group is contained in a given group $G$, and he shows that `most' members of these families (meaning away from a thin set on the base) satisfy the Mumford--Tate conjecture. Moreover, the Mumford--Tate conjecture is known for certain abelian varieties of Type III~\cite{MR2663452,MR1603865,Cantoral}.

On the other hand, there are not many known cases of the Hodge and Tate conjectures for abelian varieties of Type III---this is essentially due to the existence of exceptional classes. To prove these conjectures for the Jacobians studied in this paper, we rely on previous work of Schoen~\cite{MR930145} on abelian varieties with an automorphism.  However, Schoen's result is not enough to establish part (iii) of~\Cref{thm:MainHigherGenus}, because we also need to know the minimal field of definition of these algebraic classes---we determine this field by writing down explicit cycles representing them.

\subsection{Organization}

After dispensing with preliminaries in~\cref{sec:prelim}, in~\cref{sect:models} we present our family of hyperelliptic curves with automorphisms by $Q_8$.  In~\cref{sect:g4}, we consider a specialization in $g=4$ and prove parts of~\Cref{thm:MainHigherGenus} in this special case.  In~\cref{sec: endosclasses} and~\cref{sec: higher genera}, we proceed with the general case $g \geq 4$ by an inductive argument.  We conclude in~\cref{sec: a Shimura variety} with some final remarks, including the interpretation in $g=4$ as a moduli space of abelian varieties.  

\subsection{Acknowledgements}

The authors would like to thank Alex Betts, J\c{e}drzej Garnek, Marc Hindry, Igor Shparlinski, Alice Silverberg, Bert van Geemen, Jeff Yelton, Yuri Zarhin, and David Zywina for helpful conversations, and Giulio Bresciani for his many pertinent comments on~\cref{sec:explrec}. Lombardo is a member of the INdAM group GNSAGA and was supported by a MIUR grant (PRIN 2017 ``Geometric, algebraic and analytic methods in arithmetic"). Voight was supported by a Simons Collaboration grant (550029). 

\subsection{AI declaration}

The manuscript was prepared without AI.  After referee reports were incorporated, ChatGPT 6 Sol checked the revisions to sections 7.2--7.3 and it made several other proofreading suggestions.  

\section{Preliminaries} \label{sec:prelim}

In this section, we establish the setup and notation.

\subsection{Notation} For a vector space $V$ over a field $K$ we denote by $\GL_V$ the algebraic group given as a functor on $K$-algebras $B$ by $B \mapsto \operatorname{Aut}(V \otimes_K B)$. We denote by $\GL(V)$ the group of points $\GL_V(K) = \operatorname{Aut}_K(V)$.  We use similar notation when $V$ is a free module over a ring $R$. When $V$ is a $\Q$-vector space and $F$ is any field of characteristic 0, we write $V_F$ for $V \otimes_\Q F$. If $G$ is an algebraic subgroup of $\GL_V$, we similarly write $G_F$ for $G \times_{\Spec \Q} \Spec F$.

\subsection{Abelian varieties and associated groups}\label{subsec: abelian varieties}

Let $K \subset \C$ be a finitely generated subfield with algebraic closure $\Kbar \subset \C$.  (The choice of embedding $K\hookrightarrow \C$ does not affect any of the following constructions by work of Deligne~\cite[Theorem 2.11]{MR654325}.)  Let $A$ be an abelian variety of dimension $g$ defined over $K$, with $\Abar \colonequals A \times_K \Kbar$ its base change to $\Kbar$.  For an algebraic extension $K' \supseteq K$, we denote by $\End(A_{K'})$ the endomorphism ring of $A_{K'} \colonequals A \times_K K'$ (endomorphisms defined over $K'$), and $\Endzero{A_{K'}} \colonequals \End(A_{K'}) \otimes \Q$.  In particular, $B \colonequals \Endzero{\Abar}$ is the geometric endomorphism algebra of $A$.  

We briefly recall the definition of several algebraic groups naturally attached to $A$.  For further general reference, see Banaszak--Kedlaya~\cite[\S 4]{MR3320526}. 
A polarisation of $A$ yields a symplectic bilinear form $\psi$ on $V \colonequals H_1(A(\mathbb{C}), \Q)$. There is also a natural action of $\Endzero{\Abar}$ on $V$, giving a homomorphism $B \to \operatorname{End}(V)$.

\begin{definition}\label{def:Lefschetz}
The \defi{Lefschetz group} of $A$ is $L(A) \colonequals C_B(\Sp_{V,\psi})^0$, the identity component of the subgroup of $\Sp_{V,\psi}$ of elements commuting with $B$.
\end{definition}

The Lefschetz group is a connected, reductive $\Q$-algebraic subgroup of $\Sp_{V,\psi}$. It was first introduced by Murty~\cite[\S 2]{murty-84}. 
The (possibly disconnected) centraliser $C_B(\Sp_{V,\psi})$
is the group denoted by $S(A)$ in Milne~\cite{MR1671217} and by $\mathrm{DL}^{\mathrm{id}}_K(A)$ in
Banaszak--Cantoral-Farf\'an~\cite{ComponentsSatoTate}; in particular, $L(A)=S(A)^0=\mathrm{DL}^{\mathrm{id}}_K(A)^0$.

The Lefschetz group can be meaningfully compared with another algebraic group attached to $A$: namely, its Hodge group, defined as follows.  Recall that $V$ is naturally a Hodge structure of weight $-1$, i.e., we have $V_{\C}=V^{-1,0}\oplus V^{0,-1}$ with the property that $\overline{V^{-1,0}}=V^{0,-1}$. 
Let $\mathbb S \colonequals \operatorname{Res}_{\C|\R}(\mathbb{G}_{m,\C})$ be the Deligne torus.  Then equivalently we have a Hodge cocharacter
\begin{equation}
    h: \mathbb{S} \to \GL_{V_{\R}}
\end{equation}
where $z=(z_1,z_2)\in \mathbb{S}(\mathbb{C})=\mathbb{C}^{\times}\times \mathbb{C}^{\times}$ acts as multiplication by $z_1$ in $V^{-1,0}$ and as multiplication by $z_2$ in $V^{0,-1}$~\cite[\S 2.1]{MR3735565}.

\begin{definition}
The \defi{Mumford--Tate} group of $A$, denoted $\MT(A)$, is the smallest $\Q$-algebraic subgroup of $\GL_V$ such that $h$ factors via $\MT(A)_{\mathbb{R}}$. The \defi{Hodge group} of $A$, denoted $\Hg(A)$, is the identity component of the intersection $\MT(A) \cap \Sp_{V,\psi}$. 
\end{definition}

By construction, $\MT(A)$ is the almost-direct product inside $\GL_V$ of $\Hg(A)$ with $\mathbb{G}_m$, considered as the central torus of homotheties.

For a prime number $\ell$, let $T_\ell = T_{\ell}(A) \colonequals \varprojlim_n A[\ell^n] \simeq H^1_{\et}(\Abar,\Z_\ell)\spcheck$ be the $\ell$-adic Tate module of $A$.  The absolute Galois group $\Gal_K \colonequals \Gal(\Kbar\,|\,K)$ acts naturally on the \'etale cohomology of $\Abar$ and therefore we have a continuous $\ell$-adic representation
\begin{equation}
    \rho_{A,\ell}: \Gal_K \to \GL(T_\ell)< \GL(V_\ell),
\end{equation}
where $V_\ell = V_\ell(A) \colonequals T_\ell \otimes_{\Z_\ell} \Q_\ell$.  

\begin{definition}\label{def: monodromy group}
The \defi{$\ell$-adic monodromy group} of $A$, denoted $G_{A, \ell}$, is $\overline{\rho_{A,\ell}(\Gal_K)}^{\text{Zar}} \leq \GL_{V_\ell}$, the Zariski closure of the image of the $\ell$-adic representation.  
\end{definition}

We abbreviate $G_\ell = G_{A,\ell}$ when $A$ is clear from context, and we write $G_\ell^0=G_{A,\ell}^0$ for its identity component.  We similarly define the \defi{special $\ell$-adic monodromy group} $G_{\ell,1} \colonequals G_{\ell}\cap \Sp_{V_{\ell}}$.

\begin{remark}\label{rmk:Sato-Tate group}
We should also mention the \defi{Sato--Tate group} of $A$, a maximal compact subgroup of $G_{\ell,1}(\C)$ \cite[\S 8.3.3]{serre-lectures} (unique up to conjugation) obtained from a chosen embedding $\Q_{\ell}^{\al} \hookrightarrow \C$.  Although the Sato--Tate group will not play a featured role here, our main result informs its description for fibres in our family.
\end{remark}

The Lefschetz, Mumford--Tate, Hodge, and $\ell$-adic monodromy groups of $A$ can be compared to one another in the following precise sense.  First, we may identify $V_{\Q_\ell} \simeq V_\ell(A)$ using Artin's comparison isomorphism in \'etale cohomology. 

\begin{theorem}\label{thm:ContainmentsGlMTLefschetz}
The following hold:
\begin{enumalph}
    \item For all $\ell$, we have $G_{\ell}^0 \leq \MT(A)_{\Q_{\ell}}$ and $G_{\ell,1}^0 \leq \Hg(A)_{\Q_{\ell}}$.
    \item We have $\Hg(A) \leq L(A)$. 
\end{enumalph}
\end{theorem}

\begin{proof}
For part (a), see Deligne~\cite[Proposition 2.9, Theorem 2.11]{MR654325}. More precisely, every Hodge cycle on $A$ is absolutely Hodge \cite[Theorem 2.11]{MR654325} and (the \'etale components of) all absolutely Hodge cycles are defined over a finite extension $L$ of $K$ \cite[Proposition 2.9(b)]{MR654325}. Thus, the absolute Galois group of $L$ acts trivially on all Hodge cycles, which are precisely the fixed points of the action of the Mumford--Tate group. By Galois theory, this implies that $G_{A_L, \ell} \leq \MT(A)_{\Q_\ell}$, and in particular $G_{\ell}^0 \leq G_{A_L, \ell}$ is contained in $\MT(A)_{\Q_\ell}$, which proves (a).
For part (b), see Milne~\cite[p.665, before Proposition 4.8]{MR1671217} or Murty~\cite[Section 2, before Lemma 2.1]{murty-84}.
\end{proof}

\Cref{thm:ContainmentsGlMTLefschetz}(a) leads naturally to the following well-known conjecture.

\begin{conjecture}[Mumford--Tate]
Let $A$ be an abelian variety over a finitely generated field $K \subset \C$.  Then for all prime numbers $\ell$, the natural embedding $\calG_{A,\ell}^0 \hookrightarrow \MT(A)_{\Q_{\ell}}$ is an isomorphism.
\end{conjecture}

On the other hand, in general we cannot expect the Lefschetz and Hodge groups to agree.  Accordingly, we make the following definition.

\begin{definition}\label{def:FullyLefschetz}
We say that $A$ is \defi{fully of Lefschetz type} if the following two conditions hold:
\begin{enumroman}
\item the inclusion $\Hg(A) \leq L(A)$ is an equality; and
\item the Mumford--Tate Conjecture holds for $A$.
\end{enumroman}
\end{definition}

\subsection{Type III}

Recall that if $A$ is simple, then $A$ is \defi{of Type III} in the Albert classification if its geometric endomorphism algebra $B=\End(\Abar)_\Q$ is a totally definite quaternion algebra. The center of $B$ is then a totally real number field $F \colonequals Z(B)$, and the \defi{relative dimension} of $A$ is 
\begin{equation} 
\frac{\dim A}{2[F:\Q]}=\frac{g}{2e} \in \Z_{\geq 1}
\end{equation}
where $e \colonequals [F:\Q]$, and in particular $g=\dim A$ is even.  

\begin{lemma}\label{lemma:Lefschetz}
Suppose $A$ is simple of Type III, and suppose its geometric endomorphism algebra has center $Z(B)=\Q$.  Then the Lefschetz group $L(A)$ has rank $g/2$ and dimension $(g^2-g)/2$, and $L(A)$ is a $\Q$-form of $\SO_g$.
\end{lemma}

\begin{proof}
We refer to Milne~\cite[Summary of Section 2]{MR1671217}. Notice that the dimension and rank of an algebraic group only depend on its identity component, and our $L(A)$ is the identity component of Milne's group $S(A)$.
\end{proof}

\begin{remark}
Let $A$ be a simple abelian variety of Type III defined over a number field $K$.  By a theorem of Banaszak--Gajda--Kraso\'n~\cite[Theorem 5.11]{MR2663452}, the Mumford--Tate conjecture holds for $A$ if the relative dimension of $A$ is odd and not in an explicit list of exceptions (see Lombardo~\cite[Remark 2.27]{MR3494170} and Cantoral-Farf\'an~\cite[Corollary 1.11]{Cantoral}). 

The Mumford--Tate conjecture also holds for simple abelian fourfolds of Type III, by work of Moonen--Zarhin~\cite[\S 6]{MR1324634}.  On the other hand, for infinitely many even values of $g$, the Mumford--Tate conjecture is not known to hold for a general, simple $g$-dimensional abelian variety of Type III. 
\end{remark}

The following result was part of the inspiration for the investigations in this paper.

\begin{theorem}\label{thm:UpperBoundConnectedField}
Let $A$ be a simple abelian variety over $K$ of Type III, fully of
Lefschetz type. Then we have a natural injective group homomorphism
\begin{equation}
\Gal(\Kconn\,|\,\KEnd)\hookrightarrow (\Z/2\Z)^e,
\end{equation}
where $e=[Z(\End(A^{\al})_\Q):\Q]$. 
In particular, $[\Kconn:\KEnd]$ divides $2^e$, and if $Z(\End(A^{\al})_\Q)=\Q$, then $[\Kconn:\KEnd]\leq 2$.
\end{theorem}

\begin{proof}
Let $\ell$ be prime and abbreviate
$G_\ell=G_{A_{\KEnd},\ell}$ and $G_{\ell,1}=G_\ell\cap\Sp_{V_\ell,\psi_\ell}$. 
By \cite[(6.4)] {BanaszakKedlaya2016}, together with the fact that the connected monodromy field is independent of $\ell$, we have $\pi_0(G_\ell)\simeq\Gal(\Kconn\,|\,\KEnd)$.
Moreover, \cite[(38)]{ComponentsSatoTate}
gives an isomorphism
$\pi_0(G_\ell)\simeq\pi_0(G_{\ell,1})$.
Since the group $\mathrm{DL}^{\mathrm{id}}_K(A) = \mathrm{DL}_{\KEnd}(A)$ in \cite{ComponentsSatoTate} is Milne's group $S(A)$, by \cite[(22)]{ComponentsSatoTate} we get $G_{\ell,1}\leq S(A)_{\Q_\ell}$.
As $A$ is fully of Lefschetz type, \cite[(37)]{ComponentsSatoTate} gives $G_{\ell,1}^0 = \Hg(A)_{\Q_\ell}=L(A)_{\Q_\ell} =S(A)^0_{\Q_\ell}$.
It follows that the inclusion $G_{\ell,1}\leq S(A)_{\Q_\ell}$
induces a natural injection $\pi_0(G_{\ell,1})\hookrightarrow\pi_0(S(A)_{\Q_\ell})$.

Let $F \colonequals Z(\End(A^{\al})_\Q)$.  Milne's description of $S(A)$ for a
simple abelian variety of Type III~\cite[p.~653]{MR1671217} gives $S(A)_{\overline{\Q}_\ell} \simeq \prod_{\sigma:F\hookrightarrow\overline{\Q}_\ell} \mathrm{O}(\varphi_\sigma)$, where each factor in the product has component group isomorphic to $\Z/2\Z$.
We thus obtain $\pi_0(S(A)_{\overline{\Q}_\ell})\simeq(\Z/2\Z)^e$, and hence $\Gal(\Kconn\,|\,\KEnd) \simeq \pi_0(G_\ell) \simeq \pi_0(G_{\ell,1}) \hookrightarrow (\Z/2\Z)^e$,
which proves the result.
\end{proof}

\begin{remark}
\Cref{thm:UpperBoundConnectedField} can be extended to abelian varieties isogenous to a power of a simple abelian variety of Type III.
\end{remark}

\subsection{Connected monodromy field}

The identity component $\calG_{A,\ell}^0$ of the $\ell$-adic monodromy group is invariant under finite extensions of the ground field $K$. Moreover, as we observed in the introduction~\cref{eq:EpsilonEll}, there is a minimal extension $\Kconn$ of $K$ over which all $\ell$-adic monodromy groups become connected, and we have canonical group isomorphisms
\begin{equation}
\Gal(\Kconn\,|\,K) \simeq \pi_0(\calG_{A,\ell})
\end{equation}
for each prime $\ell$.  

We now describe the field $\Kconn$ in terms of the geometry of $A$ and of its algebraic cycles. Recall that the cycle class map of \'etale cohomology attaches to each codimension $d$ cycle $Z$ on $\Abar$ a class $[Z]$ in $H_{\et}^{2d}(\Abar, \Q_\ell(d))$, invariant under a subgroup of $\Gal_K$ of finite index. The \defi{(minimal) field of definition} of a class $\delta \in  H_{\et}^{2d}(\Abar,\Q_\ell(d))$ is the fixed field of the subgroup of $\Gal_K$ that fixes $\delta$; we will only be interested in classes $\delta$ for which this subgroup has finite index in $\Gal_K$.

In order to better understand the Galois action on $H_{\et}^{2d}(\Abar,\Q_{\ell}(d))$, recall that, for all $i \geq 0$, we have an isomorphism
\begin{equation}
H_{\et}^i(\Abar, \Q_\ell) \simeq \tbigwedge^i H_{\et}^1(\Abar, \Q_\ell) 
\end{equation}
as $\Gal_K$-representations.

\begin{proposition}\label{prop:LowerBoundConnectedField}
Let $\delta\in H^{2d}_{\et}(\Abar,\Q_\ell(d))$ be fixed by an open subgroup of $\operatorname{Gal}_K$. Then its field of definition is contained in $\Kconn$.
\end{proposition}

\begin{proof}
Let $L$ be the field of definition of $\delta$.  By definition, $\delta$ is invariant under the action of $\Gal_{L}$ hence also under the action of $\calG_{A_L, \ell}$ (via the $2d$-th exterior power of its defining representation, tensored with $\Q_\ell(d)$). In particular, it is invariant under 
\[
\calG_{A_L, \ell}^0 = \calG_{A_{\Kconn}, \ell}^0 = \calG_{A_{\Kconn}, \ell}.
\]
Thus $\Gal_{\Kconn}$ acts trivially on $\delta$, hence $L \subseteq \Kconn$.
\end{proof}

From the preceding proposition, it is not hard to characterise the field $\Kconn$ in geometric terms, at least under the assumption of the Tate conjecture.

\begin{corollary}\label{cor:apres proposition} 
Let $L$ be the minimal extension of $K$ such that, for all $k \in \Z_{\geq 1}$ and $d \in \Z_{\geq 0}$, and for all codimension $d$ cycles $Z$ on $(A^\al)^k$, the field of definition of $[Z]$ is contained in $L$. Then the following statements hold:
\begin{enumalph}
\item $L \subseteq \Kconn$.
\item If the Tate conjecture for algebraic cycles holds for $A^r$ for all $r \geq 1$, then $L \supseteq \Kconn$ (and hence $L=\Kconn$).
\end{enumalph}
\end{corollary}

\begin{proof}
Part (a) follows from~\Cref{prop:LowerBoundConnectedField}. Conversely, we now show that $K(\varepsilon_A)\subseteq L$, or equivalently, that $G_{A_L,\ell}$ is connected for all $\ell$, assuming the Tate conjecture for algebraic cycles on all powers of $A$. 

For simplicity of notation, set $G \colonequals G_{A_L, \ell}$ and assume for purposes of contradiction that $G$ is not connected. Then the finite group $G/G^0$ admits a faithful representation in some tensor construction over $V \colonequals H^1(\Abar, \Q_\ell)$.  Explicitly, by a theorem of Chevalley~\cite[Proposition 3.1(c)]{MR654325}, there exists a finite collection of pairs $(m_i, n_i)$ such that the subspace
\begin{equation}
W \colonequals \left\{ x \in \bigoplus_i V^{\otimes m_i} \otimes (V^\vee)^{\otimes n_i} : g \cdot x = x\text{ for all $g \in G^0$} \right\}
\end{equation}
has the property that the subgroup of $G$ fixing $W$ pointwise is $G^0$. Equivalently, as stated, $G$ acts on $W$ via $G/G^0$, and the induced map $G/G^0 \to \GL_W$ is injective. Now observe that, due to the existence of the $G$-equivariant bilinear form coming from a polarization, we have $V^\vee \simeq V(1)$, so that $W$ as above is contained in a direct sum of spaces of the form $V^{\otimes n}(d)$.

Next, we notice that any $V^{\otimes n}$ occurs in the cohomology of some power of $A$, because it is a direct summand of $\tbigwedge^n (V^{\oplus n}) \simeq H^n( (\Abar)^n, \Q_\ell)$. The kernel of the natural map
\[
\Gal_L \to G(\Q_\ell) \to (G/G^0)(\Q_\ell)
\]
defines a finite extension $L'$ of $L$ with the property that $\Gal_{L'}$ acts trivially on $W$. By  the Tate conjecture on algebraic cycles for $A^r$, this implies that $W$ is spanned by classes of algebraic cycles on certain powers of $A$. On the other hand, $\Gal_L$ does not act trivially on $W$, since its image in $G(\Q_\ell)$ (hence in $(G/G^0)(\Q_\ell)$) is dense. This contradicts the very definition of $L$ and finishes the proof.
\end{proof}

\begin{remark}
This also gives another proof of the containment $\KEnd \subseteq \Kconn$ \cref{eqn:KendKconn}, since the graphs of endomorphisms are in particular algebraic cycles on $A^2$ (cf.~\cref{eq: Tate}).
\end{remark}

\subsection{Connected monodromy field: type III}\label{sect:MoreGenus4}

For the remainder of this section, suppose that $A$ is of type III, fully of Lefschetz type, with $Z(\End(A^\al)_\Q)=\Q$.  In particular, $g=\dim A$ is even.  
Then \Cref{thm:UpperBoundConnectedField} shows that $[\Kconn:\KEnd] \leq 2$. We prove some properties of the extension $\Kconn \supseteq \KEnd$.  Since we are only interested in this extension, we will suppose that $\KEnd = K$.

Recall from Banaszak--Gajda--Kraso\'n~\cite{MR2663452} that for all sufficiently large primes $\ell$, there is a $\Q_\ell$-vector space $W_\ell(A)$ equipped with a $\Gal_K$-action together with a $\Gal_K$-equivariant isomorphism
\begin{equation}
V_\ell(A) \simeq W_\ell(A)^{\oplus 2};
\end{equation} 
moreover, there is a nondegenerate, symmetric bilinear form 
\begin{equation}
\psi_\ell \colon W_\ell(A) \times W_\ell(A) \to \Q_\ell
\end{equation}
that is also $\Gal_K$-equivariant.  We have $\dim W_\ell(A) = (\dim V_\ell(A))/2=g$.  
Let $\GO_{W_\ell(A), \psi_\ell}$ be the group of similitudes of $\psi_\ell$, also known as the \defi{conformal group} of $\psi_\ell$, with similitude character
\begin{equation} 
\chi_\ell \colon \GO_{W_\ell(A),\psi_\ell} \to \Q_\ell^\times 
\end{equation}
so that
\[ \psi_\ell(\gamma x, \gamma y) = \chi_\ell(\gamma) \psi_\ell(x, y) \]
and $\chi_\ell(\gamma)^g=\det(\gamma)^2$ for all $\gamma \in \GO_{W_\ell(A),\psi_\ell}$ and $x,y \in W_\ell(A)$. 
The restriction of $\chi_\ell$ to the image of $\rho_{A,\ell}$ in $G_{A,\ell}$ is the cyclotomic character.

The algebraic monodromy group is a subgroup of $\GO_{W_\ell(A), \psi_\ell}$.  Moreover, as $A$ is fully of Lefschetz type, for every $\ell$ the connected component $\calG_{A,\ell}^0$ is the group of \emph{oriented} similitudes of $\psi_\ell$, denoted by authors as one of
\[ \GO^0_{W_\ell(A),\psi_\ell}=\GO^+_{W_\ell(A),\psi_\ell} = \GSO_{W_\ell(A),\psi_\ell}, \]
defined as the kernel of the character
\begin{equation} \label{eqn:omegaell}
\begin{aligned}
\omega_{A,\ell}=\omega_\ell \colon \GO_{W_\ell(A), \psi_\ell} &\to \{\pm 1\} \\
\omega_\ell(\gamma) &= \det(\gamma)\chi_\ell(\gamma)^{-g/2}.
\end{aligned}
\end{equation}
(The character $\omega_\ell$ is trivial on scalars,
and upon restriction to $\operatorname{O}_{W_\ell(A),\psi_\ell}$ it agrees with $\det$; of course $\operatorname{SO}_{W_\ell(A),\psi_\ell}$ is the kernel of $\det$ inside $\operatorname{O}_{W_\ell(A),\psi_\ell}$, so the description follows.)

Restricting this character to $\Gal_K$ gives the following.  Let $\mathfrak{p}$ be a prime of good reduction for $A$ with absolute norm $q \colonequals \Nm(\frakp)$, and let $\Frob_\frakp \in \Gal_K$ be a $\frakp$-Frobenius element.  Then 
\begin{equation} 
\det(1-\Frob_\frakp T \,|\, V_\ell(A))
= \det(1-\Frob_\frakp T\,|\, W_\ell(A))^2 
\end{equation}
and $\det(\rho_{A,\ell}(\Frob_\frakp))=q^g$, so $\det(\Frob_\frakp\,|\, W_\ell(A))=\pm q^{g/2}=\omega_\ell(\Frob_\frakp)\chi_\ell(\Frob_\frakp)^{g/2}$.  In particular, the character $\omega=\omega_A \colonequals \omega_{A,\ell} \circ \rho_{A,\ell}$ of $\Gal_K$ is independent of $\ell$.  

\begin{lemma} \label{lem:isinGOplus}
We have $\rho_{A,\ell}(\Frob_\frakp) \in \calG_{A,\ell}^0(\Q_\ell)$ if and only if $\det(\Frob_\frakp\,|\, W_\ell(A))=q^{g/2}$.
\end{lemma} 

\begin{proof}
We have $\gamma \in \calG_{A,\ell}^0(\Q_\ell)$ if and only if $\gamma$, restricted to one copy of $W_\ell(A)$, is in $\GSO_{W_\ell(A), \psi_\ell}$. \end{proof}

\begin{lemma} \label{lem:unramprimes}
$\omega$ is unramified at primes of (good or) semistable reduction of $A$.
\end{lemma}

\begin{proof}
Let $\ell$ be prime.  Since $\rho_{A,\ell}$ is unramified outside the set consisting of the primes above $\ell$ and the primes of bad reduction of $A$, so too is $\omega$.  By independence of $\ell$, we conclude $\omega$ is only ramified at primes of bad reduction.

Let now $\mathfrak{p}$ be a (nonzero) prime of the ring of integers of $K$ at which $A$ has bad but semistable reduction.  Let $\ell > 2$ be coprime to $\frakp$.  By the Galois criterion for semistable reduction, the inertia group at $\mathfrak{p}$ acts on $T_{\ell}(A)$ via unipotent matrices $u$ satisfying $(u-I)^2=0$. Both the determinant and the multiplier of such a matrix are equal to 1, so by definition $\omega$ is trivial on the inertia group at $\mathfrak{p}$, as desired.
\end{proof}

\section{Curves with \texorpdfstring{$Q_8$}{Q8} automorphisms}\label{sect:models}

In this section, we present a certain universal family of curves with automorphisms by the quaternion group $Q_8$ of order $8$.  

\subsection{Setup}

We begin with a brief setup.  Throughout, we may relax our hypothesis on the field $K$, asking only that $\opchar K \neq 2$.  For $S$ a variety over $K$, a \defi{family of nice curves of genus $g$} over $S$ is a flat morphism $\calC \to S$ whose fibres are \defi{nice} (smooth, projective, geometrically connected) curves of genus $g$.

To frame our family, we begin with the following lemma.

\begin{lem} \label{lem:jacsimp}
Let $C$ be a nice curve over $K$ of genus $g \geq 2$ and let $G \leq \Aut(\Cbar)$ be a subgroup isomorphic to $Q_8$, stable under $\Gal_K$.  Suppose further that $\Jac(C)$ is \emph{simple}.  Then $C$ is hyperelliptic and $g$ is even.
\end{lem}

\begin{proof}
Let $\iota \in G \simeq Q_8$ be the nontrivial central element.  Then $\iota$ is characteristic in $G$ so defined over $K$.  Then the morphism $C \to C' \colonequals C/\langle \iota \rangle$ is a map of curves over $K$ of degree $2$.  Let $g'$ be the genus of $C'$.  If $g' \geq 1$, then $\Jac(C')$ is an isogeny factor of $\Jac(C)$; by Riemann--Hurwitz we have $g' < g$, so the cofactor is nontrivial, which contradicts that $\Jac(C)$ is simple.  So $g'=0$, the curve $C$ is geometrically hyperelliptic, and $\iota$ is the hyperelliptic involution.  

To prove that $C$ is in fact hyperelliptic (i.e., we have $C' \simeq \PP_K^1$), we look at the action of $G$ on the rational $2$-adic Tate module $V_2(A)$.  Since $G$ is Galois stable, $V_2(A)$ has the structure of a $\Q_2[G]$-module compatible with the $\Gal_K$-action.  In the decomposition of $V_2(A)$ into irreducibles under $G$, in the previous paragraph we proved that the trivial representation does not occur; in fact, there can also be no one-dimensional characters because $\langle \iota \rangle$ acts trivially on such characters. But we have the Wedderburn decomposition $\Q_2[Q_8] \simeq \Q_2^4 \times \quat{-1,-1}{\Q_2}$, and the quaternion algebra $\quat{-1,-1}{\Q_2}$ is a division algebra.
Thus $V_2(A)$ is a direct sum of copies of this remaining irreducible, and by dimensions we get $2g \equiv 0 \pmod{4}$.  Thus $g$ is even.  But then it is well-known that geometrically hyperelliptic implies hyperelliptic: the image $C'$ of $C$ under the canonical embedding has odd degree $g-1$, so $C'$ has an odd degree point, and from a multiple of the canonical class we obtain a divisor on $C'$ over $K$ of degree $1$ and hence $C'(K) \neq \emptyset$, whence $C' \simeq \PP_K^1$.  
\end{proof}

\begin{remark}
The classification of hyperelliptic curves with $Q_8$ action is classical. For actions on non-hyperelliptic curves, see Kimura~\cite{MR1222174}, van Geemen--Verra \cite[\S 1]{MR1928644}, or Donagi--Livn\'e \cite{DonLiv}.
\end{remark}

\begin{lem}\label{lem:Q8 action gives splitting of quaternion algebra}
Let $C$ be a hyperelliptic curve over $K$ with $Q_8 \hookrightarrow \Aut(C)$ defined over $K$.  Then $\quat{-1,-1}{K} \simeq \M_2(K)$.
\end{lem}

\begin{proof}
The $Q_8$ action on $C$ gives an action by $Q_8/\{\pm 1\} \simeq C_2 \times C_2$ on $\PP^1_K$, i.e., we have a homomorphism $Q_8 \to \PGL_2(K)$ with kernel $\{\pm 1\}$.  This projective representation lifts (by a direct calculation~\cite[Corollary 2.5]{TwistsHyperelliptic}) to $Q_8 \hookrightarrow \GL_2(K)$, hence to a surjective $K$-algebra homomorphism $K[Q_8] \to \M_2(K)$.  Again by Wedderburn, we have $K[Q_8] \simeq K^4 \times \quat{-1,-1}{K}$ so we must have $\M_2(K) \simeq \quat{-1,-1}{K}$.  
\end{proof}

\begin{cor}\label{cor:nohypqq8}
There is no hyperelliptic curve over $\Q$ with $Q_8$ action defined over $\Q$. 
\end{cor}

\begin{proof}
Immediate from~\Cref{lem:Q8 action gives splitting of quaternion algebra}, since $\quat{-1,-1}{\Q}$ is a division algebra.
\end{proof}

\subsection{Definition of family}\label{subsec: definition of the family}

By the results in the previous section, in order to define a family of hyperelliptic curves of even genus over $\Q$ with an action by $Q_8$ over an extension of $\Q$, we will have to choose a splitting field for the quaternion algebra $\quat{-1,-1}{\Q}$.

We choose $\Q(i)$ as a splitting field for $\quat{-1,-1}{\Q}$ and the homomorphism
\begin{equation} \label{eqn:Q8rho}
\begin{aligned} 
\rho \colon Q_8 &\to \GL_2(\Q(i)) \\
i,j &\mapsto \begin{pmatrix} i & 0 \\ 0 & -i \end{pmatrix}, \begin{pmatrix} 0 & i \\ i & 0 \end{pmatrix}.
\end{aligned}
\end{equation}
This choice is convenient for our calculations; any other such map over $\Q(i)$ will be conjugate by an element of $\GL_2(\Q(i))$, giving an isomorphic family below.  Another choice of splitting field would correspondingly twist this family.  We note that the projective representation $\mathrm{P}\!\rho \colon Q_8 \to \PGL_2(\Q)$ is defined over $\Q$.  

For $Q_8$ to act on a hyperelliptic curve $y^2=f(x)$ with induced action~\cref{eqn:Q8rho} on the underlying $\PP^1$, we need 
\begin{equation} \label{eqn:havetosatisfy}
f(-x)=-f(x) \quad \text{and} \quad f(1/x)=-f(x)/x^{2g+2}.
\end{equation}
A calculation of invariants then leads to the following family.

Let $g \geq 2$ be an even integer and let $d \colonequals g/2-1$. We first consider the (relatively projective) curve over $\mathbb{A}_{\Q}^{d+1}$ obtained as the completion of the relative affine curve with hyperelliptic equation
\begin{equation} \label{eqn:Cgwitha0}
y^2 = x(x^4-1)\left( \sum_{j=0}^{d} a_j( x^{2g-4-2j} + x^{2j} ) \right).
\end{equation}

\begin{proposition} \label{prop:gcwitha0univ}
Let $X$ be a hyperelliptic curve of even genus $g$ over $K$ with hyperelliptic involution $\iota$. Suppose that $X$ admits an action of $Q_8$ defined over $K(i)$ whose induced action on $X/\langle \iota \rangle \simeq \PP^1$ is conjugate to \cref{eqn:Q8rho} in $\PGL_2(K)$.  Then $X$ is isomorphic over $K$ to a specialization of the family defined in~\cref{eqn:Cgwitha0}.
\end{proposition}

\begin{proof}
The first condition in~\cref{eqn:havetosatisfy} yields that every monomial in $f(x)$ has odd degree, hence in particular $f(0)=0$. The two conditions in~\cref{eqn:havetosatisfy} taken together also imply that $f(x)$ vanishes for $x=\pm 1, \pm i$. Thus we can factor $f(x)=x(x^4-1)f_1(x)$, and~\cref{eqn:havetosatisfy} then yields that $f_1(x)$ is a self-reciprocal polynomial in $x^2$, so that $f(x)$ is as in the right-hand side of~\cref{eqn:Cgwitha0}.
\end{proof}

Up to twist, in the family~\cref{eqn:Cgwitha0} we may take $a_0=1$.  So for the remainder of the paper, we work with the following setup.  

\begin{setup}\label{setup}
Let $g \geq 4$ be even and let $d \colonequals g/2-1$, so $g=2d+2$.  Let 
\[ f(x)=f^{(g)}(x) \colonequals x(x^4-1)\left(x^{2g-4} + 1 + \sum_{j=1}^{d} a_j( x^{2g-4-2j} + x^{2j} ) \right) \in \Q[a_1,\dots,a_d][x] \]
and let $\Delta = \Delta^{(g)}(a_1,\dots,a_d) \in \Q[a_1,\dots,a_d]$ be the discriminant of $f(x)$.  Define the family of nice curves 
\begin{equation} \label{eq:HyperellipticEquationGeneralg}
\calC=\calC^{(g)} \colon y^2 = f(x)
\end{equation}
over $\AAp$, where $\AAp \colonequals \Spec \Q[a_1,\dots,a_d,\Delta^{-1}] \subseteq \A^d_\Q$ is the complement of the discriminant $\Delta$.  Let $\calJ^{(g)} \to \AAp$ be the abelian scheme given by the relative Jacobian of $\calC^{(g)}$.  (We drop the superscript whenever the genus $g$ is fixed.)

Let $C=C^{(g)}$ be the generic fibre of $\calC$ over $\Q(a)=\Q(a_1,\dots,a_d)=\Q(U)$ and $A=A^{(g)}$ its Jacobian.  

For every field extension $K \supseteq \Q$ and every $\aunder = (a_1, \ldots, a_d) \in \AAp(K)$, we denote by $\calC_\aunder$ the fibre of $\calC$ over $\aunder$ and by $\calA_\aunder$ the Jacobian of $\calC_\aunder$ (over $K$).
\end{setup}

\begin{lemma}\label{lem: discriminant}
The subscheme $\AAp \subseteq \mathbb{A}^d_\Q$ is nonempty and open.
\end{lemma}
  
\begin{proof}
The polynomial $\Delta$ does not vanish identically: indeed, specializing at $a=(0,\dots,0)$ gives $x(x^4-1)(x^{2g-4}+1)$ which, since $g \geq 2$ is even, has no repeated roots.
\end{proof}

\begin{remark}
We will not need it in what follows, but the family can be extended over the base ring $\Z[1/2]$ in place of $\Q$.
\end{remark}

We will make extensive use of the $Q_8$-automorphisms of $\calC_{U_{\Q(i)}}$ generated by
\begin{equation}\label{eq: fundamental automorphisms}
\alpha(x,y) = (-x, iy), \quad \beta(x,y) = \left( \frac{1}{x}, \frac{iy}{x^{g+1}} \right)
\end{equation}
satisfying $\alpha^2=\beta^2=\iota$ and $\beta\alpha=\iota\alpha\beta$ where $\iota(x,y)=(x,-y)$ is the hyperelliptic involution.  For each $a \in U(K)$, the action of the automorphisms $\alpha,\beta$ over $K(i)$ on the subspace 
\[ \Q(i)\,\frac{\mathrm{d}x}{y} \oplus \Q(i)\, x^{g-1}\,\frac{\mathrm{d}x}{y} \subseteq H^0(\calC_a^{(g)},\Omega^1) \]
of regular differentials of $\calC_a^{(g)}$ is given by the representation of~\cref{eqn:Q8rho}. 

For each $a \in U(\Q)$, the automorphisms $\alpha,\beta$ induce endomorphisms
$\alpha^*,\beta^*$ of $\calA_a=\Jac(\calC_a)$. These endomorphisms generate a subring of
$\End(\calA_a^\al)$ isomorphic to the Lipschitz order
\begin{equation} \label{lemma: endomorphisms generic fibre}
\mathcal{O} \colonequals \mathbb{Z} \oplus \mathbb{Z} i \oplus \mathbb{Z} j \oplus \mathbb{Z} ij
\subset \quat{-1,-1}{\Q},
\end{equation}
under the map $i\mapsto\alpha^*$, $j\mapsto\beta^*$.  Moreover, these endomorphisms are defined over $\Q(i)$. We will also consider the automorphisms $\alpha(x,y)=(-x,iy)$ and $\beta(x,y)=(1/x, iy/x^{g+1})$ of the generic fibre $C$, defined over the field $\Q(i)(a_1,\ldots,a_d)$. As above, they generate a subring of $\End(A^\al)$ isomorphic to the Lipschitz order $\mathcal{O}$.

The essential task of the rest of the paper will be to investigate the Jacobians $\calJ_a$ for $a \in \Q^d$, in particular their endomorphism field and connected monodromy field.  

\begin{remark}
We could also consider $g=2$ ($d=0$) in \Cref{setup}: the base is a single point and the curve is
\[ y^2 = x(x^4-1). \]
This well-studied curve has geometric automorphism group $\GL_2(\F_3)$ (a group of order $48$), and its Jacobian is isogenous over $\Q^{\al}$ to the square of an elliptic curve with complex multiplication by $\Q(\sqrt{-2})$. 
\end{remark}

\subsection{A base change}\label{subsect:BaseChanges}

It will be convenient to consider a cover of the base of the family in \Cref{setup} where the defining hyperelliptic polynomial factors (labelling $2$-torsion), defined as follows.  
Consider the (dominant, quasi-finite) morphism
\[
\varphi : \Spec \Q[b_1,b_1^{-1},(b_1^4-1)^{-1},\dots,b_d,b_d^{-1},(b_d^4-1)^{-1}] \to \Spec \Q[a_1,\ldots,a_d]
\]
induced by the $\Q$-algebra homomorphism
\[
\Q[a_1,\ldots,a_d] \longrightarrow
\Q[b_1,b_1^{-1},(b_1^4-1)^{-1},\ldots,b_d,b_d^{-1},(b_d^4-1)^{-1}]
\]
obtained by expanding the polynomial in~\eqref{eqn:rootsbi} and equating its coefficients with those of $f^{(g)}(x)$ in~\Cref{setup}.
Equivalently, $\varphi$ may be described functorially on points as follows: given a commutative $\Q$-algebra $R$ and a point $(b_1,\ldots,b_d) \in (R^\times)^d$ with $b_j^4-1\in R^\times$ for $j=1, \ldots, d$, we consider the polynomial
\begin{equation} \label{eqn:rootsbi}
x(x^4-1)\prod_{j=1}^d (x^2-b_j^2)\left(x^2-\frac{1}{b_j^2} \right) \in R[x],
\end{equation}
expand out and equate coefficients with the polynomial $f^{(g)}(x)$ in \Cref{setup}, and then define $\varphi(b_1,\ldots,b_d) = (a_1,\ldots,a_d)$.  

\begin{example}\label{ex:transformation a b}
For $g=4$, we have $-2a_1=b_1^2+1/b_1^2$.
\end{example}

The pullback $\varphi^*(\Delta)$ of the discriminant $\Delta$ from \Cref{setup} yields a nonempty open $V \colonequals \Spec \Q[b_1,b_1^{-1},(b_1^4-1)^{-1},\dots,b_d,b_d^{-1},(b_d^4-1)^{-1},\varphi^*(\Delta)^{-1}]$ and the restriction $\varphi \colon V \to U$ is again dominant and surjective.  By construction, we have a Cartesian diagram
\begin{equation}\label{eq:BtoAdiagram}
\begin{aligned}
\xymatrix{
\Cprim^{(g)} \ar[r] \ar[d] & \calC^{(g)}  \ar[d] \\
\BB \ar[r]^{\varphi} & \AAp 
}
\end{aligned}
\end{equation}
where $\Cprim^{(g)} \to \BB$ is the family of nice hyperelliptic curves with affine equation
\begin{equation}\label{eq:HyperellipticEquationGeneralgFactored}
\Cprim^{(g)} \colon y^2 = x(x^4-1) \prod_{j=1}^d (x^2-b_j^2)\left(x^2-\frac{1}{b_j^2} \right).
\end{equation}
We denote by
\begin{equation}\label{eq:DefinitionBg}
    \mathcal B^{(g)} \colonequals \operatorname{Jac}(\mathcal D^{(g)}_V) =\mathcal A^{(g)}\times_U V
\end{equation}
the base-change of $\mathcal{A}^{(g)} = \Jac(\calC^{(g)})$ along the map $\varphi \colon V \to U$.
Let $\calL \colonequals \Q(\BB) = \Q(b_1,\ldots,b_d)$ be the function field of $\BB$. The map $\varphi$ induces an injection $\varphi^*(\calK) \hookrightarrow \calL$ that we will simply write as an inclusion.

Adjoining the roots of $f(x)$ gives $2$-torsion classes in the Jacobian, as follows.  
In view of \Cref{eq:HyperellipticEquationGeneralgFactored}, let
\begin{equation} \label{eqn:weierstrasspoints}
S \colonequals \{0, \pm 1, \pm i\} \cup \{\pm b_j, \pm 1/b_j\}_{j=1,\dots,d} \subset \Q(i)[b_1,1/b_1, \dots, b_d, 1/b_d]
\end{equation}
so that $(x,0) \in \Cprim^{(g)}(V_{Q(i)})$ for $x \in S$, together with the unique point $\infty$ at infinity, are the Weierstrass points.  The points with $x \in \{0,\pm 1,\pm i\}$ are distinguished from others by their stabilizers under the $Q_8$-automorphism group \eqref{eq: fundamental automorphisms}: 
\begin{equation} \label{eqn:stabilizersQ8}
\begin{aligned}
\operatorname{Stab}_{Q_8}(x,0) = 
\begin{cases}
\langle \alpha \rangle, & \text{ if $x=0$;} \\
\langle \beta \rangle, & \text{ if $x=\pm 1$;} \\
\langle \alpha\beta \rangle, & \text{ if $x=\pm i$;} \\
\langle \iota \rangle, & \text{ otherwise.}
\end{cases}
\end{aligned}
\end{equation}
(The point $\infty$ like $(0,0)$ has stabilizer $\langle \alpha \rangle$.)  In particular, the points with $x \in \{0,\pm 1,\pm i\}$ have $Q_8$-stabilizers of order $4$ (the others have order $2$).  

For $x \in S$, let $e_x$ be the divisor class $[(x,0) - \infty]$ on $\calJprim^{(g)}$.
It is well-known that $\calJprim^{(g)}[2]$ has a basis consisting of the divisor classes $e_x$ for $x \in S \smallsetminus \{0\}$ and $\sum_{x \in S} e_x = 0$.  In particular, $\calJprim^{(g)}$ has full level $2$ structure over $V_{\Q(i)}$.  

The action of the $Q_8$-automorphism group on $\calJprim^{(g)}[2]$ is similar: for $x \in S$, we have
\begin{equation}\label{eq:ActionOfI}
\alpha(e_x)=e_{-x} 
\end{equation}
and since $-e_0=e_0$ is $2$-torsion, we have
\begin{equation}\label{eq:ActionOfJ}
\beta(e_x)=\begin{cases} 
e_{1/x} + e_0, & \text{ if $x \neq 0$;} \\
e_{0}, & \text{ if $x=0$.}
\end{cases}
\end{equation}

\begin{remark}
We could also write $f(x)/x^{g-2}=x(x^4-1)h(x^2+1/x^2)$ where $h(y)=y^d + \sum_{j=1}^d a_j' y^{d-j}$ with $a_j' \in a_j+\Q[a_1,\dots,a_{j-1}]$; by universality, $h(y)$ has Galois group $S_d$, so by back substitution $f(x)$ has generic Galois group $D_2 \wr S_d = D_2^d \rtimes S_d$ of order $4^d d!$ (over $\Q(i)(a_1,\dots,a_d)$), where $D_2 \simeq \Z/2\Z \times \Z/2\Z$ is the dihedral group of order $4$.  
\end{remark}

\begin{remark}
One could similarly consider the intermediate base change with defining equation $y^2 = x(x^4-1) \prod_{i=1}^d (x^4+c_ix^2+1)$, which corresponds to a (partial) level-2 structure on the Jacobian.
\end{remark}

\subsection{Moduli} 
Let $\mathcal{M}_g$ be the moduli stack of curves of genus $g$ and let $M_g$ be its coarse space.  There is a morphism $\AAp \to \mathcal{M}_g \to M_g$ (with image in the hyperelliptic locus).  The next lemma describes this map (generically).

\begin{lemma} \label{lemma:MapToModuliSpace}
The following statements hold.
\begin{enumalph}
\item The geometric automorphism group of the generic fibre $C$ is $Q_8$, generated by $\alpha,\beta$ in \Cref{eq: fundamental automorphisms} and defined over $\calK(i)$.  
\item The map $\AAp \times_\Q \Q(i) \to M_g$ is generically Galois over its image, with Galois group $S_3$; in particular, the map has degree $6$ and the image has dimension $d$.  
\item For $g=4$ and $a \in \C \smallsetminus \{\pm 1\}$, we have $\Aut \calC_a \simeq Q_8$ except for \[ 
a \in \{0,\pm 3, \pm\sqrt{-3}\}
\]
among the remaining values, the curves $\calC_a$ and $\calC_{a'}$ are isomorphic if and only if $a' \in \{\pm a, \pm(a+3)/(a-1), \pm(a-3)/(a+1) \}$.  
\end{enumalph}
\end{lemma}

For the proof of \Cref{thm:MainHigherGenus}, we will only need the weaker statement in (b) that the map is finite-to-one, and we will not make use of part (c)---but we find the additional explicitness to be quite agreeable.

\begin{proof}
Moduli of hyperelliptic curves (in terms of binary forms) are well understood, and automorphism groups of hyperelliptic curves have been completely and explicitly classified (for characteristic zero, see Brandt--Stichtenoth \cite{BrandtStichtenoth} for the determination and Shaska \cite{Shaska} for explicit moduli).  From these results, one can deduce the lemma---we take a direct, self-contained approach.  

First, part (a).  There is a natural map
\begin{equation}
\Aut C^{\textup{al}} \to \Aut(C^{\textup{al}}/\langle \iota \rangle) \simeq \Aut(\PP_{\calK^{\textup{al}}}^1) \simeq \PGL_2(\calK^{\textup{al}})
\end{equation}
with kernel $\langle \iota \rangle$; let $G$ be its image, the reduced, geometric, generic automorphism group of $C$.  Since any (geometric) automorphism of $C$ stabilizes the set of Weierstrass points, any $\gamma \in G$ permutes the set $S \cup \{\infty\}$ of their $x$-coordinates in \eqref{eqn:weierstrasspoints}.  

For the case $g=4$, we verify (using a simple loop in \textsf{Magma} \cite{Magma}) that the only linear fractional transformations that stabilize the set $S^{(4)} \colonequals \{0,\infty,\pm 1, \pm i, \pm b_1, \pm 1/b_1\}$ are the maps $\gamma(x)=\pm x, \pm 1/x$ arising from the image $V_4 \simeq C_2^2$ of $Q_8$.  More precisely, we check that for any three distinct $z_1,z_2,z_3 \in S^{(4)}$, the unique linear transformation
\begin{equation}  \label{eqn:gammaz}
\gamma(x) = \left(\frac{z_2-z_3}{z_2-z_1}\right)\frac{x-z_1}{x-z_3} 
\end{equation}
with $\gamma(z_1)=0$, $\gamma(z_2)=1$, and $\gamma(z_3)=\infty$, has $\gamma(S^{(4)}) \subseteq S^{(4)}$ if and only if $\gamma(x)=\pm x, \pm 1/x$ if and only if $\gamma$ stabilizes $\{0,\pm 1,\infty\}$.

The general case $g \geq 4$ follows by specialization, as follows.  Given $\gamma$ with $\gamma(S)=S$, we have $\gamma(z_1)=0$, $\gamma(z_2)=1$, and $\gamma(z_3)=\infty$ with $z_1,z_2,z_3 \in S$ distinct.  Assume for purposes of contradiction that $\gamma(x) \neq \pm x, \pm 1/x$.  Then the set $\{0,\pm 1,\infty\}$ is not stabilized by $\gamma$.  We specialize as follows:
\begin{itemize}
\item If $z_1 \in S^{(4)}$, we continue to the next step; otherwise $z_1= \pm b_j, \pm 1/b_j$ for some $j>1$, and we specialize $b_j=\pm b_1,\pm 1/b_1$ not equal to $z_2$ or $z_3$.  
\item If the specialization of $z_2$ lies in $S^{(4)}$, we continue; otherwise, $z_2= \pm b_j, \pm 1/b_j$ with $j>1$ and we specialize $b_j$ to an element of $\{\pm b_1, \pm 1/b_1\}$ not equal to the specialization of $z_1$ or $z_3$. 
\item We repeat the previous step with $z_3$: if its specialization lies in $S^{(4)}$ we continue, else we specialize $b_j$ to $\{\pm b_1, \pm 1/b_1\}$ not the specialization of $z_1$ or $z_2$.  
\item We specialize the remaining variables to $b_1$.
\end{itemize}
This yields a specialization homomorphism $\Q[b_1,\dots,b_d,b_1^{-1},\dots,b_d^{-1}] \to \Q[b_1,b_1^{-1}]$, restricting to a map $S \to S^{(4)}$, such that $\{z_1,z_2,z_3\}$ restricts to a subset which is not contained in $\{0,\pm 1,\infty\}$.  Specializing $\gamma$, we obtain a linear fractional transformation which stabilizes $S^{(4)}$, a contradiction.  

For (b), from part (a) there exists a nonempty Zariski open subset $\AAp' \subseteq \AAp$ such that the family of nice curves $C_{U'}$ has automorphism group $Q_8$ over $U' \times_{\Q} \Spec \Q(i)$.  We examine the fibres of the map $U' \to M_g$.  For a $\Q(i)$-algebra $R$, let $a,a' \in U'(R)$ be such that $\varphi \colon \calC_a \simeq \calC_{a'}$.  Then $\varphi$ commutes with the hyperelliptic involution, so induces an isomorphism $\overline{\varphi} \colon \PP^1 \to \PP^1$ mapping the branch locus of $\calC_a$ to that of $\calC_{a'}$.  But as we already observed in \eqref{eqn:stabilizersQ8}, the points with $x$-coordinates $\{0,\pm 1,\infty,\pm i\}$ have stabilizers of order $4$, whereas the other points have stabilizer of order $2$.  Thus $\overline{\varphi}$ permutes $\{0,\pm 1,\infty,\pm i\}$.  The subgroup of linear fractional transformations with this property form a group isomorphic to $S_4$ (naturally a subgroup of $S_6$) defined over $\Q(i)$ generated by $x \mapsto ix,1/x,(x-1)/(x+1)$.  The stabilizer of $\calC_a$ is $Q_8/\langle \iota \rangle \simeq V_4 \trianglelefteq S_4$.  So the map $U' \times_\Q \Spec \Q(i) \to M_g$ is Galois with Galois group $S_4/V_4 \simeq S_3$. 

Finally, for (c) we just carry out parts (a) and (b) explicitly.  Let $b=b_1$ and 
\[ S_b \colonequals \{0,\infty,\pm 1,\pm i, \pm b, \pm b^{-1}\} \]
with $b(b^4-1) \neq 0$.  Every linear fractional transformation permuting $S_b$ occurs among the transformations \eqref{eqn:gammaz} associated to ordered triples of distinct elements of $S_b$.  Testing them, we find that a transformation outside $V_4=\{x,-x,1/x,-1/x\}$ permutes $S_b$ if and only if $P(b)=0$ where 
\begin{equation} 
\begin{aligned}
P(x) &\colonequals x^{24} - 20x^{20} - 475x^{16} + 475x^8 + 20x^4 - 1 \\
&= (x^4-1)(x^4+1)(x^2-2x-1)(x^2+2x-1)(x^4+6x^2+1) \\
&\qquad\qquad \cdot (x^4-2x^3+2x^2+2x+1)(x^4+2x^3+2x^2-2x+1).
\end{aligned}
\end{equation}
The values with $b^4-1=0$ are excluded parameters.  By \Cref{ex:transformation a b}, we have $a=-(b^2+b^{-2})/2$ so by substitution we get 
\[ \frac{P(b)}{b^4-1}=-32b^{10}a(a^2-9)(a^2+3). \]
Thus on the smooth locus $a \ne \pm 1$, the exceptional parameters are exactly $0,\pm 3, \pm \sqrt{-3}$.  Every transformation preserving the branch set lifts to an automorphism of $\calC_a$ over $\C$ so this proves the assertion about the automorphism group.  

For nonexceptional parameters $a,a'$, an isomorphism $\mathcal{C}_a \simeq \mathcal{C}_{a'}$ induces a linear fractional transformation normalizing their full reduced automorphism group $V_4$.  As in (b), it therefore belongs to the group $S_4$ preserving $\{0,\infty,\pm 1,\pm i\}$ and its action on the parameter factors through $S_4/V_4 \simeq S_3$.  With $f_a(x) = x(x^4-1)(x^4+2ax^2+1)$ as above, we have 
\begin{equation}
\begin{aligned}
 f_a(ix) &= i f_{-a}(x)  \\
 (x+1)^{10}f_a \left( \frac{x-1}{x+1}\right) &= -16(a+1)f_{(3-a)/(a+1)}(x) 
 \end{aligned}
 \end{equation}
 Consequently, the two generators represented by $x \mapsto ix$ and $x\mapsto (x-1)/(x+1)$ induce maps $a \mapsto -a$ and $a \mapsto (3-a)/(a+1)$, respectively, giving the indicated orbit.  And conversely, these identities followed by rescaling the $y$-coordinate give isomorphisms.
\end{proof}

\section{Genus four}\label{sect:g4}

In this section, we mostly study a special case: the case $g=4$ of the family in~\Cref{setup} specialized to $a=1/2$.  (In section \ref{sec:familyg4}, we consider some consequences for other fibres in the family in genus $g=4$.)  Throughout this section, we abbreviate
\[ C \colonequals \calC_{1/2}^{(4)} : y^2 = x(x^4-1)(x^4+x^2+1) \]
and let $A \colonequals \calA_{1/2}^{(4)}=\Jac \calC_{1/2}^{(4)}$ be its Jacobian.  The discriminant of $C$ is $2^{28} 3^6$, so $C$ is nice of genus $4$ and $A$ has good reduction away from $2$ and $3$.  

\subsection{Endomorphism algebra}

Running the algorithm of Costa--Mascot--Sijsling--Voight \cite{costa-mascot-sijsling-voight-18}, we obtain the following theorem.  

\begin{theorem}\label{thm:EndomorphismsGenus4}
The Jacobian $A$ has geometric endomorphism algebra 
\[ (\End A^{\textup{al}})_\Q \simeq \quat{-1,-1}{\Q}; \]
in particular, $A$ is geometrically simple.  
\end{theorem}

\begin{proof}
For the reader's convenience, we elaborate a bit on the mechanism which underlies the algorithm \cite[\S 7]{costa-mascot-sijsling-voight-18}---see also Lombardo \cite{MR3882288}. 

We begin by computing the (geometric) endomorphism algebra of the reduction $A_{\F_p}$ of $A$ modulo a good prime $p$~\cite[Lemma 7.2.4]{costa-mascot-sijsling-voight-18}, as follows.
Let
\[ c_p(T) \colonequals \det(1- \Frob_p T\,|\, H^1_{\et}(A_{\F_p^{\textup{al}}},\Q_\ell)) \in 1 + T\Z[T] \]
be the reverse characteristic polynomial of the $p$-power Frobenius endomorphism $\Frob_p$ acting on \'etale cohomology (independent of $\ell \neq p$), computed as the numerator of the zeta function of $C_{\F_p}$ by counting points.  
Let $c_p^{\otimes 2}(T)$ be the reverse characteristic polynomial of Frobenius acting on $(H^1)^{\otimes 2}$.  
We factor
\begin{equation} \label{eqn:cphp}
c_p^{\otimes 2}(T) = h(pT) \prod_{i=1}^{n} \Phi_{k_i}(pT) 
\end{equation}
where $\Phi_{k_i}$ is the $k_i$-th cyclotomic polynomial and no root of $h(T)$ is a root of unity. Then
\begin{equation}
\dim_{\Q} \End(A_{\F_p^r})_\Q = \sum_{k_i \mid r} \deg \Phi_{k_i},
\end{equation}
and the endomorphism field of $A_{\F_p}$ is $\F_{p^k}$ where $k$ is the least common multiple of the $k_i$.  

Let $p=41$.  Computing using \textsf{Magma} \cite{Magma}, we find that $c_p(T)=g_p(T)^2$ where
\begin{equation}
g_p(T) = 1-2T-30T^2-2pT^3+p^2T^4
\end{equation}
and 
\begin{equation} \label{eqn:cp2t}
c_p^{\otimes 2}(T) = h(pT)(1-pT)^{16} 
\end{equation} 
with $h(T)$ as in \eqref{eqn:cphp}.  
Thus all endomorphisms of $A_{\F_{p}}^{\al}$ are already defined over $\F_p$ and $\dim_\Q \End(A_{\F_p})_\Q = 16$.
(By Honda--Tate theory \cite[Remark 7.2.13]{costa-mascot-sijsling-voight-18}, since $c_p(T)$ is the square of an irreducible polynomial $g_p(T)$ with middle coefficient not divisible by $p$, we conclude that $A_{\F_p}$ is isogenous (over $\F_p$) to the square of a geometrically simple, ordinary abelian surface over $\F_p$ with endomorphism algebra $\Q(\pi)=\Q[T]/(g_p(T))$.)

We now match certain lower and upper bounds.  Since $A^{\al}$ has quaternionic multiplication \Cref{eq: fundamental automorphisms}, we have
\[ \eta(A^{\al}) \geq 2\cdot 1^2 \cdot 4 = 8 \]
\cite[(7.3.16)]{costa-mascot-sijsling-voight-18} and
\[ \eta(A_{\F_p^{\al}}) = 2^2 \cdot 4 = 16 = 2\eta(A^{\al}) \]
\cite[(7.3.18)]{costa-mascot-sijsling-voight-18}.  Therefore \cite[Corollary 7.3.19(b)--(c)]{costa-mascot-sijsling-voight-18} (from $(e_1n_1,n_1\dim A_1)=(2,4)$) we conclude that either $A^{\textup{al}}$ is simple (with geometric endomorphism algebra a quaternion algebra over $\Q$) or isogenous to the square of a simple abelian surface whose endomorphism algebra is commutative (a field).  (This refines the conclusion that $A^{\al}$ is simple or isogenous to the product of two abelian surfaces, which follows from the injectivity of endomorphisms under reduction modulo $p$.)

To rule out the latter, we first compute the center of $B=\End(A^{\al})_\Q$ \cite[\S 7.4]{costa-mascot-sijsling-voight-18}.  We repeat the calculation above for $p=73$, and we get 
\[ c_{73}(T) = g_{73}(T)^2=(1+8T-2T^2+8p T^3 + p^2 T^4)^2.  \]  
We compute that the only subfield common to $\Q[T]/(g_{41}(T))$ and $\Q[T]/(g_{73}(T))$ is $\Q$, so $Z(B)=\Q$ \cite[Corollary 7.4.4]{costa-mascot-sijsling-voight-18}.  (It is no coincidence that looking at two primes is enough~\cite{costa-lombardo-voight-19}.) 
So we cannot have $A^{\al}$ isogenous to the square of an abelian surface: for then its endomorphism algebra would be $B=\M_2(L)$ where $L=Z(B)=\Q$, but we have an embedding $(-1,-1\,|\,\Q) \hookrightarrow B$, a contradiction.  
\end{proof}

\begin{corollary}\label{cor:FieldsOfDefinition}
The Jacobian $A$ has endomorphism field $\Q(\End A)=\Q(i)$ and connected monodromy field $\Q(\varepsilon_A)=\Q(\zeta_8)$.
\end{corollary}

\begin{proof}
By \Cref{thm:EndomorphismsGenus4}, the endomorphism algebra of $A$ is generated by the automorphisms of $C$, which by inspection are minimally defined over $\Q(i)$, so $\Q(\End A)=\Q(i)$.  

By~\eqref{eqn:KendKconn} we have $\Q(i) \subseteq \Q(\varepsilon_A)$ and by 
 \Cref{thm:UpperBoundConnectedField} we have $[\Q(\varepsilon_A):\Q(i)] \leq 2$. We first show that equality holds. Let $p=13$ and let $\mathfrak{p}$ be one of the two primes of $\Q(i)$ lying over $p$.  We compute 
\[ \det(1 - \rho_{A,\ell}(\Frob_\frakp) T) = \det(1-\rho_{A,\ell}(\Frob_p)T) = g_p(T)^2 \]
where 
\[ g_p(T) = 1-2T+2pT^3-p^2T^4 \] 
so recalling \Cref{eqn:omegaell} we get $\omega(\Frob_\frakp)=-1$ and then by \Cref{lem:isinGOplus} we have $\rho_{A,\ell}(\Frob_\frakp) \not\in G_{A,\ell}^0$ (for any $\ell$), so $[\Q(\varepsilon_A):\Q(i)]=2$.  

To determine this quadratic extension, we use Lemma \ref{lem:unramprimes}: $\omega$ can only be ramified at places of bad reduction of $A$, hence it is ramified at most at the places of $\Q(i)$ lying over $2$ and $3$. This means that $\Q(\varepsilon_A)=\Q(i)(\sqrt{\beta})$ with $\beta \in \langle i, 1+i, 3\rangle$, the multiplicative subgroup of $\Q(i)^\times/\Q(i)^{\times 2}$ generated by the units and the (possibly) ramified primes.  The condition that $\Q(\varepsilon_A)$ is Galois over $\Q$ implies that 
$\beta \in \langle i, 3\rangle$, so the field $\Q(\varepsilon_A)$ is one of $L_1=\Q(\sqrt{i})=\Q(\zeta_8)$, $L_2=\Q(i, \sqrt{3})$, or $L_3=\Q(i, \sqrt{3i})$.  For $p=41$ we compute $g_{41}(T) = 1-2T-30T^2-2p T^3 + p^2 T^4$, so $\omega(\Frob_{41})=+1$; but the characters of $L_2$ and $L_3$ over $\Q(i)$ both take the value $-1$ (primes are inert), so $\Q(\varepsilon_A)=L_1=\Q(\zeta_8)$.
\end{proof}

\Cref{cor:FieldsOfDefinition} shows that \Cref{thm:UpperBoundConnectedField} is sharp, in general.  

\begin{remark}
The proof of \Cref{thm:EndomorphismsGenus4} ultimately relies on explicit calculations at the primes $41$ and $73$.  Accordingly, for all $a \in \Z$ such that $2a \equiv 1 \pmod{41\cdot 73}$, the same conclusion holds for the Jacobian $\calJ_a$, already giving not just one example but infinitely many.  (For \Cref{cor:FieldsOfDefinition} we also use $p=13$, but we also need to restrict the primes of bad reduction.)

These primes might look comparatively large, but in fact they are among the smallest choices that are compatible with what we now know about $A$: we need primes that are totally split in $\Q(\varepsilon_A)=\Q(\zeta_8)$ (congruent to $1$ modulo $8$, the smallest are $17,41,73$).  For \Cref{cor:FieldsOfDefinition}, we need primes that split in $\Q(\End A)=\Q(i)$ but not $\Q(\varepsilon_A)$ (so congruent to $5$ modulo $8$, the smallest are $5,13$).  
\end{remark}

Before moving on, we pause briefly to obtain a similar result for $g=6$.  

\begin{proposition}\label{prop:Genus6}
For $g=6$ and $a=(1/2,0)$, we have $(\End \calJ_a^{\al})_\Q \simeq (-1,-1\,|\,\Q)$ and $\Q(\End \calJ_a)=\Q(i)$.  
\end{proposition}

\begin{proof}
We perform entirely analogous computations: in this case for $p=17$ we get
\[ g_{17}(T) = 1-2T-13T^2+44T^3-13pT^4-2p^2T^5+p^3T^6 \]
and for $p=41$ we get
\[ g_{41}(T) = 1-14T+91T^2-540T^3+91pT^4-14p^2T^5+p^3T^6 \]
both with factorizations like \eqref{eqn:cp2t}, and no common subfield.
\end{proof}

\subsection{Homology}\label{sec:HomologyBasis}

In this section, we describe the integral homology of $C$ over $\C$.  Let
\[
S=\{
0, 1, \zeta_6, i, \zeta_3, -1, -\zeta_6, -i, -\zeta_3, \infty
\} \subseteq \mathbb{P}^1(\mathbb{C}).
\]
The coordinate function $x \colon C_\C \to \mathbb{P}^1(\mathbb{C})$ realises $C_\C$ as a $2$-to-$1$ cover of $\mathbb{P}^1(\mathbb{C})$, branched over the points in $S$.
We consider paths $\delta_0, \ldots, \delta_4$ in the complex plane $\mathbb{C}$, thought of as the $x$-plane, as described in~\Cref{fig:Homology}. 

  \begin{minipage}{\linewidth}
      \centering
      \begin{minipage}{0.45\linewidth}
\begin{figure}[H]
\centering
\begin{tikzpicture}[scale=3.2,cap=round,>=latex,
        hatch distance/.store in=\hatchdistance,
        hatch distance=10pt,
        hatch thickness/.store in=\hatchthickness,
        hatch thickness=2pt
    ]
    \makeatletter
    \pgfdeclarepatternformonly[\hatchdistance,\hatchthickness]{flexible hatch}
    {\pgfqpoint{0pt}{0pt}}
    {\pgfqpoint{\hatchdistance}{\hatchdistance}}
    {\pgfpoint{\hatchdistance-1pt}{\hatchdistance-1pt}}%
    {
        \pgfsetcolor{\tikz@pattern@color}
        \pgfsetlinewidth{\hatchthickness}
        \pgfpathmoveto{\pgfqpoint{0pt}{0pt}}
        \pgfpathlineto{\pgfqpoint{\hatchdistance}{\hatchdistance}}
        \pgfusepath{stroke}
    } 
 
\begin{scope}
\clip (0,0) to (1cm,0cm) to (30:1.2cm) to (60:1cm) to (0,0);
\draw[pattern=flexible hatch, pattern color=red!25, hatch distance=10pt] (0,0) circle(1cm);
\end{scope}

\draw (30:0.6cm) node {\contour{white}{$A_1$}};
\draw (75:0.6cm) node {$A_2$};
\draw (105:0.6cm) node {$A_3$};
\draw (150:0.6cm) node {$A_4$};

        \draw[thick] (0cm,0cm) circle(1cm);

        \foreach \x in {0,60,90,120,180,240,270,300,360} {
                \draw[black] (0cm,0cm) -- (\x:1cm);
                \filldraw[black] (\x:1cm) circle(0.4pt);
        }

        \foreach \x/\xtext in {
            60/\zeta_6,
            90/i,
            120/\zeta_3,
            180/-1,
            240/-\zeta_6,
            270/-i,
            300/-\zeta_3,
            360/1}
                \draw (\x:1.1cm) node {$\xtext$};

\node[regular polygon, regular polygon sides=3, inner sep=1.25pt, fill=black, shape border rotate=30, label=right:$\delta_1$] at (30:1cm) {};

\node[regular polygon, regular polygon sides=3, inner sep=1.25pt, fill=black, shape border rotate=75, label=above:$\delta_2$] at (75:1cm) {};

\node[regular polygon, regular polygon sides=3, inner sep=1.25pt, fill=black, shape border rotate=105, label=above:$\delta_3$] at (105:1cm) {};

\node[regular polygon, regular polygon sides=3, inner sep=1.25pt, fill=black, shape border rotate=150, label=left:$\delta_4$] at (150:1cm) {};

\node[regular polygon, regular polygon sides=3, inner sep=1.25pt, fill=black, shape border rotate=270, label=below:$\delta_0$] at (0:0.5cm) {};

\node[regular polygon, regular polygon sides=3, inner sep=1.25pt, fill=black, shape border rotate=90, label=below:$\alpha\delta_0$] at (180:0.5cm) {};

\node[regular polygon, regular polygon sides=3, inner sep=1.25pt, fill=black, shape border rotate=210, label=left:$\alpha\delta_1$] at (210:1cm) {};

\node[regular polygon, regular polygon sides=3, inner sep=1.25pt, fill=black, shape border rotate=255, label=below:$\alpha\delta_2$] at (255:1cm) {};

\node[regular polygon, regular polygon sides=3, inner sep=1.25pt, fill=black, shape border rotate=285, label=below:$\alpha\delta_3$] at (285:1cm) {};

\node[regular polygon, regular polygon sides=3, inner sep=1.25pt, fill=black, shape border rotate=330, label=right:$\alpha\delta_4$] at (330:1cm) {};

\draw (60:0.5cm) node {\contour{white}{$\varphi_{\zeta_6}$}};
\draw (90:0.5cm) node {\contour{white}{$\varphi_{i}$}};
\draw (120:0.5cm) node {\contour{white}{$\varphi_{\zeta_3}$}};

\node[regular polygon, regular polygon sides=3, inner sep=1.25pt, fill=black, shape border rotate=330] at (60:0.7cm) {};
\node[regular polygon, regular polygon sides=3, inner sep=1.25pt, fill=black, shape border rotate=0] at (90:0.7cm) {};
\node[regular polygon, regular polygon sides=3, inner sep=1.25pt, fill=black, shape border rotate=30] at (120:0.7cm) {};
\end{tikzpicture}
\caption{\label{fig:Homology} Image of the homology basis on the $x$-plane}
\end{figure}
      \end{minipage}
      \hspace{0.05\linewidth}
      \begin{minipage}{0.45\linewidth}
\begin{figure}[H]
\centering
\begin{tikzpicture}[scale=0.5]
\foreach \i in {0,...,3}
\draw[red, very thick] (3*\i,0) ellipse (1 and 0.25);	
\draw[thick] (12,0) ellipse (1 and 0.25);	
\begin{scope}
\clip (-1,0.3) rectangle (13,0.6);
\foreach \i in {0,...,4}
\draw[fill=white, dashed, thick] (3*\i,0.3) ellipse (1 and 0.25);
\end{scope}
\begin{scope}
\clip (-1,0.3) rectangle (13,0);
\foreach \i in {0,...,4}
\draw[fill=white, thick] (3*\i,0.3) ellipse (1 and 0.25);
\end{scope}
\foreach \i in {0,...,3}{
\draw[bend left=90, looseness=3, blue, very thick] (3*\i+1,0.3) to (3*\i+2,0.3);
\draw[bend right=90, looseness=3, blue, very thick] (3*\i+1,0) to (3*\i+2,0);}
\draw[in=90, out=90, looseness=.9, thick] (-1,0.3) to (13,0.3);
\draw[in=-90, out=-90, looseness=.9, thick] (-1,0) to (13,0);

\draw[red] node at (0,-0.7) {$\gamma_3$};
\draw[red] node at (3,-0.7) {$\gamma_1$};
\draw[red] node at (6,-0.7) {$\alpha\gamma_0$};
\draw[red] node at (9,-0.7) {$\alpha\gamma_2$};

\draw[blue] node at (1.5,1.55) {$\gamma_2$};
\draw[blue] node at (4.5,1.55) {$\gamma_0$};
\draw[blue] node at (7.5,1.55) {$\alpha\gamma_1$};
\draw[blue] node at (10.5,1.55) {$\alpha\gamma_3$};

\end{tikzpicture}
\caption{Homology basis for $C$}\label{fig:HomologyX}
          \end{figure}
      \end{minipage}
  \end{minipage}

The inverse image in $C_\C$ of each $\delta_i$ is homeomorphic to the union of two closed intervals with common endpoints, so it is a copy of $\mathbb{S}^1$. We parametrize each of these copies of $\mathbb{S}^1$ with functions $\gamma_i : [0,1]\to C$. To fix the orientation, we require that $y \circ \gamma_i(t)$ has positive imaginary part for $t \in [0,1]$ sufficiently close to $0$; this uniquely determines the paths $\gamma_i$.
Notice that $\iota \circ \gamma_i$ is the path $-\gamma_i$, that is, $\gamma_i$ with the opposite orientation. 
Recall the automorphism $\alpha$ from~\cref{eq: fundamental automorphisms}.
By standard results on the homology of real orientable surfaces, the eight paths
\begin{equation}\label{eq:basis}
  \gamma_0, \alpha \gamma_0, \gamma_1, \alpha \gamma_1, \gamma_2, \alpha \gamma_2, \gamma_3, \alpha \gamma_3  
\end{equation}
form a basis of the integral homology $H_1(C_\C,\mathbb{Z})$, see also~\Cref{fig:HomologyX}.
The action of $\alpha$ in terms of this basis is immediate to write down. Together with a straightforward calculation for $\beta$, this yields the matrices representing the action of $\alpha, \beta$ on homology:
\begin{equation}\label{eq:alpha*beta*}
\alpha_* = \begin{pmatrix}
0 & -1 \\
1 & 0 \\
&& 0 & -1 \\
&& 1 & 0 \\
&&&& 0 & -1 \\
&&&& 1 & 0 \\
&&&&&& 0 & -1 \\
&&&&&& 1 & 0 \\
\end{pmatrix}, \quad \beta_* = \begin{pmatrix}
 1 & 0 & -1 & -1 &  \\
 0 & -1 & -1 & 1 & \\
 1 & 1 & 0 & -1 & \\
 1 & -1 & -1 & 0 & \\
 0 & 0 & 0 & -1 & 0 & 0 & 0 & 1 \\
 0 & 0 & -1 & 0 & 0 & 0 & 1 & 0 \\
 1 & 1 & 0 & -1 & 0 & -1 & 0 & 0 \\
 1 & -1 & -1 & 0 & -1 & 0 & 0 & 0 \\
\end{pmatrix}.
\end{equation}
Finally, the Gram matrix of the intersection form is given by
\begin{equation}\label{eq:IntersectionForm}
\begin{pmatrix}
 0 & -1 & -1 & 0 & 0 & 0 & 0 & 0 \\
 1 & 0 & 0 & -1 & 0 & 0 & 0 & 0 \\
 1 & 0 & 0 & 0 & -1 & 0 & 0 & 0 \\
 0 & 1 & 0 & 0 & 0 & -1 & 0 & 0 \\
 0 & 0 & 1 & 0 & 0 & 0 & -1 & 0 \\
 0 & 0 & 0 & 1 & 0 & 0 & 0 & -1 \\
 0 & 0 & 0 & 0 & 1 & 0 & 0 & 0 \\
 0 & 0 & 0 & 0 & 0 & 1 & 0 & 0 \\
\end{pmatrix};
\end{equation}
apart from the signs, the structure of this matrix can easily be gleaned from~\Cref{fig:HomologyX}.

\subsection{Complex lattice}
\label{sec: complex structure}

We now discuss the period lattice $\Lambda$ attached to $C$.
Let $c_1, \ldots, c_8$ be the (ordered) basis of the integral homology of $C$ in~\cref{eq:basis}.
The period matrix of the abelian variety $A=\Jac(C)$ (over $\C$) with respect to this homology basis and to the basis $\{ x^{\ellindex-1}\, \mathrm{d}x/y \}_{\ellindex=1,\ldots,4}$ for the regular differentials is by definition the $4 \times 8$ matrix 
\begin{equation} \label{eqn:zupzup}
\Pi=\left( \Pi_{\ellindex j} \right)_{\substack{\ellindex=1,\ldots,4 \\ j=1,\ldots,8}}
=\left(\int_{c_j} \frac{x^{\ellindex-1}\,\mathrm{d}x}{y}\right)_{\substack{\ellindex=1,\ldots,4 \\ j=1,\ldots,8}}
\end{equation}
The columns of $\Pi$ span a lattice $\Lambda \simeq \Z^8$ in $\mathbb{C}^4$, and we have an isomorphism $\mathbb{C}^4/\Lambda \simeq A(\C)$ of complex abelian varieties. Clearly the lattice $\Lambda$ carries an action of the Lipschitz order $\mathcal{O}$: in fact, we will show that equality holds (and more) below.

This complex description comes with two representations of the endomorphism ring.  First is the \emph{analytic} (or \emph{tangent}) \emph{representation} $M(\varphi)$ obtained by letting $\varphi \in \End A^{\al}$ act on the above basis for the space of regular differentials.  With respect to the natural embedding $\Aut C^{\al} \hookrightarrow \End A^{\al}$,
\begin{equation}\label{eqn:mimj}
M(\alpha) \colonequals i \begin{pmatrix}
1 \\
& -1 \\
&& 1 \\
&&& -1
\end{pmatrix}
\text{ and }
M(\beta) \colonequals i \begin{pmatrix}
&&& 1 \\
&& 1 \\
& 1 \\
1
\end{pmatrix}.
\end{equation}

The \emph{rational representation} of an endomorphism $\varphi \in \End(A^\al)$ is the unique matrix $R(\varphi) \in \operatorname{Mat}_{8 \times 8}(\mathbb{Z})$ such that
\[
M(\varphi) \Pi = \Pi R(\varphi),
\]
Given the definition of the period matrix, it is easy to see that for an endomorphism of $A$ induced by an automorphism $\varphi$ of $C$ the matrix $R(\varphi)$ coincides with the matrix of $\varphi_*$ acting on $H_1(C, \mathbb{Z})$ (with respect to the homology basis used to construct the period matrix). In particular, since we are using the basis of integral homology given in~\cref{eq:basis}, the matrices $R(\alpha)$ and $R(\beta)$ coincide with those given in~\cref{eq:alpha*beta*}.

\begin{proposition} \label{prop:aQendnotfree}
We have $\End A^{\al} = \mathcal{O}$ and the $\mathcal{O}$-module $\Lambda$ is not free.
\end{proposition}

\begin{proof}
For the first statement (generalized in \Cref{lemma:FromEndoAlgebraToEndoRing}), recall that the unique maximal order containing the Lipschitz order is the Hurwitz order, generated by the element $\omega=(-1+i+j+ij)/2$ over $\mathcal{O}$.  But from the rational representation \Cref{eq:alpha*beta*}, we see that $R(\omega) \not \in \M_8(\Z)$, a contradiction.  

For the second statement, in fact we will show that $\Lambda_2 \colonequals \Lambda \otimes_{\Z} \Z_2 \simeq \Z_2^8$ is not free over $\calO_2 \simeq \Z_2^4$, so $\Lambda$ is not even locally free!  The Jacobson radical 
(see Voight \cite[\S 11.1, \S\S 20.4--20.5, \S 24.3]{Voight}) 
of $\calO_2$ is $J_2 \colonequals \operatorname{rad}(\calO_2)=(1+i,1+j,1+k)$. Suppose by contradiction that $\Lambda_2$ is free over $\mathcal{O}_2$. By Nakayama's lemma, the $\mathcal{O}_2$-rank of $\Lambda_2$ is then equal to the dimension of $\Lambda_2/J_2 \Lambda_2$ over $\calO_2/J_2 \simeq \F_2$.  Since $J_2/2\calO_2$ is spanned by $1+i,1+j,1+k$ as an $\F_2$-vector space, $J_2 \Lambda_2/2 \Lambda_2$ is spanned over $\F_2$ by the columns of the matrices $\operatorname{Id}_8+R(i), \operatorname{Id}_8+R(j), \operatorname{Id}_8+R(k)$, where $\operatorname{Id}_8$ denotes the $8\times 8$ identity matrix.  
Given the explicit form of the matrices $R(i)$ and $R(j)$ given in~\cref{eq:alpha*beta*}, we compute that $\dim_{\F_2}(J_2\Lambda_2/2\Lambda_2)=5$ so 
\[ \dim_{\F_2}(\Lambda_2/J_2\Lambda_2)=\dim_{\F_2}(\Lambda_2/2\Lambda_2) - \dim_{\F_2}(J_2\Lambda_2/2\Lambda_2)=8-5=3. \]  
So if $\Lambda_2$ was free over $\calO_2$, it would be free of rank $3$, so $\Lambda_2 \simeq \calO_2^3 \simeq \Z_2^{12}$, a contradiction. 
\end{proof}

\begin{remark} \label{rmk:lackoffree}
The lack of freeness of $\Lambda$ may seem undesirable, but we cannot have it all: if $\Lambda$ is free of rank $2$ over $\mathcal{O}$ and equipped with a natural principal polarization, then necessarily $\C^4/\Lambda$ is isomorphic to the square of an abelian surface \cite[Example 9.5.6]{MR2062673}.  For more, see Birkenhake--Lange \cite[Exercise 9.10(2)]{MR2062673} (and more generally Shimura \cite{Shimura-fam}) for a criterion which determines when an abelian variety of Type III is geometrically simple or not.
\end{remark}

\subsection{Conclusions about the family} \label{sec:familyg4}

Although we will pursue this more generally in the next sections, we pause to conclude with a few remarks about our family in genus $g=4$.  We have
\begin{equation}  \label{eqn:x9x7x3}
\calC_a^{(4)} \colon y^2 = f(x) = x(x^4-1)(x^4+2ax^2+1).
\end{equation}
The discriminant of $f(x)$ is $\Delta=-2^{24}(a-1)^6(a+1)^6$, so we obtain a nice family of genus $4$ curves over $\A^1 \smallsetminus \{\pm 1\}$.  

Some of the results we proved for a specialization can be extended as follows.
\begin{itemize}
\item Since endomorphism rings inject under specialization, from~\Cref{thm:EndomorphismsGenus4} we conclude that the generic geometric endomorphism algebra of $\calJ^{(4)}$ is $(-1,-1\,|\,\Q)$ and accordingly the generic endomorphism field is $\Q(i)$.  
\item Since the rational representations~\cref{eq:alpha*beta*} and intersection form~\Cref{eq:IntersectionForm} are a discrete invariant, they are constant over the continuous parameter $a \in \C \smallsetminus \{\pm 1\}$.  Since we kept a fixed basis of holomorphic differentials, the analytic representations \Cref{eqn:mimj} are also constant.
\item The conclusion of~\Cref{prop:aQendnotfree} then holds for all $a \in \C$ such that $(\End \calJ_a^{\al})_\Q=(-1,-1\,|\,\Q)$, in particular it holds generically.
\end{itemize}

Another interesting feature (not used in the sequel) of the family in genus $g=4$ is the following.

\begin{lemma} \label{lem:puregood}
Let $K$ be a number field and let $a \in K \smallsetminus \{\pm 1\}$.  Then the abelian variety $\calA_{a}$ over $K$ has everywhere potentially good reduction away from the primes of $\Z_K$ above $2$.
\end{lemma}

\begin{proof}
We follow the proof of Oort \cite[Proof of Lemma (3.9)]{Oort}.  Let $\frakp$ be a nonzero prime of $\Z_K$ and let $K_\frakp$ be the completion of $K$ at $\frakp$.  Then  there is an extension $L_\frakP \supseteq K_\frakp$ such that $(\calA_a)_{L_\frakP}$ has semistable reduction.  Let $\mu \leq 4$ be the dimension of the toric part.  If $\mu=0$, then $\calA_a$ obtains good reduction over $L_\frakP$ as desired; so suppose $\mu>0$.  Then over $L_\frakP(i)$ we have an embedding $B=(-1,-1\,|\,\Q) \hookrightarrow \M_\mu(\Q)$.  Since $B$ is a division algebra, we must have $\mu=4$, which is to say $\calA_a$ has purely toric reduction over $L_\frakP$.  We show that this is impossible when $\frakP$ does not lie above $2$, arguing directly with the equation for $\calC_a$: modulo $\frakP$, the subset $\{\pm b, \pm 1/b\}$ reduces to a subset which can only intersect with the set $\{0,\infty,\pm 1, \pm i\}$ in at most two points, so the reduction of this set of roots has at least $4$ distinct elements, so the normalization of the reduction of the curve $\calC_a$ has normalization of genus $g \geq 1$ and hence the Jacobian cannot have purely toric reduction.
\end{proof}

\begin{remark}
Neither the statement nor the proof of \Cref{lem:puregood} generalizes to $g \geq 6$: see \Cref{rmk:oomanybadred}.  

It seems likely that the proof of \Cref{lem:puregood} can be extended to primes above $2$ by using the methods of Fiore--Yelton \cite{fiore2022clusters}, combined with a combinatorial study of the possible cluster pictures of the hyperelliptic curves $\calC_a$ in residue characteristic $2$.  (Knowing that the toric rank is either $0$ or $4$ helps to reduce the combinatorial complexity.) 
\end{remark}

\begin{remark}
We sketch two other proofs of \Cref{lem:puregood}.

First, a less computational route to \Cref{lem:puregood} (at least for primes not above $2$) may be obtained by looking at the action of the inertia group on the $2$-adic Tate module: by Grothendieck \cite[Theorem 2.4]{Grothendieck SGA 7 I}, if the toric rank is $4$, then there is a free rank-$4$ submodule that is stable under the endomorphism ring, totally isotropic for the Weil pairing, and on which the inertia group acts trivially.  It is possible to verify explicitly that no such submodule exists.

A second conceptual proof, one which also applies to places of residue characteristic $2$, argues via moduli spaces.  In \Cref{sec: a Shimura variety}, we show that the Shimura varieties which parametrize principally polarized abelian fourfolds with an action of $\mathcal{O}$ compatible with the polarization, with respect to any level structure, are projective: more precisely, they possess a finite-degree correspondence with $\mathbb{P}^1$ (\Cref{thm: from the moduli space to the Neron model}).  These Shimura varieties correspond to the reductive group $G \colonequals \MT(\calA_a)$ for generic $a$.  By work of Borel--Harish-Chandra \cite[Theorem 12.3, Corollary 12.4]{Borel-Harish-Chandra}, the compactness of the associated Shimura varieties implies that the adjoint group $G^{\operatorname{ad}}$ is anisotropic over $\mathbb{Q}$, i.e., $G(\mathbb{Q})$ contains no unipotent element.  (We suspect that these facts could also be read off directly.)  The Mumford--Tate group of $\calA_a$ is contained in $G$, so it is also $\Q$-anisotropic.
The main result of Lee \cite{MoritaConjecture} then shows that $\calA_a$ has potentially good reduction everywhere.
\end{remark}

\section{Endomorphisms and classes}\label{sec: endosclasses}

In this section, we establish results on endomorphisms and algebraic classes on the Jacobians in our family, under certain hypotheses.  In the next section, we will combine these with an inductive argument to prove our main result,~\Cref{thm:MainHigherGenus}.  Throughout, we work with \Cref{setup}: let $K$ be a number field and let $a=(a_1,\dots,a_d) \in \AAp(K)$.

\subsection{Endomorphism ring and endomorphism field}

We get started by showing that properties (i) and (ii) in~\Cref{thm:MainHigherGenus} on the endomorphism ring are implied by the (formally weaker) statement concerning the endomorphism algebra.

\begin{lemma}\label{lemma:FromEndoAlgebraToEndoRing}
Suppose that the Jacobian $\calA_{\aunder}$ of $\calC_\aunder$ has geometric endomorphism \textit{algebra} $(-1,-1\,|\,\Q)$.  Then its geometric endomorphism \textit{ring} is $\End \calA_{\aunder}^{\al} = \mathcal{O}$ and its endomorphism field is $K(i)$.

More precisely, let $\alpha, \beta$ be the automorphisms of $\calC_\aunder$ given in~\cref{eq: fundamental automorphisms}, and let $\alpha^*, \beta^*$ be the corresponding automorphisms of $\calA_{\aunder}$.
Then the inclusion $\mathcal{O} \hookrightarrow \End \calA_{\aunder}^{\al}$ induced by $i,j \mapsto \alpha^*,\beta^*$ is an isomorphism.
\end{lemma}

\begin{proof}
We start with $\mathcal{O} \hookrightarrow \End \calA_{\aunder}^{\al} \subset (\End \calA_{\aunder}^{\al})_{\Q} = (-1,-1\,|\,\Q)$, and for simplicity we identify $i=\alpha^*$, $j=\beta^*$, and $k=ij=\alpha^* \beta^*$ with their images.  The endomorphism field is $K(i)$ since this is the minimal field of definition of $i,j$.

The Lipschitz order $\mathcal{O}$ is contained in a unique maximal order of index $2$, the Hurwitz order $\Z \oplus \Z i \oplus \Z j \oplus \Z \omega$, where $\omega=(-1+i+j+k)/2$.  So to prove the lemma, we need only show that $\omega \not\in \End \calJ_a^{\al}$; that is to say, that the element $2\omega=-1+i+j+k$ is not divisible by $2$ in $\End \calJ_a^{\al}$, which is equivalent to it acting nontrivially on the $2$-torsion $\calJ_a^{\al}[2]$.  To this end, we recall~\cref{subsect:BaseChanges}, in particular the calculation of the action of $i,j$ on $\calJ_a^{\al}[2]$ \cref{eq:ActionOfI}--\cref{eq:ActionOfJ}.  We then verify directly that $2\omega$ acts nontrivially, as desired.
\end{proof}

\subsection{Connected monodromy field: lower bounds}\label{sect:SatoTateField}

In this section we prove the following lower bound on the connected monodromy field.

\begin{theorem}\label{thm:LowerBoundSatoTate}
Suppose $i \in K$ and  $a_d \neq 0$.  Then 
$H^g_{\et}(\calA_\aunder^{\al}, \Q_\ell(g/2))$
contains algebraic classes whose field of definition is $K(\zeta_8)$. 
In particular, the connected monodromy field of $\calA_\aunder$ contains $K(\zeta_8)$.
\end{theorem}

\begin{example}\label{rmk: ExampleZywina}
\Cref{thm:LowerBoundSatoTate} justifies the computation of the connected monodromy field in Zywina \cite[Example 1.8]{Zywina2020}.  This example is in our family in genus $g=10$, with parameter $a=(7,1,7,1/2)$.  We recall \eqref{eqn:KendKconn} that $\Q(\End A)=\Q(i) \subseteq \Q(\varepsilon_A)$. By~\Cref{thm:UpperBoundConnectedField} over $K=\Q(i)$, the degree $[\Kconn : \KEnd]$ is at most 2. Now~\Cref{thm:LowerBoundSatoTate} implies that $\Kconn = \Q(\zeta_8)$.
\end{example}

The proof of~\Cref{thm:LowerBoundSatoTate} will occupy this section: we will construct some nontrivial cohomology classes in $H_{\et}^{g}((\calA_\aunder)^{\al}, \Q_\ell(g/2))$ (with $\ell$ an auxiliary prime) and show that they are minimally defined over $K(\zeta_8)$.  This implies~\Cref{thm:LowerBoundSatoTate}, with the final statement following from~\Cref{prop:LowerBoundConnectedField}.  We follow closely Schoen~\cite[Proof of Theorem 2.0]{MR930145}.

We begin by finding a (singular) model of $\calC_\aunder$ on which the action of an automorphism of order 4 takes a particularly simple form. 
Up to birational equivalence, we may take several models for $\calC_a$ of the form
\begin{equation}\label{eq:CaSingularGeneral}
z^4 = f_\ellindex(t) \colonequals t (t^2-1)^2 \left( t^{g-2}+1+ \sum_{i=1}^{d} a_i( t^{g-2-i} + t^{i} ) \right)^2 (t-\gamma_\ellindex)^{-2g-4},
\end{equation}
where $t=x^2$, $z=y \cdot (t-\gamma_\ellindex)^{-g/2-1}$ and the elements $\gamma_\ellindex \in K$, for $\ellindex=1, \ldots, g+1$, will be chosen below. Note that the field $K(t, z)$ coincides with $K(x,y)$, so~\cref{eq:CaSingularGeneral} is indeed a model of $\calC_a$. The automorphism $\alpha$ of $\calC_\aunder$ given in~\cref{eq: fundamental automorphisms} corresponds to
\begin{equation}\label{eq:Sigma}
\sigma \colon (t,z) \mapsto (t,iz).
\end{equation}
The quotient $\QuotC=\calC_\aunder / \langle \sigma \rangle$ is isomorphic to $\mathbb{P}^1$, with the quotient map $\psi : \calC_a \to \mathbb{P}^1$ given by $(x, y) \mapsto x^2$, which corresponds to $(t,z) \mapsto t$. In particular, $\psi$ has $g+2$ branch points, the roots of $f_\ellindex(t)$ together with $\infty$.  
We choose $\gamma_1, \ldots, \gamma_{g+1} \in K$ to be distinct and different from all of these branch points.
The topological Euler characteristic of $\mathbb{P}^1 \setminus \{ f_\ellindex=0\}$ is therefore $2-(g+2)=-g$, and so the invariant denoted $h$ by Schoen~\cite{MR930145} is equal to $g$ ~\cite[Lemma 1.5]{MR930145}.

\begin{remark} 
The choice of $\gamma_1, \ldots, \gamma_{g+1}$ ensures that the open subschemes 
$\operatorname{Sym}^{g}\left(\calC_a\setminus\psi^{-1}(\gamma_i)\right)$
cover all of $\operatorname{Sym}^{g} \calC_a$. 
Indeed, every geometric point of $\operatorname{Sym}^g(\calC_a)$ is a multiset of $g$ points of $\calC_a$,
so it meets at most $g$ of the $g+1$ pairwise disjoint fibres $\psi^{-1}(\gamma_s)$;
hence it avoids at least one such fibre.  See Milne \cite[Proposition 3.1]{MR861976}.
\end{remark}

Write $N \colonequals \{ (\sigma_1,\ldots,\sigma_g) \in (\mathbb{Z}/4\mathbb{Z})^g : \sum_i \sigma_i=0\}$ and let $G$ be the subgroup of $\operatorname{Aut}(\calC_\aunder^g)$ generated by $N$ and the group $S_g$ acting by permuting factors. In our situation,~\cite[Diagram (2.1)]{MR930145} reads
\begin{equation}\label{eq:SchoenDiagram}
\begin{aligned}
\xymatrix{
\calC_\aunder^g \ar[r] \ar[d] & \Sym^g \calC_\aunder \ar[d] \\
N \backslash \calC_\aunder^g  \ar[r] \ar[d] & W_0 \ar[d]^{\pi_0} \\
\QuotC^g \ar[r]_(0.4){\gamma_\QuotC} & \Sym^g \QuotC,
}
\end{aligned}
\end{equation}
where $W_0 \colonequals G \backslash \calC_\aunder^g$ and all arrows except $\Sym^g \calC_\aunder \to W_0$ are quotient maps for the obvious (finite) group actions. It will be useful to identify $\Sym^g \QuotC \simeq \Sym^g \mathbb{P}^1$ with the projectivization of the space of homogeneous polynomials of degree $g$ in two variables $X, Y$. In these coordinates, the horizontal map $\gamma_\QuotC$ becomes
\begin{equation}\label{eq:gammaX}
\begin{array}{cccc}
\gamma_\QuotC : & \left(\mathbb{P}^1\right)^g & \to & \Sym^g \mathbb{P}^1\\
& ([x_1 : y_1], \ldots, [x_g : y_g]) & \mapsto & [(y_1 X-x_1 Y) \cdots (y_g X-x_g Y)].
\end{array}
\end{equation}
We may consider $f_\ellindex(t)$ as a function on $\mathbb{P}^1$ and write its divisor as
\[
(f_\ellindex) = 2(1) + 2(-1) + (0) + 3(\infty) + 2 \left( \sum_{i=1}^d (\beta_i) + \sum_{i=1}^d (1/\beta_i) \right)  +4D_\ellindex,
\] 
where $D_\ellindex \colonequals -((g/2)+1)(\gamma_\ellindex)$. We order the points $b_1, \ldots, b_{g+2}$ that appear in the support of $(f_\ellindex)$ with positive coefficient as $1, -1, 0, \infty$, followed by the $g-2$ roots of $t^{g-2}+1+ \sum_{i=1}^{d} a_i( t^{g-2-i} + t^{i} )$ (in any order). We will denote these roots by $\beta_1, 1/\beta_1, \ldots, \beta_d, 1/\beta_d$, so that
\begin{equation}\label{eq:Factorisation}
t(t^2-1)^2\left(t^{g-2}+1+ \sum_{i=1}^{d} a_i( t^{g-2-i} + t^{i} )\right)^2 = t(t^2-1)^2 \prod_{j=1}^{d} (t-\beta_j)^2(t-1/\beta_j)^2.
\end{equation}
Notice that it is not necessarily the case that each point $b_i$ is rational over $K$, but the sum $\sum_{j=5}^{g+2} (b_j)$ is certainly stable under $\Gal_K$.

To make it easier to make the comparison with Schoen's article, we point out that, in the notation of the equation $(f_\ellindex) = \sum_{1 \leq j \leq 2r} \alpha_jb_j + m D_\ellindex$ \cite[bottom of page 12]{MR930145}, we have $m=4$, $r=g/2+1$, $\alpha_j=2$ for all $j \neq 3,4$, $\alpha_3=1$, and $\alpha_4=3$. 

The branch locus of $\pi_0$ is the union of the linear subspaces $\gamma_{\QuotC}(b_j \times \QuotC^{g-1})$ of $\Sym^g \QuotC \simeq \mathbb{P}^g$, for $j=1,\ldots,g+2$ \cite[Lemma 2.2]{MR930145}.  This union is defined over $K$.
Looking at~\cref{eq:gammaX}, one sees immediately that $\gamma_\QuotC( b_j \times (\mathbb{P}^1)^{g-1} )$ is the linear subspace $S_j \subseteq \mathbb{P}^g$ of homogeneous polynomials that vanish at $b_j$. Following~\cite[Page 13]{MR930145}, let $V_i = \bigcap_{j \neq 2i-1, 2i} S_j$ for $i=1,\ldots,g/2+1$ (notice that, in our setting, $(P_\infty)$ is the whole $\Sym^g\QuotC$ since $\QuotC$ is of genus $0$). Each $V_i$ is a single point: it is the projective point corresponding to the one-dimensional linear subspace of degree-$g$ homogeneous polynomials that vanish at the $g$ points $\{ b_j : j \neq 2i-1, 2i\}$. For $i=1,\ldots,g/2+1$ we can write down an obvious explicit generator $q_i(X,Y)$ of the linear subspace corresponding to $V_i$: homogenizing, we see that we may take
\begin{equation}
\begin{aligned}
q_1(X,Y) &= XY \prod_{j=1}^{d} (X-\beta_jY)(X-1/\beta_j Y), \\
q_2(X,Y) &= (X^2-Y^2) \prod_{j=1}^{d} (X-\beta_jY)(X-1/\beta_j Y),
\end{aligned}
\end{equation}
and for $i=3,\dots,d+2$, 
\begin{equation}
q_i(X,Y) = XY(X^2-Y^2) \prod_{j \neq i-2} (X-\beta_jY)(X-1/\beta_j Y).
\end{equation}
Notice that $q_1(X,Y)$ is symmetric under the exchange $X \leftrightarrow Y$, while $q_i(X,Y)$ is anti-symmetric for $i \geq 2$. Once again following~\cite[Page 13]{MR930145}, we let $P_0$ denote the projective span of the points $V_1, \ldots, V_{g/2+1}$; then $P_0$ has the expected dimension $g/2$ \cite[Lemma 2.3(v)]{MR930145}. We now describe the linear equations (with rational coefficients) that vanish on the subspace $P_0$. We let $Z_0,\ldots,Z_g$ denote linear coordinates on the space of homogeneous polynomials of degree $g$ in $X, Y$, where $Z_i$ is the coefficient of $X^iY^{g-i}$. By a slight abuse of language, we also denote by $Z_0, \ldots, Z_g$ the projective coordinates that these functions induce on $\Sym^g(\QuotC) \simeq \mathbb{P}^g$.

Recall that $q_2,\ldots,q_{g/2+1}$ are anti-symmetric polynomials, so they satisfy $(Z_j + Z_{g-j})(q_i) = 0$ (in particular, $Z_{g/2}(q_i)=0$) for $j=0,\ldots,g$ and $i=2,\ldots, g/2+1$. The quantity $Z_{g/2}(q_1)$ is nonzero: indeed, this is the coefficient of $X^{g/2-1}Y^{g/2-1}$ in
\[
\prod_{j=1}^d (X-\beta_j Y)(X- 1/\beta_j Y) = X^{g-2} + Y^{g-2} + \sum_{i=1}^d a_i(X^{g-2-i}Y^{i} + X^i Y^{g-2-i}),
\]
so it is equal to $2a_d \neq 0$ by assumption. It follows that for $j=0,\ldots,g$ there exists a unique $c_j \in K$ such that
\begin{equation}\label{eq:defci}
(Z_j + Z_{g-j})(q_1) = c_j Z_{g/2}(q_1).
\end{equation}
For $j=0,\ldots,g/2-1$, we see that the $g/2$ homogeneous linear equations given by $Z_j+Z_{g-j}-c_j Z_{g/2}=0$ are satisfied by $q_1,\ldots,q_{g/2+1}$, so that we get linear equations that vanish on $P_0$. Since these $g/2$ equations are clearly independent, they define a subspace of dimension $g/2 = \dim P_0$.  We have proven the following lemma.

\begin{lemma}\label{lemma:EquationsP0}
The following holds:
\[
P_0 = \{ [Z_0 : \cdots : Z_g] \in \mathbb{P}^g : Z_j + Z_{g-j} = c_j Z_{g/2} \textup{ for } j=0,\ldots,g/2-1 \}.
\]
\end{lemma}

\begin{remark}\label{rmk:coefficients}
We remark that by symmetry we have $c_j = c_{g-j}$ for all $j=0,\ldots,g$, and that $c_{g/2} = 2$. Also notice that we have $c_0=c_g=0$: indeed, $q_1(X,Y)$ has no term in $X^g$ and no term in $Y^g$, because by construction it is divisible by $XY$.
\end{remark}

The coefficients $c_j$ also satisfy the following:
\begin{lemma}\label{lemma:Coefficientsci}
The equality
\[
\sum_{j=0}^g c_j u^j = \frac{2}{Z_{g/2}(q_1)} q_1(u,1)
\]
holds in the ring of polynomials $K[u]$.
\end{lemma}

\begin{proof}
By construction we have 
\[ c_j = \frac{(Z_j + Z_{g-j})(q_1)}{Z_{g/2}(q_1)}, \] 
so
\begin{equation}
Z_{g/2}(q_1) \sum_{j=0}^g c_j u^j = \sum_{j=0}^g (Z_j + Z_{g-j})(q_1) u^j.
\end{equation}
Recalling that $q_1(X,Y)$ is symmetric in $X, Y$, we have the tautological identity
\[
\sum_{j=0}^g Z_j(q_1) u^j = q_1(u,1)
\]
and the claim follows. 
\end{proof}

Recall that $t$ is a coordinate on $\QuotC$. We denote by $t_1, \ldots, t_{g}$ the corresponding coordinates on $g$ copies of $\QuotC$.  Let $\ellindex \in \{1,\ldots, g+1\}$ and let $f(t) \colonequals f_\ellindex(t)$, $\gamma \colonequals \gamma_\ellindex$, and
\[
\hat{f} \colonequals \prod_{i=1}^g f(t_i),
\]
which we may consider as an element of the function field $K( \Sym^g(\QuotC) )$. 

We now need a slightly involved, but completely algebraic computation; rather than inserting it here, we relegate it to its own section: in~\cref{appendix:A} we show explicitly that
\begin{equation} 
-\hat{f}|_{P_0} = \hat{u}^4 
\end{equation}
for a suitable $\hat{u}$ in $K(P_0)$.
It follows that the minimal field $F$ containing $K$ and such that $\hat{f}$ is a fourth power in $F(P_0)$ is $F = K(\sqrt[4]{-1}) = K(\zeta_8)$.

Then the key result of Schoen \cite[Theorem 2.0]{MR930145} gives nontrivial algebraic cohomology classes in
$H^g_{\et}(\Sym^g \calC_a, \Q_\ell(g/2))$ which are defined \cite[Remark 2.13]{MR930145} over the smallest extension $F'$ of $K$ for which $\hat{f}$ is the fourth power of a function in $F'(P_0)$, that is, $F' = K(\zeta_8)$. Note that we are not yet claiming that $F'$ is the \textit{minimal} field of definition of these classes, even though this is true and will be checked below.
More precisely, Schoen claims the desired conclusion only under the assumption that the branch points $\beta_j$ be separately rational over $K$. However, the only place where this hypothesis is used in the argument is to ensure that the projective space $P_0$ be defined over $K$, which is true in our case even when the $\beta_j$ are not rational over $K$, since we have explicitly found $K$-rational equations for $P_0$: the fact that $P_0$ is defined over $K$ can be explained \textit{a priori} because of the special structure of the ramification locus, in turn induced by the presence of a second automorphism of order 4 of $\calC_\aunder$. Finally, one observes~\cite[Remark 1.4]{MR930145} that the cycles thus constructed on $\Sym^g \calC_\aunder$ actually come (via the isomorphism $\Xi^*$~\cite[Remark 1.4]{MR930145}, composed with the isomorphism $\operatorname{Alb}(\calC_a) \simeq \calA_a$ given by the canonical principal polarization) from $\calA_\aunder$---which, as it is well known, is birational to $\Sym^g \calC_\aunder$. 

It only remains to show that the cohomology classes in question are not defined over $K$ when $\zeta_8 \not \in K$. In order to do this we need to say something about the construction of these classes. Let
\begin{equation}
\begin{aligned}
\delta \colon \mathbb{Z}/4\mathbb{Z} & \to  \Aut(\calC_\aunder^g) \\
t & \mapsto ( \sigma^t, \operatorname{id},\dots,\operatorname{id}).
\end{aligned}
\end{equation}
Referring to \cref{eq:SchoenDiagram}, we may take as generators for the function field $K(\Sym^g(\QuotC))$ the elementary symmetric functions $e_1, \ldots, e_g$ in the $g$ variables $t_1, \ldots, t_g$.
So the function field of $W_0$ is $K(e_1,\ldots,e_g, \hat{z})$\ \cite[page 12]{MR930145}, where
\[
\hat{z}^4 = \prod_{i=1}^g f(t_i) = \hat{f}.
\]

Notice that $\delta(t)$ descends to the automorphism of $W_0$ induced by $\hat{z} \mapsto i^t \hat{z}$, and that our computation of $\hat{f}|_{P_0}$ implies that $P_0 \times_{\pi_0} W_0$ consists of four geometrically irreducible components, each defined over $K(\zeta_8)$: the fibre of $W_0$ over the generic point of $P_0$ has equation $\hat{z}^4 = - \hat{u}^4$. If $\zeta_8 \not \in K$, the extension $K(\zeta_8) / K$ has degree 2 (recall that by assumption $K$ contains a primitive fourth root of unity): let $\tau$ be the nontrivial automorphism of $K(\zeta_8)/K$. The action of $\tau$ on the four irreducible components of $P_0 \times_{\pi_0} W_0$ exchanges them in pairs: this can be checked on the generic points, and the four components have equations $\hat{z} = \zeta_8^i \hat{u}$ for $i \in \{1,3,5,7\}$. Letting $Q_i$ denote the component corresponding to $\hat{z} = \zeta_8^i \hat{u}$, for $i \in \{1,3,5,7\}$, we have $\tau(Q_i) = Q_{i+4 \bmod 8}$. 

Let now $\chi$ be a primitive character of $\mathbb{Z}/4\mathbb{Z}$ and consider the cycle, with coefficients in $\Q(i)$, given by \cite[page 15]{MR930145}
\[
z_\chi \colonequals \sum_{t \in \mathbb{Z}/4\mathbb{Z}} \chi(-t)\delta(t)_* Q_1.
\]
Noticing that $\delta(t)_*(Q_1) = Q_{1-2t \bmod 8}$ (check this on generic points), we obtain  that $\tau(z_\chi) = -z_\chi$, hence that the same holds for the cohomology class defined by $z_\chi$. Then Schoen \cite[Lemma 2.6]{MR930145} shows that
\begin{equation}
[z_\chi]\in H^g_\et(W_0^\al, \Q_\ell(g/2)) = H^g_\et((C_{\aunder}^{\al})^g, \Q_\ell(g/2))^{G} \subset H^g_\et(\Sym^g C_{\aunder}^{\al}, \Q_\ell(g/2))
\end{equation}
is nonzero. Since $\tau$ acts nontrivially on these nonzero classes (we have constructed two of them, corresponding to the two group isomorphisms $\chi \colon \mathbb{Z}/4\mathbb{Z} \to \{\pm 1, \pm i\}$), we have shown as desired that they are not defined over $K$. This completes the proof of~\Cref{thm:LowerBoundSatoTate}. 

For later use we also record the following lemma.

\begin{lemma}\label{lemma:TwoDimensionalSubspaceAlgebraicClasses}
The subspace of $H^g(\calA_\aunder(\mathbb{C}), \mathbb{C})$ consisting of algebraic classes is at least $2$-dimensional.
\end{lemma}

\begin{proof}
Combine results of Schoen \cite[Theorem 2.0, Lemma 1.2, and Remark 1.4]{MR930145} with the calculation above.
\end{proof}

\subsection{Algebraic classes}

In this section, we determine the connected monodromy field of $\calA_\aunder$, under the assumption that its geometric endomorphism algebra is as small as possible and that its associated Galois representations are as large as possible.
\begin{theorem}\label{thm:UpperBoundSatoTate}
Suppose that the following conditions hold:
\begin{enumroman}
    \item $a_d \neq 0$, 
    \item $\End(\calA_\aunder^\al)_{\Q} = (-1,-1\,|\,\Q)$ (in particular, $\calA_\aunder$ is geometrically simple), 
    \item $\operatorname{Hg}(\calA_\aunder)_{\mathbb{C}}\simeq \SO_{g,\mathbb{C}}$, and 
    \item $\calA_\aunder$ satisfies the Mumford--Tate conjecture. 
\end{enumroman}
Then the following statements hold:
\begin{enumalph}
\item The connected monodromy field of $\calA_\aunder$ is $K(\zeta_8)$.
\item The Tate conjecture and the Hodge conjecture are true for $\calA_\aunder$ and all its powers.
\end{enumalph}
\end{theorem}

We begin by a slightly more general discussion, so let $A$ be any $g$-dimensional abelian variety over $K$. Suppose that $\Endzero{\Abar} = B = (-1,-1\,|\,\Q)$ and that $\operatorname{Hg}(A)_{\mathbb{C}} \simeq \SO_{g, \mathbb{C}}$. Let $F$ be a quadratic imaginary subfield of $B$. 
The group $\left(B \otimes_{\Q} \mathbb{C}\right)^\times \simeq \operatorname{GL}_{2}(\mathbb{C})$ acts on $H^1(A(\C), \mathbb{C})$. This action makes $H^1(A(\C), \mathbb{C})$ into a representation of $\operatorname{GL}_{2}(\mathbb{C})$ isomorphic to $g$  copies of the standard $2$-dimensional representation. Upon restriction to the subgroup $(F \otimes \mathbb{C})^\times \simeq (\mathbb{C}^\times)^2  \subset \operatorname{GL}_{2}(\mathbb{C})$, the space $H^1(A(\C), \mathbb{C})$ decomposes as the direct sum of two subspaces $Z_\sigma$ and $Z_{\overline{\sigma}}$ on which $x \in F^\times$ acts respectively via the conjugate embeddings $\sigma, \overline{\sigma} : F \hookrightarrow \mathbb{C}$. One can show that $\dim Z_{\sigma} = \dim Z_{\overline{\sigma}}=g$ (see~Abdulali \cite[\S 4]{MR1715177} or~van Geemen--Verra \cite[Lemma 4.5]{MR1928644}).

It is not hard to show that $\tbigwedge^g(Z_\sigma)$ and $\tbigwedge^g(Z_{\overline{\sigma}})$, considered as subspaces of $\tbigwedge^g(Z_\sigma \oplus Z_{\overline{\sigma}}) = \tbigwedge^g H^1(A(\C), \mathbb{C}) = H^g(A(\C), \mathbb{C})$, are invariant under the action of the Hodge group, and therefore consist of Hodge classes. The subspace $\tbigwedge^g(Z_\sigma) \oplus \tbigwedge^g(Z_{\overline{\sigma}})$ of $H^g(A(\C), \mathbb{C})$ is defined over $\Q$~\cite[\S 4.3]{MR1928644}. We denote by $W_F \subset H^g(A(\C), \Q)$ the corresponding rational subspace and call it the space of Weil classes for the quadratic field $F$. Finally, we let $W_{B}$ be the subspace of $H^g(A(\C),\Q)$ spanned by $x \cdot W_F$ for $x \in B$. By~\cite[Proposition 4.7]{MR1928644}, the subspace $W_{B}$ is independent of the choice of the field $F$, consists of Hodge classes, and has dimension $g+1$.

Let now $\mathcal{A}^\bullet$ be the subalgebra of $H^\bullet(A(\C), \mathbb{C})$ given by Hodge classes and let $\mathcal{D}^\bullet$ be the subalgebra of $\mathcal{A}^\bullet$ generated by the classes of divisors. A straightforward computation with Hodge groups shows that $\mathcal{A}^2 = \mathcal{D}^2$ is $1$-dimensional, with generator given by the class $E$ of a polarization. Theorem 4.1 in~\cite{MR1715177} shows that
\begin{equation}\label{eq:HodgeClasses}
\mathcal{A}^{2i} = \langle \tbigwedge^i E \rangle \quad \text{for } i \neq g/2, \quad \mathcal{A}^{g} = \langle \tbigwedge^{g/2} E \rangle \oplus W_B.
\end{equation}
Finally, the proof of~\cite[Lemma 4.7]{MR1928644} shows that $W_B \otimes \mathbb{C}$ is an irreducible submodule of $H^g(A(\C), \mathbb{C})$ for the action of $(B \otimes \mathbb{C})^\times$, and this clearly implies that $W_B$ is an irreducible submodule for the action of $B^\times$.

\begin{remark}\label{rmk:DeterminantIsAlgebraicClass}
The previous discussion shows that the Weil classes for a subfield $F$ of $B$ are essentially geometric incarnations of the determinant of a certain sub-$F$-representation of $H_1(\calA_\aunder(\mathbb{C}), \C)$ of dimension $g$. The proof of~\Cref{cor:FieldsOfDefinition} proceeds precisely by looking at a certain determinant (a coefficient of a reduced characteristic polynomial) and showing that it is not stabilised by the whole absolute Galois group of $\Q(i)$. Thus, the proof of~\Cref{cor:FieldsOfDefinition} essentially amounts to explicitly verifying that certain Weil classes are not defined over $\Q(i)$.
\end{remark}

\begin{remark}\label{rmk: exceptional class as Weil class}
Schoen \cite[Corollary 3.1]{MR930145} shows that the algebraic classes considered in~\cref{sect:SatoTateField} are essentially the Weil classes for the subfield $\Q(i)$ of $B$, where $i$ is identified with the automorphism of $\Jac(\calC_\aunder)$ induced by the automorphism $\sigma$ of order 4 described in~\cref{eq:Sigma}.
\end{remark}

We can now prove~\Cref{thm:UpperBoundSatoTate}.

\begin{proof}[Proof of~\textup{\Cref{thm:UpperBoundSatoTate}}]
We specialize the above discussion to $A=\calA_\aunder$. 
The $2$-dimensional linear subspace $U \subseteq H^g(\calA_\aunder(\C), \C)$ of algebraic classes given by~\Cref{lemma:TwoDimensionalSubspaceAlgebraicClasses} cannot be contained in $\mathcal{D}^g$, which is 1-dimensional. Since algebraic classes are Hodge, this shows that $U\subset \mathcal{A}^g$ intersects $W_B$ nontrivially. Moreover, the subspace of algebraic classes is stable under the action of $B=\End (\calA_\aunder^\al)_{\Q}$, so the whole $W_{B}$ (which is an irreducible $B$-module) consists of algebraic classes. Using~\cref{eq:HodgeClasses} we have then proved that the Hodge conjecture holds for $\calA_\aunder$, and since by assumption the Mumford--Tate conjecture is true for $\calA_\aunder$ we obtain that the Tate conjecture also holds for $\calA_\aunder$~\cite[Proposition 2.3.2]{MR3735565}.

We have proved that $\mathcal{A}^\bullet$ is generated by $\mathcal{D}^\bullet$ and by the translates of Weil classes under the action of endomorphisms. In particular, the minimal field of definition of the ($\ell$-adic) cohomology classes of all the algebraic cycles on $\calA_\aunder$ 
is the minimal extension of $K(\End A_\aunder)$ over which at least one Weil class is defined. Note that by~\Cref{lemma:FromEndoAlgebraToEndoRing} we have $K(\End \calA_a)=K(i) \subseteq K(\zeta_8)$. \Cref{thm:LowerBoundSatoTate} shows that some Weil classes are defined over $K(\zeta_8)$, hence they all are. This implies that the minimal field of definition of all algebraic classes on $\calA_a$ is $K(\zeta_8)$.

Abdulali~\cite[Theorem 4.1]{MR1715177} shows that, since all Hodge classes on $\calA_\aunder$ are algebraic, the Hodge conjecture is true for all powers $\calA_\aunder^k$ of $\calA_\aunder$. Since the Mumford--Tate conjecture holds for $\calA_\aunder$, by the main theorem of Commelin~\cite{MR4009176} the same is true for $\calA_\aunder^k$ for all $k \geq 1$. As a consequence, there is a canonical identification between Hodge classes and Tate classes on any power of $\calA_\aunder$. Moreover, as above, Moonen \cite[Proposition 2.3.2]{MR3735565} implies that the Tate conjecture holds for all powers of $\calA_\aunder$.

We claim that all cohomology classes of algebraic cycles on all powers of $\calA_\aunder$ are defined over $K(\zeta_8)$. Write as above $H^1(\calA_\aunder(\C), \C)$ as the direct sum of two $g$-dimensional subspaces $Z_{\sigma}, Z_{\overline{\sigma}}$ on which $\Hg(\calA_\aunder)_\C$ acts via the standard representation. The proof of Abdulali~\cite[Theorem 4.1]{MR1715177} shows that the Hodge classes in the cohomology of $\calA_\aunder^k$ are generated by divisors and by $\left(\tbigwedge^{g}Z_\sigma \right)^k \oplus \left(\tbigwedge^{g}Z_{\overline{\sigma}} \right)^k$. We have already seen that the classes in $\tbigwedge^{g}Z_\sigma, \tbigwedge^{g}Z_{\overline{\sigma}} \subset H^g(\calA_\aunder(\C), \C)$ can be represented by cycles defined over $K(\zeta_8)$, so it suffices to show that the same is true for all divisor classes on $\calA_\aunder^k$. To see this, recall that divisor classes are precisely the Hodge classes in $H^2(\calA_\aunder^k(\C), \C) \simeq \tbigwedge^2\left(H^1(\calA_\aunder(\C), \C)^{\oplus k} \right)$. This space is a direct sum of copies of $\tbigwedge^2 H^1(\calA_\aunder(\C), \C) \simeq H^2(\calA_\aunder(\C), \C)$ and $H^1(\calA_\aunder(\C), \C) \otimes H^1(\calA_\aunder(\C), \C)$. Hodge classes in $H^2(\calA_\aunder(\C), \C)$ are in particular Hodge classes on $\calA_\aunder(\C)$, hence they are defined over $K(\zeta_8)$ by the above (and in fact, even over $K$, since the class of the polarization of $\calA_\aunder$, which generates $\mathcal{A}^2$, is defined over $K$). Hodge classes in $H^1(\calA_\aunder(\C), \C) \otimes H^1(\calA_\aunder(\C), \C)$ can be identified with endomorphisms of $\calA_\aunder(\C)$, and these are all defined over $K(i)$. This concludes the proof of the claim.

Since we have already shown the Tate conjecture for all powers of $\calA_\aunder$, we are in a position to apply~\Cref{cor:apres proposition}, which finally proves $K(\varepsilon_{A_\aunder}) = K(\zeta_8)$.
\end{proof}

\subsection{A function field computation}\label{appendix:A}

We complete the argument of~\cref{sect:SatoTateField} by proving that $-\hat{f}|_{P_0}$ is a fourth power in the function field $K(P_0)$. We keep all the notation from~\cref{sect:SatoTateField}.

As (natural) generators for the function field $K( \Sym^g(\QuotC) )$ we may take the elementary symmetric functions $e_1, \ldots, e_g$ in the $g$ variables $t_1,\ldots,t_g$. To make the notation more uniform we also set $e_0=1$. From the well-known relation between coefficients of a polynomial and symmetric functions of its roots we obtain $e_i = (-1)^i Z_{g-i}/Z_g$, so that the equations defining $P_0$ inside $\Sym^g(\QuotC)$ give
\begin{equation}\label{eq:EquationsP0SymmetricFunctions}
e_i + e_{g-i} = (-1)^{g/2+i} c_i e_{g/2}
\end{equation}
upon restriction to $K(P_0)$. Notice in particular that the equation corresponding to $i=0$ (or equivalently $i=g$) gives $e_g = -1$ (by~\Cref{rmk:coefficients} we have $c_g=c_0=0$). In order to compute $\hat{f}|_{P_0}$ we then need to write $\hat{f}$ as a rational function $h(e_1,\ldots,e_g)$ of the elementary symmetric functions, and then evaluate it at
\[
e_{g-i} = -e_i + (-1)^{g/2+i} c_i e_{g/2} \quad \text{ for } i=0,\ldots,\frac{g}{2}-1.
\]
We start by expressing $\hat{f}$ in terms of symmetric functions. Set for simplicity $\delta_j = -\beta_j$.
We have
\begin{equation}\label{eq:main equation}
\begin{aligned}
\prod_{i=1}^g (t_i-\gamma)^{2g+4} \cdot\hat{f} & =  \prod_{i=1}^g \left( t_i(t_i^2-1)^2 \prod_{j=1}^{d} (t_i-\beta_j)^2\left(t_i-\frac{1}{\beta_j} \right)^2 \right) \\
& =  e_g \prod_{i=1}^g (1+t_i)^2 \cdot  \prod_{i=1}^g(1-t_i)^2 \cdot \prod_{j=1}^{d}\prod_{i=1}^g  (t_i + \delta_j)^2 \left(t_i+\frac{1}{\delta_j} \right)^2 \\
& = e_g \left( \sum_{i=0}^g e_i \right)^2 \cdot  \left( \sum_{i=0}^g (-1)^i e_i \right)^2 \cdot \prod_{j=1}^{d}\left( \sum_{i=0}^g \delta_j^i e_{g-i} \right)^2 \left( \sum_{i=0}^g \delta_j^{-i} e_{g-i} \right)^2.
\end{aligned}
\end{equation}

Notice that, as $g$ is even, $\prod_{i=1}^g (t_i-\gamma)^{2g+4}$ is obviously a fourth power.
We now evaluate every factor on the right upon pull-back to $K(P_0)$. Clearly $e_g$ is simply $-1$. Our task is therefore to show that all the remaining factors multiply up to a fourth power.
Next we have
\[
\begin{aligned}
\left( \sum_{i=0}^g e_i \right) \cdot  \left( \sum_{i=0}^g (-1)^i e_i \right) & = \left( \frac{1}{2} \sum_{i=0}^g (e_i+e_{g-i}) \right) \cdot  \left( \frac{1}{2} \sum_{i=0}^g (-1)^i (e_i+e_{g-i}) \right) \\
& = \left( \frac{1}{2} \sum_{i=0}^g (-1)^{g/2+i}c_i e_{g/2} \right) \cdot  \left( \frac{1}{2} \sum_{i=0}^g (-1)^{g/2}c_i e_{g/2} \right) \\
& =  \left( \sum_{i=0}^g (-1)^{i}c_i   \right)  \left(  \sum_{i=0}^g c_i \right) \cdot \frac{e_{g/2}^2}{4} \\
& = q_1(1,1)q_1(-1,1) \frac{e_{g/2}^2}{Z_{g/2}(q_1)^2},
\end{aligned}
\]
where in the last equality we used~\Cref{lemma:Coefficientsci}. Finally, from the definition of $q_1(X,Y)$ we obtain
\[
\begin{aligned}
\left( \sum_{i=0}^g e_i \right) \cdot  \left( \sum_{i=0}^g (-1)^i e_i \right) & =\prod_{j=1}^d (1-\delta_j^2)(1-\delta_j^{-2}) \cdot \frac{e_{g/2}^2}{Z_{g/2}(q_1)^2} \\
& =(-1)^d \prod_{j=1}^d (\delta_j-\delta_j^{-1})^2 \cdot \frac{e_{g/2}^2}{Z_{g/2}(q_1)^2}.
\end{aligned}
\]
Finally we consider
\[
\begin{aligned}
\left( \sum_{i=0}^g \delta_j^i e_{g-i} \right) & \left( \sum_{i=0}^g \delta_j^{-i} e_{g-i} \right)= \\
& = \left( \sum_{i=0}^d (\delta_j^i e_{g-i} + \delta_j^{g-i}e_i) + \delta_j^{g/2} e_{g/2} \right)
\left( \sum_{i=0}^d (\delta_j^{-i} e_{g-i} + \delta_j^{i-g}e_i) + \delta_j^{-g/2} e_{g/2} \right).
\end{aligned}
\]
Replacing $e_{g-i}$ by $-e_i + (-1)^{g/2+i} c_i e_{g/2}$ and multiplying the first (resp.~second) term by $\delta_j^{-g/2}$ (resp.~$\delta_j^{g/2}$) we obtain
\[
\begin{aligned}
\left( \sum_{i=0}^g \delta_j^i e_{g-i} \right) & \left( \sum_{i=0}^g \delta_j^{-i} e_{g-i} \right)= \\
& = \left( \sum_{i=0}^d e_i (\delta_j^{g/2-i} - \delta_j^{i-g/2}) + e_{g/2} + (-1)^{g/2}e_{g/2} \sum_{i=0}^d \delta_j^{i-g/2} (-1)^{i} c_i \right) \cdot \\
& \quad \, \left( \sum_{i=0}^d e_i (\delta_j^{i-g/2} - \delta_j^{g/2-i}) + e_{g/2} + (-1)^{g/2} e_{g/2} \sum_{i=0}^d \delta_j^{g/2-i} (-1)^{i} c_i \right).
\end{aligned}
\]
We will show that the two factors in this last expression are negative of each other, so that their product is minus a square. It is clear that
\[
\sum_{i=0}^d e_i (\delta_j^{g/2-i} - \delta_j^{i-g/2}) = - \sum_{i=0}^d e_i (\delta_j^{i-g/2} - \delta_j^{g/2-i}).
\]
We now claim that
\begin{equation}\label{eq:SurprisingSymmetry}
e_{g/2} + (-1)^{g/2}e_{g/2} \sum_{i=0}^d \delta_j^{i-g/2} (-1)^{i} c_i = -\left(  e_{g/2} + (-1)^{g/2} e_{g/2} \sum_{i=0}^d \delta_j^{g/2-i} (-1)^{i} c_i \right).
\end{equation}
To show this, we begin by factoring out $e_{g/2}$, multiplying both sides by $(-1)^{g/2} \delta_j^{g/2}$, and bringing everything to the left hand side: our claim can then be rewritten as
\[
2 (-1)^{g/2} \delta_j^{g/2} +  \sum_{i=0}^d \delta_j^{i} (-1)^{i} c_i + \sum_{i=0}^d \delta_j^{g-i} (-1)^{i} c_i = 0.
\]
Recalling that $c_{g-i}=c_i$ and $c_{g/2}=2$, and using~\Cref{lemma:Coefficientsci} again, the above expression may be rewritten as
\[
\sum_{i=0}^g c_i(-\delta_j)^{i} = \frac{2}{Z_{g/2}(q_1)} q_1(-\delta_j,1) = \frac{2}{Z_{g/2}(q_1)} q_1(\beta_j,1)=0,
\]
as claimed. Inserting this information in our previous computations we obtain
\[
\begin{aligned}
\left( \sum_{i=0}^g \delta_j^i e_{g-i} \right) & \left( \sum_{i=0}^g \delta_j^{-i} e_{g-i} \right)= \\
& = - \left( \sum_{i=0}^d e_i (\delta_j^{g/2-i} - \delta_j^{i-g/2}) + e_{g/2} + (-1)^{g/2}e_{g/2} \sum_{i=0}^d \delta_j^{i-g/2} (-1)^{i} c_i \right)^2.
\end{aligned}
\]
Using~\cref{eq:SurprisingSymmetry} again, we further rewrite this as
\begin{equation} \label{eqn:iodeltagta}
- \left( \sum_{i=0}^d e_i \frac{\delta_j^{g/2-i} - \delta_j^{i-g/2}}{\delta_j - \delta_j^{-1}} + \frac{1}{2}(-1)^{g/2}e_{g/2} \sum_{i=0}^d  (-1)^{i} c_i \frac{\delta_j^{i-g/2} - \delta_j^{g/2-i}}{\delta_j - \delta_j^{-1}} \right)^2 (\delta_j - \delta_j^{-1})^2.
\end{equation}
Recall the well-known fact that for every $k \geq 0$ the expression $\frac{\delta_j^k - \delta_j^{-k}}{\delta_j - \delta_j^{-1}}$ is a polynomial in $(\delta_j + \delta_j^{-1})$. Notice that the coefficients of this polynomial are independent of $j$, so we write the first factor in \eqref{eqn:iodeltagta} as $\tilde{Q}(\delta_j+\delta_j^{-1})$, where $\tilde{Q}$ is a polynomial in $K(P_0)[x]$.
It follows that the product
\[
\begin{aligned}
Q & \colonequals \prod_{j=1}^d (-1)\left( \sum_{i=0}^d e_i \frac{\delta_j^{g/2-i} - \delta_j^{i-g/2}}{\delta_j - \delta_j^{-1}} + \frac{1}{2}(-1)^{g/2}e_{g/2} \sum_{i=0}^d  (-1)^{i} c_i \frac{\delta_j^{i-g/2} - \delta_j^{g/2-i}}{\delta_j - \delta_j^{-1}} \right) 
\\&= \prod_{j} \tilde{Q}(\delta_j+\delta_j^{-1})
\end{aligned}
\]
is a symmetric function of the variables $\{\delta_1+\delta_1^{-1},\ldots,\delta_d+\delta_d^{-1}\}$, and therefore, it is a rational function (with coefficients in $K(e_1,\ldots,e_{g/2})=K(P_0)$) of the coefficients of the polynomial $\prod_{j=1}^d (t+\delta_j)(t+\delta_j^{-1})$. But this polynomial is in $K[t]$ by definition, so $Q \in K(P_0)$. Putting everything together,~\cref{eq:main equation} gives
\[
\prod_{i=1}^g (t_i-\gamma)^{2g+4} \cdot \hat{f}|_{P_0} = - \frac{e_{g/2}^4}{Z_{g/2}(q_1)^4} Q^4 \prod_{j=1}^d (\delta_j - \delta_j^{-1})^8.
\]
It just remains to notice that
\[
R \colonequals \prod_{j=1}^d (\delta_j - \delta_j^{-1})^2
\]
is invariant under $\delta_j \leftrightarrow \delta_j^{-1}$ for all $j$ and under all permutations of the indices, hence lies in $K$. Therefore $\prod_{j=1}^d (\delta_j - \delta_j^{-1})^8=R^4$ is a fourth power in $K(P_0)$, as claimed.

\section{Proof of main result}\label{sec: higher genera}

\subsection{Generic monodromy: overview and setup}\label{subsec: generic monodromy}

Recall from~\Cref{setup} the family $\calC^{(g)} \to \AAp$, its generic fibre $C^{(g)}$ over $\Q(a)$, and the Jacobian $A^{(g)}=\Jac C^{(g)}$.  (Notice that $\calK$ can be considered as a finitely-generated subfield of $\C$, so all the constructions and considerations of \cref{subsec: abelian varieties} apply.)

Recalling~\Cref{def:FullyLefschetz}, our main goal in this section is to prove the following theorem.

\begin{theorem}\label{thm:GenericMonodromy}
For every even integer $g \geq 4$ the geometric endomorphism algebra of $A^{(g)}$ is $(-1,-1\,|\,\Q)$ and $A^{(g)}$ is fully of Lefschetz type.
\end{theorem}

We will conclude the main result~\Cref{thm:MainHigherGenus} from this theorem, as follows.
The proof of~\Cref{thm:GenericMonodromy} is by induction, but in the inductive step we will rely on~\Cref{thm:MainHigherGenus} for genera $g-2$ and $g-4$. The reasoning is not circular: we will show below (\cref{sect:VariationOfGaloisInFamilies}) that~\Cref{thm:GenericMonodromy} for genus $g$ implies~\Cref{thm:MainHigherGenus} for genus $g$, see~\Cref{fig:flow}. Finally, in~\Cref{lemma: thm generic monodromy holds for 4 and 6} we will also show~\Cref{thm:GenericMonodromy} for $g=4, 6$ without relying on~\Cref{thm:MainHigherGenus}, which will complete the inductive proof of both results.

\begin{figure}[ht]
    \centering
    \begin{tikzpicture}
        \node[draw,rectangle,rounded corners=3pt] (0) at (0,3)  {\Cref{thm:MainHigherGenus} in genus $g-2, g-4$ };
        \node[draw,rectangle,rounded corners=3pt] (1) at (0,0)  {\Cref{prop:GeometricallySimpleHigherGenus}};
        \node[draw,rectangle,rounded corners=3pt] (2) at (8,0)  {\Cref{thm:GenericMonodromy} in genus $g$};
        \node[draw,rectangle,rounded corners=3pt] (3) at (8,3)  {\Cref{thm:MainHigherGenus} in genus $g$};
        \draw [-{Implies},double,line width=0.4pt,double equal sign distance,
            shorten <=6pt,shorten >=6pt] (0)--(1);
        \draw [-{Implies},double,line width=0.4pt,double equal sign distance,
            shorten <=6pt,shorten >=6pt] (1)--(2);
        \draw [-{Implies},double,line width=0.4pt,double equal sign distance,
            shorten <=6pt,shorten >=6pt] (2)--(3);
    \end{tikzpicture}
    \caption{Flow of the proof of~\Cref{thm:MainHigherGenus}}
    \label{fig:flow}
\end{figure}

Before embarking, we discuss specializations of abelian schemes (and their endomorphisms) in our geometric setting.  Let $K$ be a number field, let $U$ be a normal integral scheme over $K$, let $\mathcal{A} \to U$ be an abelian scheme of relative dimension $g \geq 1$, and let $\overline{\eta}$ be the geometric generic point of $U$.
There is a representation $\rho_{\mathcal{A}}$ of the \'etale fundamental group $\pi_1(U, \overline{\eta})$ with values in $\operatorname{Aut}( \textstyle{\varprojlim}_m \mathcal{A}_{\overline{\eta}}[m]) \simeq \operatorname{GL}_{2g}(\widehat{\mathbb{Z}})$ \cite[Introduction]{ZywinaFamilies}.  For a prime $\ell$, we may also construct an $\ell$-adic representation of the fundamental group
\begin{equation}
\rho_{\mathcal{A}, \ell} \colon \pi_1(U, \overline{\eta}) \to \operatorname{Aut}\Bigl( \varprojlim_m \mathcal{A}_{\overline{\eta}}[\ell^m]\Bigr) \simeq \operatorname{GL}_{2g}(\mathbb{Z}_\ell).
\end{equation}
The two constructions are compatible in an obvious sense: $\rho_{\mathcal{A}, \ell}$ is the composition of $\rho_{\mathcal{A}}$ with the natural projection $\operatorname{GL}_{2g}(\widehat{\mathbb{Z}}) \to \operatorname{GL}_{2g}(\mathbb{Z}_\ell)$.
For every $u \in U(K)$ there is a homomorphism $u_* : \Gal_K = \pi_1(\Spec K,\Spec \Kbar) \to \pi_1(U,\overline{\eta})$, well-defined up to conjugacy, such that the adelic Galois representation attached to the abelian variety $\mathcal{A}_u$ may be identified (up to conjugacy) with the composition $\rho_\mathcal{A} \circ u_*$. Similarly, we have $\rho_{\mathcal{A}, \ell} \circ u_* = \rho_{\mathcal{A}_u, \ell}$.

\begin{lemma}\label{lem:specialization}
The image of the adelic representation $\pi_1(U, \overline{\eta}) \to \Aut(\varprojlim_m \mathcal{A}_{\overline{\eta}}[m])$ attached to the abelian scheme $\mathcal{A} \to U$ coincides with the image of the adelic representation $\rho_A : \Gal_{K(U)} \to \Aut(\varprojlim_m \mathcal{A}_{\overline{\eta}}[m])$ attached to the abelian variety $A/K(U)$, where $A$ is the generic fibre of $\mathcal{A}$.
\end{lemma}

\begin{proof}
Let $\eta = \Spec K(U)$ be the generic point of $U$ and $\overline{\eta}$ be the corresponding geometric generic point.
By~\cite[\href{https://stacks.math.columbia.edu/tag/0BQM}{Proposition 0BQM}]{Stacks}, the natural map $\eta_* : \pi_1( \eta, \overline{\eta} ) \to \pi_1( U, \overline{\eta} )$ is surjective. The claim follows from the formula
$
\rho_{A} = \rho_{\mathcal{A}} \circ \eta_*.
$
\end{proof}

In particular, one may view the image of the Galois representation attached to $\mathcal{A}_u$ as a subgroup of $\rho_{\mathcal{A}}(\pi_1(U, \overline{\eta}))$. When $U$ is an open subscheme of $\mathbb{P}^n_K$, for most $u \in U(K)$ this inclusion has small index, as made precise by the following result.

\begin{theorem}[{Zywina~\cite[Theorem 1.1]{ZywinaFamilies}}]\label{thm:ZywinaFamilies}
Let $U$ be an open subscheme of $\mathbb{P}^n_K$. Let $H$ be the absolute multiplicative height on $\mathbb{P}^n(K)$. Then there exists a constant $c$ such that
\begin{equation}\label{eq:FullDensity}
\lim_{x \to \infty} \frac{\# \{u \in U(K) : H(u) \leq x, [\rho_{\mathcal{A}}( \pi_1(U,\overline{\eta}) )  : \rho_{\mathcal{A}_u} (\Gal_K) ] \leq c\} }{\# \{u \in U(K) : H(u) \leq x \}} = 1.
\end{equation}
\end{theorem}

Suppose now that $U \subseteq \A_\Q^d$ is a nonempty open subscheme.  Write $A$ for the generic fibre of $\mathcal{A}$ (over $\Q(U)$).  Consider further an integral, closed, normal subscheme $T \subseteq U$.  Let $\mathcal{A}_T \colonequals \mathcal{A} \times_U T$ and its special fibre $A_T$ over the function field of $T$.  We collect in the next lemma some basic properties of the specialization homomorphism.

\begin{lemma}\label{lem:specialization2}
The following statements hold.
\begin{enumalph}
\item The endomorphism ring of $\mathcal{A}$ coincides with the endomorphism ring of its generic fibre $A$.
\item The endomorphism ring of $\mathcal{A}$ injects into the endomorphism ring of $\mathcal{A}_T$.
\item The image of the $\ell$-adic monodromy representation attached to $\mathcal{A}_T$ (resp.~$A_T$) is contained in the image of the $\ell$-adic monodromy representation attached to $\mathcal{A}$ (resp.~$A$).
\item The endomorphism ring of $A$ (resp.~$A^{\al}$) injects into the endomorphism ring of $A_T$ (resp.~$A_T^{\al}$).
\end{enumalph}
\end{lemma}

\begin{proof}
For (a), it suffices to show that any endomorphism of $\mathcal{A}$ that is defined over a nonempty open subscheme $V$ of $U$ extends to all of $U$, and this follows immediately from work of Faltings \cite[Lemma 1]{MR718935}.

For (b), we apply the rigidity lemma~\cite[Corollary 6.2]{GIT} (in the notation there, take $X=G=A$, $f=\phi$, and $g=0$).

For (c), we start with the case of $\mathcal{A}_T$ and $\mathcal{A}$. Notice that the given embedding $T \hookrightarrow U$ induces a map of fundamental groups $\pi_1(T, \overline{\eta_T}) \to \pi_1(U, \overline{\eta})$, well-defined up to conjugation (which takes care of the arbitrary choice of the base point $\overline{\eta_T}$), such that the following diagram commutes:
\[
\xymatrix{
\pi_1(T, \overline{\eta_T}) \ar[r] \ar[d]_{\rho_{\mathcal{A}_T, \ell}} & \pi_1(U, \overline{\eta_U}) \ar[d]^{\rho_{\mathcal{A}, \ell}} \\
\Aut( T_\ell \mathcal{A}_T ) \ar[r] & \Aut( T_\ell \mathcal{A} )
}
\]
The claim follows immediately.

The statement for $A_T$ and $A$ then follows from~\Cref{lem:specialization}, since the images of the Galois representations are the same for the scheme and the generic fibre.

Finally, we prove (d).  The statement concerning $A$ and $A_T$ is an immediate consequence of (a) and (b).  Let now $L$ be the subfield of $\Q(U)^\al$ over which all endomorphisms of $A$ are defined. Letting $U_L$ and $T_L$ be the normalizations of $U, T$ in $L$, we can apply (a) and (b) to $\mathcal{A}_{U_L}$ and $\mathcal{A}_{T_L}$ to get injectivity at the level of the geometric endomorphism rings.
\end{proof}

Suppose furthermore that $T$ is an open subscheme of $\mathbb{A}^1_{\Q}$ (in particular, $T$ is normal). Then $\mathcal{A}_T$ admits a N\'eron model $\mathfrak{A}$ over all of $\mathbb{A}^1_{\Q}$.
In this more specific situation we have the following lemma.

\begin{lemma}\label{lem:specialization3}
Let $y \in \mathbb{A}^1(\Qbar)$ be such that $\mathfrak{A}_y$ is an abelian variety.  Then the following statements hold.
\begin{enumalph}
\item The endomorphism ring of $A_T$ (resp.~$A_T^{\al}$) injects into the endomorphism ring of the special fibre $\mathfrak{A}_y$ of $\mathfrak{A}$ at $y$ (resp.~$\mathfrak{A}^{\al}_y$).
\item The image of the $\ell$-adic Galois representation attached to $\mathfrak{A}_y$ injects into the image of the $\ell$-adic Galois representation attached to $\mathcal{A}_T$.
\end{enumalph}
\end{lemma}

\begin{proof}
By Faltings~\cite[Lemma 1]{MR718935}, the endomorphism ring of $\mathcal{A}_T$ is the endomorphism ring of $A_T$.
The universal property of N\'eron models implies that endomorphisms of $\mathcal{A}_T$ extend to $\mathfrak{A}$, so that $\End (A_T) = \End (\mathcal{A}_T) = \End(\mathfrak{A})$.

We now prove the desired injectivity. Let $R$ be a Dedekind domain containing $\Q$, $F=\operatorname{Frac}(R)$, let $A$ be an abelian variety over $F$, and let $\mathfrak{A}$ its N\'eron model over $R$. Let $y$ be a closed point of $\Spec R$ and let $\kappa(y)$ be the residue field at $y$. Suppose that $\mathfrak{A}_y$ is an abelian variety.
The specialization map $A(F) = \mathfrak{A}(R) \to \mathfrak{A}_y(\kappa(y))$ is injective on torsion points~\cite[Proposition 3 in \S 7.3]{MR1045822} (notice that all positive integers are invertible in $R$). We will prove in general that $\End( \mathfrak{A} ) \to \End(\mathfrak{A}_y)$ is injective. Note that the universal property of N\'eron models shows $\End( \mathfrak{A} ) = \End(A)$. To prove injectivity, take an element $\varphi$ in the kernel. Let $P \in A(F^\al)_{\tors}$ and let $L$ be a finite extension of $F$ over which $P$ is defined. Denote by $R_L$ the integral closure of $R$ in $L$ and by $\mathfrak{A}_L$ the N\'eron model of $A_L$. Fix a point $y_L$ of $\Spec R_L$ over $y$. Since $\mathfrak{A}_y$ is an abelian variety, $(\mathfrak{A}_L)_{y_L}$ is simply the base-change of $\mathfrak{A}_y$ to $\kappa(y_L)$. In particular, $\varphi$ induces zero also on $(\mathfrak{A}_L)_{y_L}$. Hence, the image of $\varphi(P)$ in $(\mathfrak{A}_L)(\kappa(y_L))$ is zero, but $\varphi(P) \in \mathfrak{A}_L(L)=A(L)$ is torsion, and specialization is injective on torsion points, so $\varphi(P)=0$. Since this holds for all torsion points $P$, the endomorphism $\varphi$ is zero, as claimed in part (a). The statement about geometric endomorphism rings is proved exactly in the same way.

For part (b), arguing as in part (a) and passing to the limit over all $L$ (equivalently, over all torsion points), we get canonical identifications $A(K^\al)_{\tors} \simeq (\mathfrak{A}_y)(\kappa(y)^\al)_{\tors}$. These identifications are clearly Galois-equivariant, and the claim follows.
\end{proof}

\Cref{lem:specialization2,lem:specialization3} imply the following.

\begin{proposition}\label{prop:SpecialisationArguments}
Let $U$ be an open subscheme of $\mathbb{A}^d_\Q$, and let $\mathcal{A}$ be an abelian scheme over $U$, with generic fibre $A$ over $\Q(U)$. Let $T$ be a closed subscheme of $U$ isomorphic to an open subscheme of $\A^1_\Q$. Let $\mathfrak{A} / \A^1_\Q$ be the N\'eron model for the restriction $\mathcal{A}_T$ of $\mathcal{A}$ to $T$. For every $x \in \A^1_\Q(\Q)$ such that the special fibre $\mathfrak{A}_x / \Q$ is an abelian variety, 
there are injective specialization maps $\End(A) \to \End(\mathfrak{A}_x)$ and $\End(A^\al) \to \End(\mathfrak{A}_x^\al)$. The image of the $\ell$-adic Galois representation attached to $\mathcal{A}$ contains the image of the $\ell$-adic Galois representation attached to $\mathfrak{A}_x$.
\end{proposition}

\begin{remark}\label{rmk:specialization-over-number-fields}
    The same statements hold, with the same proof, over any number field $K$, also after pulling back $\mathcal{A}$ along a morphism $T \to U$, where $T$ is a nonempty open subscheme of $\A^1_K$.
\end{remark}

\begin{remark}\label{rmk:SpecialisationSplittings}
Decompositions up to isogeny of an abelian variety $A$ correspond to idempotents in the endomorphism algebra of $A$. As a consequence, whenever endomorphism specialize injectively (such as in the situation of the previous proposition), splittings of abelian varieties also specialize.
\end{remark}

\subsection{Generically geometrically simple: base cases}

Recall the abelian variety $A^{(g)}$ from the beginning of~\cref{subsec: generic monodromy}.  Over the next two sections, we show that $A^{(g)}$ is geometrically simple.  We start by taking care of the special cases $g=4$ and $g=6$.  

\begin{lemma}\label{lemma: geometric endomorphism algebra in genera 4 and 6}
Let $g \in \{4,6\}$. Then the abelian variety $A^{(g)}$ has geometric endomorphism algebra $(-1,-1\,|\,\Q)$.  In particular, $A^{(g)}$ is geometrically simple.
\end{lemma}

\begin{proof}
By~\Cref{lemma: endomorphisms generic fibre}, we have an injective map $(-1,-1\,|\,\Q) \hookrightarrow \End (A^{(g),\al})_\Q$.  On the other hand, by~\Cref{thm:EndomorphismsGenus4} for $g=4$ and \Cref{prop:Genus6} for $g=6$, there exist specializations $\calA_a^{(g)}$ with $(\End \calA_a^{(g), \al} )_\Q = (-1,-1\,|\,\Q)$.  By~\Cref{prop:SpecialisationArguments} and a dimension count, we conclude that $\End (A^{(g),\al})_\Q=(-1,-1\,|\,\Q)$.
\end{proof}

In our inductive argument, we will use a degeneration (specialization) argument.  Since the argument is a bit notationally heavy, we briefly work with the special case $g=4$ to introduce some of the key ideas; we use this also in \cref{sec:explrec}.  

Recall from~\cref{eq:HyperellipticEquationGeneralgFactored} the family $\Cprim^{(g)}$ of nice curves over $\BB$ with Jacobian $\calJprim^{(g)} \colonequals \Jac \Cprim^{(g)}$.  
The generic fibre $\Jprim^{(g)}$ of $\calJprim^{(g)}$ is the base change of $A^{(g)}$ under $\calK \hookrightarrow \calL$.

\begin{prop} \label{prop:goodreductiong4}
Let $T$ be the nice curve of genus $2$ over $\Q$ defined by $t^2=b(b^4-1)$ with projection $\pi \colon T \to \PP^1_b$.  Then the N\'eron model of the pullback of $\calJprim^{(4)}=\Jac \Cprim^{(4)}$ under $\pi$ is an abelian scheme over $T$.
\end{prop}

\begin{proof}
The essential information necessary to understand the reduction properties of a hyperelliptic curve are contained in its \emph{cluster picture} \cite{cluster,best2021users}, observing that the results extend to the situation of a DVR with perfect infinite residue field.

Our family is smooth away from $b \in 0,\pm 1, \infty, \pm i$, so it suffices to prove good reduction of the pullback Jacobian at the points of $T$ above them, exactly the branch points of $\pi$.  

Let $P$ be a branch point of $\pi$, let $\mathcal{O}_K \colonequals \widehat{\mathcal O}_{T,P}$ be the completed local ring, and let $K$ be its field of fractions.  The roots of $f_b(x)=x(x^4-1)(x^2-b^2)(x^2-b^{-2})$ are $\mathcal R \colonequals \{0,\pm 1, \pm i,\pm b, \pm b^{-1}\}$ and live in the unramified extension $K(i) \supseteq K$.  The invariant attached to a proper cluster $\fraks$ is 
\[ \nu_{\fraks}
    =(\#\fraks)d_{\fraks}
     +\sum_{r\in\mathcal R\setminus\fraks}
        d_{\fraks\wedge\{r\}}, \]
where $\fraks\wedge\{r\}$ denotes the smallest cluster containing $\fraks$ and $r$.  

First suppose $b=0$.  The proper clusters are 
\begin{equation}
\mathcal{R} \supset \{0,\pm b,\pm 1, \pm i\} \supset \{0,\pm b\}
\end{equation}
of sizes $9,7,3$, respectively.  Their depths are $-1,0,1$ so their depths over $K$ are $-2,0,2$.  The corresponding values of $\nu_{\fraks}$ are $-18,-4,2$.  All three clusters are odd and principal.  

The criterion for good reduction of the Jacobian \cite[Theorem 5.4]{best2021users} has three parts:
\begin{enumerate}
\item the set of roots $\mathcal{R}$ is defined over $K(i)$, which is unramified over $K=\QQ((t))$; 
\item all (proper) clusters $\mathfrak{s} \neq \mathcal{R}$ are odd; and 
\item all proper clusters $\mathfrak{s}$ are principal with $v_{\mathfrak{s}} \in 2\Z$, since all depths are even.    
\end{enumerate}
(In particular, repeating the calculation in (3) for the family over $\QQ[b]$ we see it has bad, potentially good, reduction: two clusters have odd depth.)  

Although it is not necessary to prove the lemma, we also obtain a description of the special fibre as coming from maximal principal odd subclusters: we obtain the equations (over $\Q^{\al} \simeq \QQ^{\al}[[t]]/(t)$)
\begin{equation}
\begin{aligned}
y^2 = x(x^2-1) \\
y^2 = -x(x^4-1) \\
y^2 = x(x^2-1)
\end{aligned}
\end{equation}
for each of the three clusters, respectively, attached from top to bottom.  We do a similar calculation, but in a self-contained way, with explicit equations in Lemma \ref{lemma: specialisation to (g-2)-dimensional x E^2}.  

We repeat this at the other valuations.  The case $b=\infty$ is identical, replacing $b$ by $1/b$.  For $b=1$, now the proper clusters are $\mathcal{R}$ and its two disjoint subclusters $\fraks_+ \colonequals \{1,b,1/b\}$ and $\fraks_- \colonequals \{-1,-b,-1/b\}$.  Their sizes are $9,3,3$ and depths are $0,2,2$ over $K$, with corresponding values of $\nu_{\fraks}$ given by $0,6,6$.  Again, all proper clusters are odd and principal and the good reduction criterion applies, giving $y^2=x(x^4-1)$ with two elliptic curves $y^2=-x(x^2-1)$ attached.  At $b=-1,\pm i$, the roots similarly form two triples with the same cluster sizes, depths, and $\nu_{\fraks}$-values.  
\end{proof}

We will see a similar calculation using explicit equations in \Cref{lemma: specialisation to (g-2)-dimensional x E^2}.

\subsection{Generically geometrically simple: induction}

We now make our key inductive argument.

\begin{proposition}\label{prop:GeometricallySimpleHigherGenus}
Let $g \geq 8$ be even. Suppose \Cref{thm:MainHigherGenus} holds in genera $g-2$ and $g-4$. Then $A^{(g)}$ has geometric endomorphism algebra $(-1,-1\,|\,\Q)$.
\end{proposition}
As $\Jprim^{(g)}$ is a base-change of $A^{(g)}$, to prove the proposition it suffices to show that $\Jprim^{(g)}$ is geometrically simple with geometric endomorphism algebra $(-1,-1\,|\,\Q)$. 
Note that $\calJprim^{(g)}$ is an abelian scheme over an open subscheme of $\A^d_\Q$, so we can apply~\Cref{prop:SpecialisationArguments} to $\calJprim^{(g)}$.  Adapting the notation of that proposition, an abelian variety of the form $\frakJprim_x$ will be called a \textit{specialization} of $\calJprim^{(g)}$.

We will prove \Cref{prop:GeometricallySimpleHigherGenus} by constructing two specializations of $\calJprim^{(g)}$ whose structure is sufficiently different to force the geometric endomorphism algebra of $\Jprim^{(g)}$ to be $(-1,-1\,|\,\Q)$. These specializations are constructed in the next two lemmas.

\begin{lemma}\label{lemma: specialisation to (g-4)-dimensional x E^4}
Let $g \geq 8$ be even and suppose that \Cref{thm:MainHigherGenus} holds in genus $g-4$. Let $c \in \Q^\times \setminus \{\pm 1\}$. Then there exists a number field $L$ and a specialization $Z$ over $L$ of $\calJprim^{(g)}$ (in the sense defined immediately above) whose isogeny decomposition over $\Q^{\al}$ is given by $X \times (E_c)^4$, where
\begin{itemize}
\item $X$ over $\Q^{\al}$ is simple with $\dim(X)=g-4$ and $\End(X)_\Q \simeq (-1,-1\,|\,\Q)$, and 
\item $E_c$ is the elliptic curve over $\Q^{\al}$ with $j$-invariant $j(c)$ given by \cref{eq: j invariant}.
\end{itemize}
\end{lemma}

\begin{proof}
Let $L$ be a number field and $\overline{b}  \colonequals  (\overline{b_1}, \ldots, \overline{b_{d-2}}) \in L^{d-2}$ be such that the Jacobian of the hyperelliptic curve over $L$ given by
\[
\calC^{(g-4)}_{\overline{b}} :  y^2 = x(x^4-1)\prod_{n=1}^{d-2} \left(x^2-\overline{b_n}^2\right)\left(x^2-1/\overline{b_n}^{2}\right)
\]
is geometrically simple with geometric endomorphism algebra $(-1,-1\,|\,\Q)$. 
Such $\overline{b_i}$ exist by~\Cref{thm:MainHigherGenus} in genus $g-4$, which holds by assumption.
To ease the notation, we will denote $\calC^{(g-4)}_{\overline{b}}$ simply by $\Cbfour$.

Further fix $c \in \Q^\times \setminus \{\pm 1\}$. Consider now the family $\mathcal{X}_c$ of nice curves given by the Cartesian diagram
\[
\xymatrix{
\mathcal{X}_c \ar[r] \ar[d] & \Cprim^{(g)}_L \ar[d] \\
\mathbb{A}_{L}^1  \setminus \Delta' \ar[r] & \BB_L,
}
\]
where the bottom arrow sends $t$ to $(\overline{b_1}, \ldots, \overline{b_{d-2}}, t^2, ct^2)$, and $\Delta'$ is a suitable proper closed subscheme of $\mathbb{A}_{L}^1$ such that $\calX_c$ is smooth over its complement. Concretely, $\mathcal{X}_c$ is the hyperelliptic curve with equation
\[
y^2 = x(x^4-1) (x^2-t^4)(x^2-t^{-4})(x^2-(ct^2)^2)(x^2-(ct^2)^{-2}) \prod_{n=1}^{d-2} \left(x^2-\overline{b_n}^2\right)\left(x^2-1/\overline{b_n}^{2}\right),
\]
where we have identified $\mathbb{A}^1_{L}$ with $\operatorname{Spec} {L}[t]$.

Let $\calA_{\mathcal{X}_c}$ be the Jacobian of $\mathcal{X}_c$, and let $\frakA_{\mathcal{X}_c}$ be its N\'eron model over the Dedekind domain $L[t]$. We claim that the special fibre of $\frakA_{\mathcal{X}_c}$ at $t=0$ is the product of $\Jac(\Cbfour)$ by $A_c^{+} \times A_c^{-}$, where $A_c^{\pm}$ is the Jacobian of the genus 2 hyperelliptic curve with equation $y^2=\pm z(z^2-1)(z^2-c^2)$. To see this, we compute the special fibre of the stable model $\mathcal{Z}_c$ of $\mathcal{X}_c$ at $t=0$. Over the open locus where $t$ is invertible (that is, the complement of the closed subscheme defined by $t$), the curve $\mathcal{X}_c$ is isomorphic to
\[
\mathcal{X}'_c : \, Y^2 = x(x^4-1) (x^2-t^4)(t^4x^2-1)(x^2-(ct^2)^2)(t^4x^2-c^{-2}) \prod_{n=1}^{d-2} \left(x^2-\overline{b_n}^2\right)\left(x^2-1/\overline{b_n}^{2}\right),
\]
with $Y=t^4y$. The curve $\mathcal{X}'_c$ clearly extends over the closed subscheme defined by $(t)$, and its special fibre at $t=0$ is given by
\[
\overline{\Cbfour} : Y^2 = c^{-2} x^4 x(x^4-1) \prod_{n=1}^{d-2} \left(x^2-\overline{b_n}^2\right)\left(x^2-1/\overline{b_n}^{2}\right).
\]
Reabsorbing the square factor $c^{-2} x^4$ into $Y^2$, we see that the normalisation of this curve is simply $\Cbfour$. To understand the other components of the special fibre of the stable model we blow-up the special fibre at $x=0$, $y=0$. Writing $x=t^2z$ and $y=t v/c$ we obtain
\begin{equation}\label{eq:BlowUpOrigin}
\mathcal{X}'_c : \, (v/c)^2 = z(t^8 z^4-1) (z^2-1)(t^8z^2-1)(z^2-c^2)(t^8z^2-c^{-2}) \prod_{n=1}^{d-2} \left(t^4z^2-\overline{b_n}^2\right)\left(t^4z^2-1/\overline{b_n}^{2}\right).
\end{equation}
This maps to the blow-up at the origin of $\mathbb{A}^1_{L[t]}$ along $(x, t^2)$. Let $E$ be the exceptional divisor of the blow-up of $\A^1_{L[t]}$ at $0$; the function $z$ is naturally a coordinate on $E$. Specialising the above equation at $t=0$, we obtain the equation of a double cover of $E$: it is the genus-2 hyperelliptic curve given explicitly by $v^2 = -z(z^2-1)(z^2-c^2)$. It meets $\overline{\Cbfour}$ at the point $x=0$ of the latter.

The automorphism $(x,y) \mapsto (1/x,iy/x^{g+1})$ of the family restricts to the special fibre of the stable model, and acts on $\overline{\Cbfour}$ by sending $x \mapsto 1/x$. In particular, this automorphism sends the genus-$2$ component at $x=0$ to an isomorphic one at $x=\infty$. Note that the isomorphism is only defined over $L(i)$, since this holds for the automorphism (the same result can also be obtained by a direct calculation completely analogous to the above, which also gives the equation for the component at $x=\infty$).
Thus, the stable model $\mathcal{Z}_c$ of the special fibre of $\mathcal{X}_c'$ is given by the union of the three curves $\overline{\Cbfour}$ (with normalisation $\Cbfour$), $v^2=-z(z^2-1)(z^2-c^2)$ and $v^2=z(z^2-1)(z^2-c^{-2})$, meeting as in~\Cref{fig:SpecialFibreStableModel}.

\begin{figure}
\tikzset{every picture/.style={line width=0.75pt}} 
\begin{center}
\begin{tikzpicture}[x=0.75pt,y=0.75pt,yscale=-1,xscale=1]

\draw    (45,66) .. controls (85,36) and (182,98) .. (200,67) ;
\draw    (74,11) .. controls (99,31) and (65,124) .. (95,162) ;
\draw    (166,6) .. controls (191,26) and (158,123) .. (188,161) ;
\draw  [fill={rgb, 255:red, 0; green, 0; blue, 0 }  ,fill opacity=1 ] (81,57.5) .. controls (81,56.12) and (82.12,55) .. (83.5,55) .. controls (84.88,55) and (86,56.12) .. (86,57.5) .. controls (86,58.88) and (84.88,60) .. (83.5,60) .. controls (82.12,60) and (81,58.88) .. (81,57.5) -- cycle ;
\draw  [fill={rgb, 255:red, 0; green, 0; blue, 0 }  ,fill opacity=1 ] (173,76.5) .. controls (173,75.12) and (174.12,74) .. (175.5,74) .. controls (176.88,74) and (178,75.12) .. (178,76.5) .. controls (178,77.88) and (176.88,79) .. (175.5,79) .. controls (174.12,79) and (173,77.88) .. (173,76.5) -- cycle ;

\draw (5,59.4) node [anchor=north west][inner sep=0.75pt]    {$\overline{\Cbfour}$};
\draw (86,62.4) node [anchor=north west][inner sep=0.75pt]    {$0$};
\draw (154,77) node [anchor=north west][inner sep=0.75pt]    {$\infty $};
\draw (-105,120) node [anchor=north west][inner sep=0.75pt]    {$v^{2} =-z\left( z^{2} -1\right)\left( z^{2} -c^{2}\right)$};
\draw (180,30) node [anchor=north west][inner sep=0.75pt]    {$v^{2} =z\left( z^{2} -1\right)\left( z^{2} -c^{-2}\right)$};
\end{tikzpicture}
\caption{The special fibre of the stable model}\label{fig:SpecialFibreStableModel}
\end{center}
\end{figure}

As a consequence (see~\cite[Example 8 on p.~246]{MR1045822}), the special fibre of the N\'eron model $\frakA_{\mathcal X_c}$ at $t=0$ is the Jacobian of the (reducible) stable curve $\mathcal{Z}_c$, hence it is isomorphic to the product of the Jacobians of the curves in~\Cref{fig:SpecialFibreStableModel}. We denote by $A_c^+, A_c^-$ the Jacobians of the two genus-2 components. Notice that $A_c^+, A_c^-$ are isomorphic over $L(i)$.

We claim that $A_c^{\pm}$ is isogenous over $\Qbar$ to $E_c^+ \times E_c^-$, with $E_c^{\pm}$ the elliptic curve 
\begin{equation}\label{eq:Ec}
E_c^{\pm} : w^2 = (u \pm 2\sqrt{c})(u^2-(c+1)^2).
\end{equation}
To prove this it is enough to notice that there are maps
\[
v^2 = -z(z^2-1)(z^2-c^2) \to w^2 = (u \pm 2\sqrt{c})(u^2-(c+1)^2)
\]
given by 
\[
u = z + \frac{c}{z}, \quad w = iv\frac{z \pm \sqrt{c}}{z^2}.
\]
The $j$-invariant of each curve $E_c^{\pm}$ is given by the non-constant rational function in $c$
\begin{equation}\label{eq: j invariant}
j(E_c^{\pm}) = 1728 \frac{(c+1/3)^3(c+3)^3}{(c-1)^4(c+1)^2}.
\end{equation}
In particular, $E_c^{+}$ and $E_c^{-}$ are geometrically isomorphic. 
The lemma follows upon taking $E=E_c^{+, \al} \simeq E_c^{-, \al}$ and $X = \Jac((\Cbfour)^\al)$.
\end{proof}
We now look for a different specialization whose geometric endomorphism algebra is sufficiently different from what we just found.

\begin{lemma}\label{lemma: specialisation to (g-2)-dimensional x E^2}
Let $g \geq 6$ be even and assume that \Cref{thm:MainHigherGenus} holds in genus $g-2$.  Then the following statements hold.
\begin{enumalph}
\item There exist a number field $L$ and a specialization $Z$ over $L$ of $(\Jprim^{(g)})_L$ isomorphic over $L$ to $Y \times E^2$, where $Y$ is a geometrically simple abelian variety over $L$ of dimension $g-2$, with $\End(Y^\al)_\Q \simeq (-1,-1\,|\,\Q)$ and $G_{Y, \ell}^0$ isomorphic to a form of $\mathbb{G}_m \cdot \SO_{g-2}$ for every prime $\ell$, and $E$ is an elliptic curve over $L$ with $j$-invariant $1728$.
\item There exists an injection 
\begin{equation} \label{eqn:minus1jpirmg}
\begin{aligned}
\End(\Jprim^{(g), \al})_\Q &\hookrightarrow (-1,-1\,|\,\Q) \times \operatorname{M}_{2}(\Q(i)) \\
[-1] \in \End(\Jprim^{(g), \al})_\Q &\mapsto \left(-1,\begin{pmatrix} -1 & 0 \\ 0 & -1 \end{pmatrix}\right).
\end{aligned}
\end{equation}
\item The abelian variety $\Jprim^{(g), \al}$ has a simple isogeny factor of dimension at least $g-2$.
\end{enumalph}
\end{lemma}

\begin{proof}
The proof is similar to that of~\Cref{lemma: specialisation to (g-4)-dimensional x E^4}, so we give fewer details.  Let $a \in \Q^{d-1}$ satisfy the conclusion of~\Cref{thm:MainHigherGenus} in genus $g-2$, which we have assumed to hold. We may then choose a number field $L$ containing $\Q(i)$ and a tuple $\overline{\beta} \colonequals  (\overline{\beta_1}, \ldots, \overline{\beta_{d-1}}) \in L^{d-1}$ such that $(\calC_{a}^{(g-2)})_L$ admits the model
\[
\Gamma^{(g-2)} : y^2 = x(x^4-1) \prod_{n=1}^{d-1} \left(x^2-\overline{\beta_n}^2 \right)\left(x^2-\overline{\beta_n}^{-2} \right).
\]
Let $Y/L$ be the Jacobian of $\Gamma^{(g-2)}$. By the choice of $a$, the abelian variety $Y$ is geometrically simple, has geometric endomorphism algebra $\End(Y^\al)_\Q \cong (-1, -1 \mid \Q)$, and the identity component of its monodromy group $G_{Y, \ell}^0$ is isomorphic to a form of $\mathbb{G}_m \cdot \operatorname{SO}_{g-2}$ for every prime $\ell$.

Proceeding as in the proof of~\Cref{lemma: specialisation to (g-4)-dimensional x E^4}, consider the family $\mathcal{Y}$ of nice curves of genus $g$ given by the Cartesian diagram
\[
\xymatrix{
\calY \ar[r] \ar[d] & \Cprim^{(g)}_{L} \ar[d] \\
\mathbb{A}_{L}^1  \setminus \Delta'' \ar[r] & \BB_{L},
}
\]
where the bottom arrow sends $t$ to $(\overline{\beta_1}, \ldots, \overline{\beta_{d-1}}, t^2)$, and $\Delta''$ is a suitable proper closed subscheme of $\mathbb{A}^1_L$ such that $\calY$ is smooth over its complement. Denote by $\frakA_{\calY}$ the N\'eron model over $\mathbb{A}^1_L$ of the Jacobian of $\calY$.
By a calculation completely analogous to that in the proof of~\Cref{lemma: specialisation to (g-4)-dimensional x E^4}, the special fibre at $t=0$ of $\frakA_{\calY}$ is geometrically the product of the simple abelian variety $Y^\al = \Jac(\Cbtwo)^\al$ and $E_{\Qbar}^2$, where $E$ over $L$ is the elliptic curve $y^2=x^3-x$. Notice that $\Jac(\Cbtwo)^\al$, being simple of dimension $g-2 > 1$, has no common isogeny factors with $E^2$. Extending $L$ if necessary, we may assume that the isomorphism $(\frakA_{\calY})_{t=0} \cong \Jac(\Cbtwo) \times E^2$ is defined over $L$, which proves (a).

The desired injection 
\[
\End(\Jprim^{(g), \al})_\Q \hookrightarrow \End(\Jac(\Cbtwo)^\al \times E^2)_\Q \simeq (-1,-1\,|\,\Q) \times \operatorname{Mat}_{2 \times 2}(\Q(i))
\]
is obtained by specialization (\Cref{prop:SpecialisationArguments}), and takes multiplication by $[-1]$ to $\left(-1,\begin{pmatrix} -1 & 0 \\ 0 & -1 \end{pmatrix}\right)$ as in \eqref{eqn:minus1jpirmg}.
Moreover, since any (geometric) splitting of $\Jprim^{(g)}$ would be reflected by a splitting of the geometric special fibre of $\frakA_{\mathcal{Y}}$ at $t=0$ (see~\Cref{rmk:SpecialisationSplittings}), this implies that $\Jprim^{(g), \al}$ has a simple isogeny factor of dimension at least $g-2$.
\end{proof}

We are now ready to prove~\Cref{prop:GeometricallySimpleHigherGenus}.

\begin{proof}[Proof of~\Cref{prop:GeometricallySimpleHigherGenus}]
By~\Cref{lemma: specialisation to (g-2)-dimensional x E^2}, $\Jprim^{(g), \al}$ has a simple isogeny factor of dimension at least $g-2$:
\begin{equation}  \label{eqn:Bxy}
B^{(g),\textup{al}} \sim X \times Y 
\end{equation}
where $X$ is simple over $\calL^\al$ with $\dim X = g-2,g-1,g$ (and $\dim(X)+\dim(Y)=g$, but $Y$ may not be simple).  For dimension reasons ($g \geq 8$) and since $X$ is simple, there is no isogeny factor in common between $X$ and $Y$, so 
\[ \End(\Jprim^{(g),\al})_\Q \simeq \End(X)_\Q \times \End(Y)_\Q. \]

Since $\mathcal{O} \subset (-1,-1\,|\,\Q) \hookrightarrow (\End B^{(g),\textup{al}})_\Q$ and since $(-1,-1\,|\,\Q)$ acts nontrivially on every subspace in its tangent representation (see the proof of~\Cref{lem:jacsimp}), there is an embedding of $(-1,-1\,|\,\Q)$ in both $\End(X)_\Q$ and in $\End(Y)_\Q$. 
There is no nontrivial action of $(-1,-1\,|\,\Q)$ on an elliptic curve in characteristic $0$, so to conclude that $\Jprim^{(g)}$ is geometrically simple, we assume for purposes of contradiction that $\dim(X)=g-2$ in \eqref{eqn:Bxy}.  Again, we have $(-1,-1\,|\,\Q) \hookrightarrow (\End Y)_\Q$; but over an algebraically closed field of characteristic $0$, the only abelian surfaces $Y$ (up to isogeny) that admit such an action are of the form $Y \sim E_{\operatorname{CM}}^2$, where $E_{\rm CM}$ is a CM elliptic curve.  Thus, $E_{\rm CM}$ is a (geometric) isogeny factor of any specialization of $\Jprim^{(g)}$ (see~\Cref{rmk:SpecialisationSplittings}, and notice that $E_{\rm CM}$ can only specialize to $E_{\rm CM}$, since it is already defined over a number field). 

However, for each $c \in \Q^\times \setminus \{\pm 1\}$, \Cref{lemma: specialisation to (g-4)-dimensional x E^4} gives a specialization of $\Jprim^{(g)}$ which is geometrically isogenous to $X \times E_c^4$, with $X$ (geometrically) simple of dimension $g-4>1$ and where $E_c$ has $j$-invariant given by \cref{eq: j invariant}. Choosing $c$ in such a way that $E_c$ is not geometrically isogenous to a CM elliptic curve we obtain a contradiction. Such $c$ are easily seen to exist: for example, $c=2$ will do, because $j(E_2)$ is not an integer by \cref{eq: j invariant}, so $E_2$ does not have potential complex multiplication and is therefore not (potentially) isogenous to $E_{\rm CM}$. This proves as desired that $\Jprim^{(g)}$ is geometrically simple.

Knowing this, we conclude the proof as follows.  We have $\Q$-algebra maps
\[ \End(\Jprim^{(g),\al})_\Q \hookrightarrow (-1,-1\,|\,\Q) \times \M_2(\Q(i)) \to (-1,-1\,|\,\Q); \]
the first map comes from the specialization in \Cref{lemma: specialisation to (g-2)-dimensional x E^2}(b), the second by projection onto the first factor.  The composition of these maps is nonzero, since \Cref{lemma: specialisation to (g-2)-dimensional x E^2}(b) provides that $[-1] \mapsto (-1,-\operatorname{Id}_2) \mapsto -1$.  Since $\End(\Jprim^{(g),\al})_\Q$ is a simple $\Q$-algebra, this composition is injective.  But by construction $(-1,-1\,|\,\Q) \hookrightarrow \End(\Jprim^{(g),\al})_\Q$, so by dimensions this must be an isomorphism.
\end{proof}

\subsection{Proof of generic monodromy}

We now have all the ingredients to prove the main result of this section,~\Cref{thm:GenericMonodromy}, conditionally on~\Cref{thm:MainHigherGenus} holding for genera $g-2$ and $g-4$. We will also prove~\Cref{thm:GenericMonodromy} for $g=4, 6$ without relying on~\Cref{thm:MainHigherGenus}.

\begin{proposition}
If the statement of \Cref{thm:MainHigherGenus} holds for $g-2, g-4$, then \Cref{thm:GenericMonodromy} holds for $g$.  
\end{proposition}

\begin{proof}
The first part follows from~\Cref{lemma: geometric endomorphism algebra in genera 4 and 6} (for $g \leq 6$) and~\Cref{prop:GeometricallySimpleHigherGenus} (for $g \geq 8$). 
Consider now the specialization $Z \simeq_L Y \times E^2$ of $(\Jprim^{(g)})_L$ given by~\Cref{lemma: specialisation to (g-2)-dimensional x E^2}. By~\Cref{prop:SpecialisationArguments}, the image of the $\ell$-adic Galois representation attached to $Z_{L}$ is contained in the image of the $\ell$-adic Galois representation attached to $\Jprim^{(g)}$ (hence the same holds for the Zariski closures). By~\cite[Lemma 3.4]{MR3494170}, since $E$ is a CM elliptic curve, the rank of the $\ell$-adic monodromy group of $Z \cong Y \times E^2$ is equal to
\[
\left( \operatorname{rank} (\calG^0_{Y, \ell})-1 \right) + \left( \operatorname{rank} (\calG^0_{E, \ell})-1 \right) + 1,
\]
where the summands $\pm 1$ take into account the torus of homotheties. 
By~\Cref{lemma: specialisation to (g-2)-dimensional x E^2}(a) we have $\operatorname{rank} (\calG^0_{Y, \ell}) = \operatorname{rank} (\mathbb{G}_m \cdot \SO_{g-2}) = 1 +\frac{g-2}{2}$, and since $E$ is CM it is well-known that $\operatorname{rank} \calG_{E, \ell}^0 = 2$, so we obtain that the rank of $\calG_{\Jprim^{(g)}, \ell}^0$ 
is at least $g/2+1$. On the other hand, by~\Cref{lemma:Lefschetz} we know that $L(\Jprim^{(g)})$ has rank $g/2$.
Thus $\calG_{\Jprim^{(g)}, \ell}^0$ and $( L(\Jprim^{(g)}) \cdot \mathbb{G}_m)_{\Q_\ell}$ are connected reductive groups of the same rank, with the former contained in the latter, and they have the same centralizer (in both cases this is $\End(\Jprim^{(g),\al})_{\Q_\ell} = (-1,-1\,|\,\Q) \otimes_{\Q} \Q_\ell$: for $\calG^0_{\Jprim^{(g)}, \ell}$ this follows from Faltings's proof of Tate's conjecture, for $L(\Jprim^{(g)}) \cdot \mathbb{G}_m$ it follows from the definition and the double centralizer theorem). By~\cite[Lemma 7]{MR1944805}, this implies $\calG_{\Jprim^{(g)}, \ell}^0 = (L(\Jprim^{(g)}) \cdot \mathbb{G}_m)_{\Q_\ell}$, and since $\calG_{\Jprim^{(g)}, \ell}^0 \subseteq  \MT(\Jprim^{(g)})_{\Q_\ell} \subseteq (L(\Jprim^{(g)}) \cdot \mathbb{G}_m)_{\Q_\ell}$ we obtain the final statement of the theorem, i.e., $\Jprim^{(g)}$ (and hence $A^{(g)}$) is fully of Lefschetz type. 
\end{proof}

\subsection{Variation of the Galois image in the family $\calA_\aunder$}\label{sect:VariationOfGaloisInFamilies}

In this section we conclude the inductive step by showing that the truth of~\Cref{thm:GenericMonodromy} for a certain genus $g$ implies the truth of~\Cref{thm:MainHigherGenus} for the same genus.
We apply~\Cref{thm:ZywinaFamilies} to the abelian scheme $\calJ^{(g)} \to \AAp$.
Let $c$ be the corresponding constant. The set
\[
X_g = \{ \aunder \in \AAp(\Q) : [\rho_{\calJ^{(g)}}( \pi_1(\AAp,\overline{\eta}) )  : \rho_{\calA^{(g)}_\aunder} (\Gal_\Q) ] \leq c \}
\]
has full density inside $\AAp(\Q)$ in the sense of~\cref{eq:FullDensity}.

\begin{lemma}\label{lemma:FromGalRepToEndoAlgebra}
Suppose that~\Cref{thm:GenericMonodromy} holds in genus $g$.  Then for every $\aunder \in X_g$ we have the following:
\begin{enumalph}
\item $\End(\calA_\aunder^\al)_\Q = (-1,-1\,|\,\Q)$;
\item $\Hg(\calA_\aunder)_\C \cong \SO_{g,\C}$; and
\item $\calA_\aunder$ satisfies the Mumford--Tate conjecture.
\end{enumalph}
\end{lemma}

\begin{proof}
It is easy to see that, given two subgroups $G_1 \subseteq G_2 \subseteq \GL_{2g}(\Q_\ell)$ with $[G_2 : G_1] < \infty$, the Zariski closures of $G_1, G_2$ in $\GL_{2g, \Q_\ell}$ have the same dimension (consider their Lie algebras).
As a consequence, the defining property of $X_g$ shows that, for every $a \in X_g$ and every prime $\ell$, the $\ell$-adic monodromy group of $\calA_a$ has the same dimension as the $\ell$-adic monodromy group of the generic fibre $A^{(g)}$. In particular, both groups have the same connected component of the identity.
By Tate's conjecture on endomorphisms for abelian varieties over finitely generated fields of characteristic 0 (proved by Faltings~\cite{MR766574}), the connected component of the $\ell$-adic monodromy group determines the dimension of the geometric endomorphism algebra. More explicitly, if $A$ is an abelian variety over a field $K$ of characteristic $0$ and $\ell$ is a prime number, we have
\[
\End(A^{\al})_{\Q_\ell} \simeq \End_{G_{A,\ell}^0}(T_\ell A) \otimes_{\Z_\ell} \Q_\ell.
\]
It follows from the above that $\dim_{\Q} \End(\calA_a^{\al})_\Q = \dim_{\Q_\ell} \End(\calA_a^{\al})_{\Q_\ell} = \dim_{\Q_\ell} \End(A^{(g),\al})_{\Q_\ell} = 4$, where the last equality follows from the fact that $\End(A^{(g),\al})_\Q \simeq (-1,-1\,|\,\Q)$ by~\Cref{thm:GenericMonodromy}, which we are assuming. By~\Cref{lemma: endomorphisms generic fibre} there is an injection of $(-1,-1\,|\,\Q)$ in $\End(\calA_\aunder^{\al})_\Q$, so we obtain the desired equality $\End(\calA_\aunder^{\al})_\Q = (-1,-1\,|\,\Q)$, that is, (a).

Now observe that $\calA_\aunder$ has the same polarization and geometric endomorphism algebra, and therefore the same Lefschetz group, as the generic fibre $A^{(g)}$. 
By definition of the set $X_g$, the image $\rho_{\calA_\aunder, \ell}(\Gal_{\Q})$ is open inside $\rho_{\calA^{(g)},\ell}( \pi_1(\AAp,\overline{\eta}) ) = \rho_{A^{(g)}, \ell}(\Gal_{\calK})$ (see~\Cref{lem:specialization} for the last equality). By~\Cref{thm:GenericMonodromy}, $A^{(g)}$ satisfies the Mumford--Tate conjecture, so $\rho_{A^{(g)}, \ell}(\Gal_{\calK}) \cap \MT(\calA_\aunder)(\Q_\ell)$ is open inside the $\Q_\ell$-points of $L(A^{(g)}) \cdot \mathbb{G}_{m} = L(\calA_\aunder)\cdot \mathbb{G}_{m}$. By~\Cref{thm:ContainmentsGlMTLefschetz}(a)--(b) we also have the containments
\[
\rho_{\calA_\aunder, \ell}(\Gal_{\Q}) \cap \MT(\calA_\aunder)(\Q_\ell) \subseteq \MT(\calA_\aunder)(\Q_\ell) \subseteq (L(\calA_\aunder) \cdot \mathbb{G}_m) (\Q_\ell).
\]
It follows from the above discussion that these inclusions have open image.  In particular, $\MT(\calA_\aunder)$ and $L(\calA_\aunder)\cdot \mathbb{G}_{m}$ have the same dimension, so the inclusion $\Hg(\calA_\aunder) \subseteq L(\calA_\aunder)$ is an equality. 
Since $L(\calA_\aunder)_{\mathbb{C}}=L(A^{(g)})_{\mathbb{C}}$ is isomorphic to $\SO_{g,\C}$ by~\Cref{lemma:Lefschetz}, this proves (b). 

Finally, the Zariski closure of $\rho_{\calA_\aunder, \ell}(\Gal_{\Q}) \cap \MT(\calA_\aunder)(\Q_\ell)$ is $\MT(\calA_\aunder)_{\Q_\ell}$ by the openness just proved, so the Mumford--Tate conjecture holds at the prime $\ell$. By~\cite[Theorem 4.3]{MR1339927}, this implies that Mumford--Tate holds for all primes, that is, (c). 
\end{proof}

We now give a proof of~\Cref{thm:GenericMonodromy} for $g=4, 6$ that does not rely on~\Cref{thm:MainHigherGenus}: this will establish the base cases of our induction to show~\Cref{thm:MainHigherGenus}.
\begin{lemma}\label{lemma: thm generic monodromy holds for 4 and 6}
\Cref{thm:GenericMonodromy} holds for $g=4, 6$.
\end{lemma}
\begin{proof}
The first statement in \Cref{thm:GenericMonodromy} follows from~\Cref{lemma: geometric endomorphism algebra in genera 4 and 6}. An abelian variety of dimension 4 or 6 with geometric endomorphism algebra $(-1,-1\,|\,\Q)$ is fully of Lefschetz type, see~\cite{MR1324634} for $g=4$, and~\cite{MR2663452} for $g=6$.
\end{proof}

We are finally ready to prove~\Cref{thm:MainHigherGenus}. 

\begin{proof}[Proof of~\Cref{thm:MainHigherGenus}]
The proof is by induction, see~\Cref{fig:flow}.
We have shown in~\cref{subsec: generic monodromy} that~\Cref{thm:MainHigherGenus} in genera $g-2$ and $g-4$ implies~\Cref{thm:GenericMonodromy} in genus $g$. Since~\Cref{lemma: thm generic monodromy holds for 4 and 6} takes care of the base cases $g=4$ and $g=6$ of~\Cref{thm:GenericMonodromy}, it suffices to show that~\Cref{thm:GenericMonodromy} in genus $g$ implies~\Cref{thm:MainHigherGenus} in genus $g$.

From~\Cref{thm:ZywinaFamilies}, \Cref{lemma:FromGalRepToEndoAlgebra}, and the definition of the set $X_g$ it follows that for a density $1$ subset of points $\aunder \in \AAp(\Q)$ when ordered by height all of the following hold: the curve $\calC_\aunder^{(g)}$ is smooth, $a_d$ is nonzero, the Jacobian $\calA_\aunder^{(g)}$ has geometric endomorphism algebra $(-1,-1\,|\,\Q)$, satisfies the Mumford--Tate conjecture, and has complex Hodge group isomorphic to $\SO_{g, \C}$. For any such $\aunder$, properties (i) and (ii) in~\Cref{thm:MainHigherGenus} follow from~\Cref{lemma:FromEndoAlgebraToEndoRing}, 
(iii) and (iv) follow from~\Cref{thm:UpperBoundSatoTate}, and (v) is clear. Finally, the last part of the statement is a consequence of~\Cref{cor:InfinitelyManyIsogenyClasses} below.
\end{proof}

\subsection{Infinitely many isogeny classes}

We conclude by showing the final statement in~\Cref{thm:MainHigherGenus}.

\begin{proposition}\label{prop:InfinitelyManyIsogenyClasses}
Let $S$ be a variety defined over a number field $K$, and let $\calC \to S$ be a family of smooth, projective, geometrically connected curves of genus $g$. Suppose that the induced moduli map $S(\Kbar) \to M_g(\Kbar)$ has finite fibres. The map
\[
\begin{array}{ccc}
S(K) & \to & \{ \text{abelian varieties over }\Kbar \} / \sim \\
s & \mapsto & \operatorname{Jac}(\calC_s)^\al,
\end{array}
\]
where $\sim$ denotes isogeny over $\Kbar$, has finite fibres. In particular, if $S(K)$ is an infinite set, the abelian varieties $\{\Jac(\calC_s) : s \in S(K) \}$ fall into infinitely many distinct isogeny classes over $\Kbar$.
\end{proposition}

\begin{proof}
For the sake of simplicity write $A_s$ for $\Jac(\calC_s)$. It is an abelian variety over $K$.
By contradiction, let $s_0 \in S(K)$ be such that $\calT(s_0) \colonequals \{s \in S(K) : A_s \sim_{\Kbar} A_{s_0} \}$ is an infinite set. By~\cite{MR1154704}, every homomorphism between $A_s$ and $A_{s_0}$ is defined over the field $F_s \colonequals K(A_s[3], A_{s_0}[3])$, whose degree over $\Q$ is bounded independently of $s_0$ and $s$. Hence, for all $s \in \calT(s_0)$, the abelian varieties $A_s$ and $A_{s_0}$ are isogenous over $F_s$, and by the explicit isogeny theorem~\cite[Th\'eor\`eme 1.4]{MR3263028}, there exists an $F_s$-isogeny $A_{s_0} \to A_s$ of degree at most $b( h(A_{s_0}), [F_s:\Q], g )$, where $b( h(A_{s_0}), g, [F_s:\Q] )$ depends only on the stable Faltings height $h(A_{s_0})$ of $A_{s_0}$, on the dimension $g$, and on the degree $[F_s : \Q]$. Since $h(A_{s_0})$, $g$ and $[F_s:\Q]$ are bounded independently of $s \in \calT(s_0)$, we obtain that there exists an integer $N=N(s_0)$ such that for every $s \in \calT(s_0)$ there exists a $\Kbar$-isogeny $A_{s_0} \to A_s$ of degree at most $N$. Since $A_{s_0}$ has only finitely many subgroups of order at most $N$, the infinitely many abelian varieties $A_s^\al$ for $s \in \calT(s_0)$ fall into finitely many isomorphism classes. Since every abelian variety carries at most finitely many isomorphism classes of principal polarizations~\cite{MR653950}, we see that infinitely many $s \in \calT(s_0)$ give rise to (geometrically) isomorphic \textit{principally polarized} abelian varieties $A_s$. Let $\calT'(s_0)$ be an infinite subset of $\calT(s_0)$ for which $A_s, A_{s'}$ are geometrically isomorphic principally polarized abelian varieties for all $s, s' \in \calT'(s_0)$.
By Torelli's theorem, the curves $\calC_s$ are then geometrically isomorphic for all $s \in \calT'(s_0)$. This contradicts the assumption that the moduli map $S(\Kbar) \to M_g(K^\al)$ has finite fibres and finishes the proof of the proposition.
\end{proof}

\begin{remark}
For a given $s_0 \in S(K)$ one could ask for a bound on the cardinality of the set 
\[
\{ s \in S(K) : \Jac (\calC_s) \sim_{\Kbar} \Jac(\calC_{s_0}) \}.
\]
 Following the proof above, one would need to bound the number of (isomorphism classes of) principal polarizations on the abelian varieties $\Jac (\calC_s)$. As shown in~\cite{MR904944}, there is no uniform bound on this quantity depending only on the genus of $\calC$, and in fact, Rotger~\cite{MR1998611} proved that for every positive integer $N$ there are $N$ non-isomorphic genus-2 curves with isomorphic unpolarized Jacobian.
\end{remark}

The following corollary implies the last statement in~\Cref{thm:MainHigherGenus}.

\begin{corollary}\label{cor:InfinitelyManyIsogenyClasses}
As $\aunder = (a_1,\ldots,a_d)$ varies in $\AAp(\Q)$, the Jacobians $\calA_\aunder$ over $\Q$ of the curves $\calC_\aunder$
fall into infinitely many distinct isogeny classes over $\Q^{\al}$.
\end{corollary}

\begin{proof}
Follows from~\Cref{prop:InfinitelyManyIsogenyClasses} and~\Cref{lemma:MapToModuliSpace}(b) (finite-to-one).
\end{proof}

\begin{remark} \label{rmk:oomanybadred}
Let $g>4$ and $d=g/2-1$. We sketch a different proof that the Jacobians $\calA_\aunder$ fall into infinitely many distinct isogeny classes over $\Q^{\al}$. 
It is straightforward to see that for every prime $p$ sufficiently large relative to $d$ there are distinct rational numbers $b_1,\ldots,b_d$ different from $0, \pm 1$ such that $v_p(b_1-b_2)>0$, but
$
v_p(b_1 - r) = 0
$ for all roots $r \neq b_1, b_2$ of the separable polynomial defining
\eqref{eq:HyperellipticEquationGeneralgFactored}.
We take $\aunder=(a_1,\ldots,a_d)=\phi(b_1,\dots,b_d)$ as in \cref{subsect:BaseChanges} (just expanding out the polynomial).  Then by the criterion given by the cluster picture \cite[Theorem 5.5]{best2021users}, the Jacobian $\calA_\aunder$ does \emph{not} have potentially good reduction at $p$: the cluster $\{ x \in \overline{\Q_p} : v_p(x-b_1) \geq 1 \}$ is even.  

It follows that for all sufficiently large primes $p$, there exists a specialization $\calA_\aunder$ that does not have potentially good reduction at $p$.  Since the set of places of potential good reduction is a $\Q^{\al}$-isogeny invariant, the Jacobians $\calA_\aunder$ must belong to infinitely many distinct $\Q^{\al}$-isogeny classes.  (When $g=4$, an analogue of the previous argument is prevented by \Cref{lem:puregood}.)
\end{remark}

\section{Moduli interpretation and final remarks}\label{sec: a Shimura variety}

We conclude by relating the family $\calC^{(4)}$ of~\Cref{setup} to moduli.  We begin in \cref{subsec:complexunif} by a complex analytic description of the moduli space; we then represent the moduli functor in \cref{subsec: moduli space} and compare it to our family in \cref{sec:explrec}.  

\subsection{Complex uniformization} \label{subsec:complexunif}

We begin by relating our family to the analytic (complex) uniformization, following the notation in Birkenhake--Lange~\cite[\S 9.5]{MR2062673}.  (See also Krieg \cite[Chapter II]{Krieg:quat} more generally for quotients of quaternionic upper half-space.) 

The dimension of the moduli space of principally polarized complex abelian fourfolds equipped with definite quaternion multiplication by $\mathcal{O}$ (and a fixed polarization) is $1=(e/2)m(m-1)$ \cite[p.~261]{MR2062673} (where $e=1$ as the degree of the base field $\Q$ and $m=g/(2e)=2$).  The complex analytic moduli space is therefore obtained as the quotient of the quaternionic half-space of complex dimension $1$; this may be identified with the more familiar complex upper half-plane \cite[Example 9.5.6]{MR2062673}.  

In the rest of this section, we work out more precisely the rest of the analytic description.  Let $\varphi \colon (-1,-1\,|\,\Q)_{\mathbb{R}}^2 \to \Lambda \otimes_{\mathbb{Z}} \mathbb{R}$ be an $(-1,-1\,|\,\Q)$-equivariant isomorphism of $\mathbb{R}$-vector spaces, for example $\varphi\begin{pmatrix}
1 \\ 0
\end{pmatrix} = e_1 \otimes 1$ and $\varphi \begin{pmatrix}
0 \\ 1
\end{pmatrix} = e_5 \otimes 1$, where $e_1, \ldots, e_8$ is the basis of $\Lambda$ given by the columns of the period matrix $\Pi$ from~\cref{sec: complex structure}. 

We consider the pullback to $(-1,-1\,|\,\Q)^2_{\mathbb{R}}$ of the (imaginary part of the) polarization on $\Lambda \otimes_{\mathbb{Z}} \mathbb{R}$, that is, the intersection form given in~\cref{eq:IntersectionForm}. There exists $T \in \operatorname{Mat}_2((-1,-1\,|\,\Q)_{\mathbb{R}})$ such that
\[
\operatorname{trd} ( {}^t \underline{a} T \underline{b}') = (\operatorname{Im} H)(\varphi(\underline{a}), \varphi(\underline{b}) ) \quad \forall \underline{a}, \underline{b} \in (-1,-1\,|\,\Q)_{\mathbb{R}}^2,
\]
where $\operatorname{trd}$ denotes the reduced trace,
$\operatorname{Im} H$ is the anti-symmetric bilinear form on $\Lambda \otimes_{\mathbb{Z}} \mathbb{R}$ corresponding to the imaginary part of the polarization, and for a vector $\underline{a}=\begin{pmatrix}
h_1 \\ h_2 
\end{pmatrix} \in (-1,-1\,|\,\Q)_{\mathbb{R}}^2$ we have set $
\underline{a}' = \begin{pmatrix}
\overline{h_1} \\ \overline{h_2}
\end{pmatrix}$.
For every $\alpha \in (-1,-1\,|\,\Q)_{\mathbb{R}}$ we have
\[
2\alpha = \trd(\alpha) + \trd(-\alpha i) i + \trd(-\alpha j) j+ \trd(-\alpha k) k;
\]
applying this identity to the $(1,1)$-coefficient of $T$, denoted by $t_{11}$, and using the fact that we have $
\trd ( {}^t\underline{a} T \underline{b}') = (\operatorname{Im} H)(\varphi(\underline{a}), \varphi(\underline{b}) )$, we obtain
\[
\begin{aligned}
2t_{11} &=  (\operatorname{Im} H) \left( \varphi \begin{pmatrix}
1 \\ 0
\end{pmatrix}, \varphi \begin{pmatrix}
1 \\ 0
\end{pmatrix} \right) + (\operatorname{Im} H) \left( \varphi \begin{pmatrix}
1 \\ 0
\end{pmatrix}, \varphi \begin{pmatrix}
i \\ 0
\end{pmatrix} \right)i \\
 & + (\operatorname{Im} H) \left( \varphi \begin{pmatrix}
1 \\ 0
\end{pmatrix}, \varphi \begin{pmatrix}
j \\ 0
\end{pmatrix} \right)j + (\operatorname{Im} H) \left( \varphi \begin{pmatrix}
1 \\ 0
\end{pmatrix}, \varphi \begin{pmatrix}
k \\ 0
\end{pmatrix} \right)k
\end{aligned}.
\]
Using the explicit matrix representations of the imaginary part of the polarization \cref{eq:IntersectionForm} and of the action of $i$ and $j$ \cref{eq:alpha*beta*} we obtain $2t_{11}=-(i+j)$. Repeating the same calculation for the other coefficients of $T$ yields
\begin{equation}
T = \frac{1}{2} \begin{pmatrix} 
 -(i+j) & 0 \\
 0 & -k \\
\end{pmatrix}.
\end{equation}
We observe that $T$ satisfies $T' = -T$, as predicted by the general theory, where $'$ denotes the conjugate transpose.  We have $\nrd(T)=\nrd(-(i+j)/2)\nrd(k/2)=1/8$, which is not a rational square \cite[Exercise 9.10 (2)(a)]{MR2062673}.

\subsection{Moduli space}\label{subsec: moduli space}

We now represent the associated moduli functor of abelian varieties.  We follow Lan~\cite[\S 1.2]{MR3186092}, including his notation.

As above, let $B \colonequals \quat{-1,-1}{\Q}$ and let $\mathcal O \colonequals \Z + \Z i + \Z j + \Z ij \subset B$ be the Lipschitz order.  
Let $L \simeq \Z^8$ be the (left) $\mathcal O$-lattice defined by the action in \eqref{eq:alpha*beta*}; equip $L$ with the symplectic (nondegenerate) pairing $\langle \,,\, \rangle$ induced by the canonical principal polarization of $A$, as in \eqref{eq:IntersectionForm}.  For the period matrix $\Pi$ in \eqref{eqn:zupzup}, 
the unique matrix $J \in \M_8(\R)$ such that $\Pi J = i \Pi$ defining the complex structure extends to a unique $\R$-algebra homomorphism $h \colon \C \to \End_{\calO \otimes \R}(L \otimes \R)$ with $h(i)=J$.  The tuple $(L,\langle \,,\,\rangle,h)$ is called a \emph{PEL-type $\mathcal{O}$-lattice} (and elsewhere, an \emph{integral PEL datum}, taking the standard involution as the positive involution of $\mathcal{O}$); the needed compatibility and positivity conditions follow from the canonical principal polarization of $A$.  As Lan elaborates \cite[\S 5.1.1--5.1.2]{Lanexample}, a PEL-type $\mathcal{O}$-lattice yields a complex abelian variety with PEL (polarization, endomorphism, and level) structure: in our case, 
\begin{itemize}
\item the abelian variety $A_0 \colonequals (L \otimes_\Z \R)/L$ whose complex structure is given by $h$, 
\item the principal polarization $\lambda_0 \colon A_0 \to (L \otimes_\Z \R)/L^\#$ where $L^{\#} \simeq L$ is the dual under $\langle \,,\,\rangle$,
\item endomorphisms $\mathcal{O} \to \End(A_0)$, and
\item full level $2$-structure given by $(\frac{1}{2}L)/L \xrightarrow{\sim} A_0[2]$.  
\end{itemize}

Moving from one object to the general case, we define an associated moduli problem $\mathcal{M}_2(L,\langle\,,\,\rangle,h)$ following Lan~\cite[Definition 1.4.1.2]{MR3186092} as follows (taking for $\square$ the set of odd primes).  Let $S_{0} \colonequals \Spec (\Z[1/2])$.  Define the category fibred in groupoids over the category $\mathsf{Sch}_{S_{0}}$ of schemes over $S_0$, whose fibre over each scheme $S$ is the groupoid $\mathcal{M}_2(L,\langle\,,\,\rangle,h)(S)$ described as follows.  The objects of $\mathcal{M}_2(L,\langle\,,\,\rangle,h)(S)$ are tuples $(A,\lambda,\iota,\alpha_2)$, where:
\begin{enumerate}
    \item $A$ is an abelian scheme over $S$;
    \item $\lambda \colon A\to A^{\vee}$ is a $\Z[1/2]^{\times}$-polarization of $A$ \cite[Definition 1.3.2.19]{MR3186092}; 
    \item $\iota \colon \mathcal O\to \End_S(A)$ defines an $\mathcal O$-structure of $(A,\lambda)$;
    \item $\underline{\Lie}_{A/S}$ with its $\mathcal O \otimes_{\Z} \Z[1/2]$-module structure given naturally by $\iota$ satisfies the Kottwitz determinantal condition \cite[Definition 1.3.4.1]{MR3186092} given by $(L\otimes_{\Z} \R, \langle\,,\,\rangle,h)$;
    \item $\alpha_2 \colon (L/2L)_S \xrightarrow{\sim} A[2]$ is an (integral) principal level-$2$ structure of $(A,\lambda,\iota)$ of type $L\otimes_{\Z} \Z_{2},\langle\,,\,\rangle$.
\end{enumerate}
The reason for the polarization structure as it is given in (2) is a bit technical, see \cite[\S 1.3.1]{MR3186092}.  However, the degree of the polarization $\lambda$ must be the index $[L^\#:L]$; here, we have $[L^\#:L]=1$ so the polarizations in (2) are principal.  The notion of isomorphism \cite[Definition 1.4.1.2]{MR3186092} is omitted.  

\begin{proposition} \label{prop:M2Ld}
The moduli problem $\mathcal{M}_2(L,\langle\,,\,\rangle,h)$ is a Deligne--Mumford stack of relative dimension $1$ over $S_0$ which is separated, smooth, and of finite type over $S_0$.  
\end{proposition}

\begin{proof}
As in Lan \cite[(1.2.5.1)]{MR3186092}, we have a decomposition $L \otimes_{\mathbb{Z}} \mathbb{C} \simeq V_0 \oplus \overline{V_0}$ for a certain $\mathbb{C}$-submodule $V_0$ of dimension equal to $\frac{1}{2} \operatorname{rk}_{\mathbb{Z}} L=4$.  We check the definition \cite[Definition 1.2.5.4]{MR3186092}, and show that for every $\sigma \in \operatorname{Aut}(\mathbb{C}\,|\,\Q)$, we have an isomorphism $V_0 \simeq V_0 \otimes_{\mathbb{C}, \sigma} \mathbb{C}$ of $B \otimes_{\Q} \mathbb{C}$-modules. Indeed, as $B \otimes_{\Q} \mathbb{C} \simeq \operatorname{M}_{2}(\mathbb{C})$, every $B \otimes_{\Q} \mathbb{C}$-module is a direct sum of copies of the standard module $\mathbb{C}^2$. Since $V_0$ and $V_0 \otimes_{\mathbb{C}, \sigma} \mathbb{C}$ have the same dimension over $\mathbb{C}$, they are both isomorphic to $(\mathbb{C}^2)^{(\dim V_0)/2} =(\C^2)^2$.

Lan also gives an explicit comparison \cite[Lemma 2.5.6]{Lan:comparison} to the moduli problem over the complex numbers, defined analytically as in Birkenhake--Lange as a normal complex analytic space (good complex orbifold) as in \cref{subsec:complexunif}: it therefore has dimension $1$ as a stack.

The result now follows from~\cite[Corollary 1.4.1.12]{MR3186092}.
\end{proof}

\begin{remark}
Moduli stacks of abelian varieties with PEL structure are also Shimura stacks (see Deligne~\cite{MR0498581} and Milne~\cite[\S 8]{MilneShimuraVarieties}): for the explicit description of the Mumford--Tate datum $(G,X)$, we refer also to Abdulali~\cite[\S 4]{MR1715177}.
\end{remark}

\begin{remark}
Although Lan makes an assumption (Condition 1.2.5) in the analytic comparison \cite{Lan:comparison} which does not hold in our case---the lattice $L$ does not have an action by the Hurwitz order (more generally see \Cref{rmk:lackoffree})---this assumption is only used in the treatment of the toroidal compactification, and so is not relevant for the above result.
\end{remark}

\subsection{Explicit recognition} \label{sec:explrec}

We now recall \cref{subsect:BaseChanges}, in particular the family of nice curves defined by \cref{eqn:rootsbi}.  We saw in \Cref{prop:goodreductiong4} that the N\'eron model of the pullback of $\calJprim^{(4)}$ to $T \colon t^2=b(b^4-1)$ is an abelian scheme of dimension $4$ over $T$.  Over $\Q(i)$, the polarization, $\calO$-action, and full level $2$-structure extend as well, which gives an object of $\mathcal{M}_2(L,\langle\,,\,\rangle,h)(T)$ and so defines a morphism
\[ T_{\Q(i)} \to \mathcal{M}_2(L,\langle\,,\,\rangle,h)_{\Q(i)}. \]
Every object parametrized by $\mathcal{M}_2(L,\langle\,,\,\rangle,h)$ has as automorphism $-1$.  We let 
\[ \mathcal{X} \colonequals \mathcal{M}_2(L,\langle\,,\,\rangle,h) /\!\!/ \{\pm 1\} \]
be the rigidification, a stacky curve, and let \begin{equation}
\begin{aligned} \label{eqn:pp1qi}
\varsigma \colon T_{\Q(i)} \to \mathcal{X}_{\Q(i)}
\end{aligned}
\end{equation}
be the composition of the moduli map with rigidification.  
For properties of stacky curves, we refer to Voight--Zureick-Brown \cite[Chapter 5]{VZB}.

The hyperelliptic involution $\tau(b,t)=(b,-t)$ on $T$ defines a stack quotient $[T/\langle \tau \rangle]$ with coarse space $\mathbb{P}^1_b$, trivial generic stabilizer, and six stacky points of order $2$ at the points with $b=0,\pm 1,\pm i,\infty$: locally at each such point it is of the form $[\Spec k[[s]]/\mu_2]$.  

\begin{theorem}\label{thm: from the moduli space to the Neron model}
The following statements hold.
\begin{enumalph}
\item The map $\varsigma$ in \eqref{eqn:pp1qi} is finite \'etale of degree $2$ over its image.
\item The map $\varsigma$ factors through an isomorphism from $[T/\langle \tau \rangle]_{\Q(i)}$ onto a component of $\mathcal{X}_{\Q(i)}$.
\end{enumalph}
\end{theorem}

\begin{proof}
For part (a), we recall the calculation in the proof of \Cref{lemma:MapToModuliSpace}(c).  Here it is simpler, as the curve is still defined up to isomorphism by the parameter $b$ (giving a degree $2$ map), and on each smooth fibre we have marked the $10$ Weierstrass points, so the only possible automorphism is $\pm 1$; but these automorphisms do not change the parameter $b$.  Thus $\varsigma$ has degree $2$.

For (b), the natural $\tau$-equivariant structure of the generic fibre extends uniquely to the abelian scheme over $T_{\Q(i)}$, together with its polarization, $\calO$-action, and level structure.  Therefore the map $\varsigma$ factors through the stack quotient which is generically an isomorphism.

At each of the six branch points, the stable fiber has a genus $2$ component and two elliptic components.  In the local coordinates for stable reduction, $\tau$ acts trivially on the genus $2$ component and by the hyperelliptic involution on each elliptic component.  Its action on the Jacobian of the special fibre is therefore $(+1,-1,-1)$, which remains nontrivial after rigidification by the scalar subgroup $\{\pm 1\}$.  Thus the map is injective on geometric stabilizers and hence representable.
Finally, since the source is proper and target is separated, the map is proper; it is nonconstant hence quasi-finite, and therefore finite.  Being birational with normal target, it is an isomorphism, which proves (b) and that $\varsigma$ is \'etale, finishing the proof of (a).  
\end{proof}

\begin{remark}
One could extend the above description to arbitrary genus $g \geq 6$, and one would obtain a map from our family of curves $U$ to an associated moduli stack $\mathcal{M}$ of abelian varieties with PEL structure.  However, the dimension of the image $U \to \mathcal{M}_g \hookrightarrow \mathcal{A}_g$ is $d=g/2-1$ by \Cref{lemma:MapToModuliSpace}(b), whereas the dimension of $\mathcal{M}$ is $(g/2-1)(g/4)$; so only for $g=4$ do these dimensions match.  
Somewhat surprisingly, we find that for $g=4$, the moduli space $\mathcal{M}$ lies \emph{inside} the Torelli locus!  For general (even) $g \geq 4$, the expected dimension of the intersection between the Torelli locus and $\mathcal{M}$ is 
\[ (g/2-1)(g/4) + 3g-3 - g(g+1)/2 = (3/2)(2-g/2)(g/2-1). \]
Already for $g=4$, this gives expected dimension $0$, so we have an unlikely intersection.  For $g=6$, the expected dimension is $-3$, but our family has dimension $2$.  We do not have an explanation for this curious phenomenon.
\end{remark}

\newcommand{\etalchar}[1]{$^{#1}$}

\end{document}